%% file: Paper.tex
\documentclass[a4paper,11pt]{article}
\usepackage[utf8]{inputenc}
\usepackage[english]{babel}

\usepackage[square,sort,comma,numbers]{natbib}
\usepackage{mathtools}
\usepackage{amsthm,amssymb,amsmath, thmtools}
\usepackage[top=2.5cm, bottom=2.5cm, left=2.5cm, right=2.5cm]{geometry}
\usepackage{graphicx}
\usepackage{fancybox}
\usepackage{url}
\usepackage{xfrac}
\usepackage{tikz-cd}
\usetikzlibrary{arrows}
\usepackage{authblk}
\usepackage{chngcntr}
\counterwithin{figure}{section}
\graphicspath{{figures/}}
\usepackage{color,soul}
\usepackage{lipsum}

\usepackage{fancyhdr}
\usepackage{tikz}
\usetikzlibrary{calc}
\usetikzlibrary{decorations.pathmorphing}

\usepackage[ colorlinks = true,
linkcolor = blue,
urlcolor  = blue,
citecolor = blue,
anchorcolor = blue]{hyperref}

\usepackage{makeidx}
\makeindex
\usepackage{tkz-tab,amsmath,amssymb}
\usepackage{subfigure}
\usepackage[perpage]{footmisc}
\usepackage{mathrsfs}
\usepackage{setspace}
\usepackage{enumitem}
\usepackage{array, xcolor}
\usepackage{microtype}

\usepackage{graphicx}
\definecolor{lightgray}{gray}{1}
\newcolumntype{L}{>{\raggedleft}p{0.14\textwidth}}
\newcolumntype{R}{p{0.8\textwidth}}

\theoremstyle{definition}
\newtheorem{definition}{Definition}[section]
\newtheorem{theorem}[definition]{Theorem}
\newtheorem{lemma}[definition]{Lemma}
\newtheorem{proposition}[definition]{Proposition}
\newtheorem{corollary}[definition]{Corollary}

\newtheorem{question}[definition]{Question}
\newtheorem{remark}[definition]{Remark}
\theoremstyle{definition}

\numberwithin{equation}{section}
\newcommand{\alphas}{\boldsymbol{\alpha}}
\newcommand{\betas}{\boldsymbol{\beta}}

\newcommand{\Hcal}{\mathcal{H}}
\newcommand{\Mfrak}{\mathfrak{M}}
\newcommand{\Pfrak}{\mathfrak{P}}
\newcommand{\pfrak}{\mathfrak{p}}
\newcommand{\Z}{\mathbb{Z}}
\newcommand{\ra}{\rightarrow}
\newcommand{\Mhat}{\mathfrak{N}}
\newcommand{\Cscr}{\mathscr{C}}

\title{Three-manifolds with boundary and the Andrews-Curtis transformations}
\author{Neda Bagherifard}
\newcommand{\Addresses}{{
  \bigskip
  \footnotesize

  \textit{E-mail address}, Neda Bagherifard: \texttt{neda.bagherifard@gmail.com}
  }}

\date{}

\begin{document}
\maketitle
\begin{abstract}
We investigate an extended version of the stable Andrews-Curtis transformations, referred to as EAC transformations,  and compare it with  a notion of equivalence in a family of $3$-manifolds with boundary,  called the {\emph{simple balanced $3$-manifolds}}. A simple balanced $3$-manifold is a $3$-manifold with boundary, such that every connected component $N$ of it has unique positive and negative boundary components $\partial^+N$ and $\partial^-N$, such that $\pi_1(N)$ is the normalizer of the image of $\pi_1(\partial^\pm N)$ in $\pi_1(N)$. Associated with every simple balanced $3$-manifold $N$  is the EAC equivalence class of a balanced presentation of the trivial group,  denoted by $P_N$,  which remains unchanged as long as $N$ remains in a fixed equivalence class of simple balanced $3$-manifolds. In particular, the isomorphism class of the corresponding group is unchanged. Motivated by the Andrews-Curtis conjecture, we study the equivalence class of a trivial balanced $3$-manifold (obtained as the product of a closed oriented surface with the unit interval). We show that every balanced $3$-manifold in the trivial equivalence class admits a {\emph{simplifier}} to a trivial balanced $3$-manifold. 

\end{abstract}
\tableofcontents

\section{Introduction} 
\subsection{Group presentations and the Andrews-Curtis transformations}

\begin{definition} (c.f. \cite{John})
Let $F=F(X)$ be the free group on a set $X$ and $R$ be a subset of $F$. The group $G=( X|R)$ is defined as the quotient $F/N$ where $N$ is the smallest normal subgroup of $F$ containing $R$. $(X,R)$ is called  a \emph{presentation} of $G$. The elements of $X$ and $R$ are called the \emph{generators} and  the \emph{relators}, respectively. $G$ is called \emph{finitely presented} if   $G\simeq(X|R)$ where both $X$ and $R$ are finite.  A finite presentation $( X|R)$ is called  \emph{balanced} if $|X|=|R|$.
\end{definition}

\begin{sloppypar}
An {\emph{extended Andrews-Curtis transformation}} (or EAC-transformation for short) on a presentation $P=(a_1,\ldots,a_n|b_1,\ldots,b_m)$ of a group $G$  is defined to be one of the following transformations, or its inverse, which gives another presentation of $G$ \citep{WR} (see also \citep{Hog-Met}):
\end{sloppypar}

\begin{itemize}
\setlength\itemsep{-0.3em}
	\item[1.] {\bf{Composition}}: Replace the relator $b_i$ with $b_ib_j$ for some $j\neq i$;
	\item[2.] {\bf{Inversion}}: Replace the relator $b_i$ with $b_i^{-1}$;
	\item[3.] {\bf{Cancellation}}: Replace a relator $b_i=b_i'gg^{-1}b_i''$ with $b_i'b_i''$, where $g$ is one of $a_j$ or its inverse;
	\item[4.] {\bf{Stabilization}}: Add  $a_{n+1}$ as both a generator and a relator;
	\item[5.] {\bf{Replacement}}: Replace $a_ia_j$ or $a_ia_j^{-1}$ for $a_i$ in all the relators for some $j\neq i$.
\end{itemize} 

\begin{definition}
If $P'=( X'|R')$ is obtained from $P=(X| R)$ by a finite sequence of EAC transformations and their inverses (and renaming the generators), $P$ and $P'$ are called {\emph{EAC equivalent}} and we write $P\sim P'$. It then follows that $( X|R)\simeq ( X'|R')$  and  $|X|-|R|=|X'|-|R'|$. The set of EAC-equivalence classes of finite balanced group presentations is denoted by $\Pfrak$. 
\end{definition}

If we exclude replacement, the remaining four transformations are called the {\emph{stable Andrews-Curtis transformations}}. Correspondingly, we may talk about stable Andrews-Curtis equivalence (or SAC-equivalence).  The {\emph{stable Andrews-Curtis conjecture}} \citep{A-C}, which is widely believed not to be true, states that every balanced presentation of the trivial group is SAC-equivalent to the trivial presentation (see \citep{Brown, Bu-Mac, Miller, Shpil} for some potential counterexamples). In other words, the trivial group only admits one balanced SAC-equivalence class (see the survey \citep{Bu-Mac}). This is not true for non-trivial groups. For instance, $P=( a,b,c|ab,bc, ac^{-1})$ is a presentation of  $\Z$, which is not SAC-equivalent to the trivial presentation (i.e. the presentation with one generator $a$ and the trivial relator); since the sum of exponents of $a,b$ and $c$ in every relator remains even as we change $P$ by SAC-transformations, no presentation in the SAC-equivalence class of $P$ can have fewer than $3$ generators. Nevertheless, we have
\begin{align*}
P&\sim \left( a,b,c\ \big|\ (ab^{-1})b, bc, (ab^{-1})c^{-1}\right)
\sim \left( b,c\  \big|\ bc, b^{-1}c^{-1}\right)
\sim \left( b,c\ \big|\ (bc^{-1})c,(cb^{-1})c^{-1}\right)
\sim\left( c\ \big|\ 1\right).
\end{align*}
There is a surjective map $\mathfrak{q}$ from $\mathfrak{P}$ to the set  $\mathfrak{G}$ of isomorphism classes of finitely presented groups. Inspired by the stable Andrews-Curtis conjecture, one may ask about the size of $\mathfrak{q}^{-1}(G)$ for  $G\in\mathfrak{G}$, and weather the $\mathfrak{q}$-fiber of the trivial group $T\in \mathfrak{G}$ is trivial. \\

There is a $2$-complex $\mathcal{K}_P$ associated to each group presentation $P$ with one $0$-cell, one $1$-cell for each generator which is attached along its boundaries to the $0$-cell, and one $2$-cell for each relator which is attached along its boundary to the $1$-cells associated with the letters in the relator. Two presentations $P$ and $P'$ of the same group $G$ are EAC-equivalent iff $\mathcal{K}_P$ $3$-deforms to $\mathcal{K}_{P'}$ \citep{WR}. This gives an equivalent geometric expression for the SAC conjecture, since the EAC-equivalence class of $P$ matches its SAC-equivalence class when $\pi_1(\mathcal{K}_P)=1$ \citep{WR} (see also the topological survey \citep{Hog-Met} for the SAC conjecture). For many groups, including the trefoil group \citep{Dun} and many finite abelian groups \citep{Met}, the corresponding $\mathfrak{q}$-fibers are not trivial. \\

Let $P$ be a balanced presentation of the trivial group. $\mathcal{K}_P$ may be embedded in $\mathbb{R}^5$ and the boundary of a regular neighborhood of $\mathcal{K}_P$ in $\mathbb{R}^5$ is a homotopy $4$-sphere $\Sigma_P$. $\Sigma_P$ is homeomorphic to  $S^4$ \citep{freedman1982} but not necessarily diffeomorphic to the standard $S^4$. If $P$ is SAC equivalent to the trivial presentation of the trivial group, $\Sigma_P$ is diffeomorphic to the standard $S^4$. Conversely, corresponding to each handlebody decomposition of a homotopy $4$-sphere $\Sigma$ with no  $3$-handles, there is a balanced presentation $P_\Sigma$ of the trivial group, where the generators and the relators correspond to the $1$-handles and  the $2$-handles of $\Sigma$ respectively.  Akbulut and Kirby (in \citep{Akb}) construct a handle structure with no $3$-handles for a homotopy $4$-sphere $\Sigma_0$. $\Sigma_0$ and $P_{\Sigma_0}$ were considered potential counterexamples of the the smooth four dimensional Poincar\'{e} conjecture (SPC4 for short) and the SAC conjecture, respectively (see also \citep{Akb2}).  However $\Sigma_0$ was later proved to be diffeomorphic to the standard $S^4$ \citep{Gom}.  In fact, many of potential counterexamples to SPC4 are shown to be diffeomorphic to the standard $S^4$ in recent years. Nevertheless, many experts still believe that the SPC4 is incorrect and \citep{Freedman} proposes a method to disprove both the SAC conjecture and the SPC4.

\subsection{Heegaard diagrams and \texorpdfstring{$3$}{Lg}-manifolds}\label{HD}
Set $I=[-1,+1]$,  $\partial^+ I=\{+1\}$, $\partial^-I=\{-1\}$ and $I^\circ=(-1,+1)$ throughout the paper. Let  $N$ be a compact oriented $3$-manifold with boundary $\partial N=-\partial^-\amalg \partial^+$, where the orientation of $\partial^+$ (resp. $\partial^-$) matches with (resp. is the opposite of) the orientation inherited as the boundary of $N$. If each component of $N$  intersects precisely one component of  $\partial^+$ and one component of $\partial^-$ with  the same genus, $N$ is called a {\emph{balanced}} $3$-manifold. The space of all balanced  $3$-manifolds, up to homeomorphism, is denoted by $\mathfrak{M}$. Choose $N\in\Mfrak$ and consider a Heegaard diagram  
\[\Hcal=(\Sigma,\boldsymbol{\alpha}=\{\alpha_1,\dots,\alpha_k\},\boldsymbol{\beta}=\{\beta_1,\dots,\beta_k\})\]
for $N$, where $\Sigma$ is a  Heegaard surface in $N$ and $\alphas$   (respectively, $\betas$) is a collection of $k$ disjoint simple closed curves on $\Sigma$. $N$ is obtaining by attaching $2$-handles to $\Sigma\times I$ along $\alphas\times\{-1\}$ and $\betas\times\{1\}$. Since $N$ is balanced, each connected component of $\Sigma$ contains the same number of curves from $\alphas$ and $\betas$.  Heegaard diagrams are changed to one another by a sequence of Heegaard moves (isotopies, handle-slides and stabilizations/destabilizations). \\

Associated with every Heegaard diagram $\mathcal{H}$ as above, we may introduce a balanced group presentation $P_\Hcal$ as follows. First, choose an orientation on the curves in $\alphas$ and $\betas$, and a {\emph{start point}} $p_i$ on each $\beta_i$ away from the intersections of $\beta_i$ with $\alphas$, for $i=1,\ldots,k$. Associated with each $\beta_i$, we define a relator $b_i$ in the free group $F(a_1,\ldots,a_k)$. As $\beta_i$ is traversed following its orientation and starting from $p_i$, we face intersections with the curves $\alpha_{j_1},\ldots,\alpha_{j_r}$, where the intersection number is $\epsilon_1,\ldots,\epsilon_r\in\{-1,+1\}$, respectively. We then set $b_i=a_{j_1}^{\epsilon_1}\ldots a_{j_r}^{\epsilon_r}$, and define the group presentation $P_\Hcal$ by  $\left( a_1,\ldots,a_k|b_1,\ldots,b_k\right)$.\\

The choices of the orientation and the start points $p_1,\ldots,p_k$ do not change the EAC-equivalence class of the presentation $P_\Hcal$.  Moreover, if a Heegaard diagram $\Hcal'$ is obtained from $\Hcal$ by a Heegaard move, it follows that $P_{\Hcal'}$ is obtained from $P_\Hcal$ by   EAC-transformations, for compatible choices of the start points and the orientations. In fact, handle-slides in Heegaard diagrams among the $\beta$ curves correspond to compositions, {together with the inverse of some cancellations}, in group presentations, while handle-slides among the $\alpha$ curves correspond to replacements, together with the inverse of some cancellations. Isotopies in Heegaard diagrams correspond to cancellations in group presentations (and their inverses) and stabilizations in the Heegaard diagrams correspond to stabilizations in group presentations. 
In particular, it follows that the EAC-equivalence class of the group presentation $P_\Hcal$ only depends on the $3$-manifold $N$, giving a map $\pfrak_\alpha:\Mfrak\ra \Pfrak$. We denote the aforementioned equivalence class by 
\[\pfrak_\alpha(N)=P_N=[P_\Hcal]=\left[\left( a_1,\ldots,a_k|b_1,\ldots,b_k\right)\right].\]
If $N\in\mathfrak{M}$ is connected, $\mathfrak{q}(\pfrak_\alpha(N))$ is the quotient of $\pi_1(N)$ by the normal subgroup generated by $\iota_*^-\pi_1(\partial^-)$, where $\iota^\pm:\partial^\pm\rightarrow N$ are the inclusion maps.
If we change the role of $\alphas$ and $\betas$, we obtain a second map $\pfrak_\beta:\Mfrak\ra \Pfrak$, and $\mathfrak{q}(\pfrak_\beta(N))$ is the quotient of $\pi_1(N)$ by the normal subgroup generated by $\iota_*^+(\pi_1(\partial^+))$, provided that $N$ is connected. A connected balanced $3$-manifold $N\in\mathfrak{M}$ is called \emph{simple} if the quotient of $\pi_1(N)$ by the normal subgroup generated by either of $\iota_*^\pm(\pi_1(\partial^\pm))$ is trivial. An arbitrary $N\in\Mfrak$ is called simple if every connected component of $N$ is simple. For a simple balanced $3$-manifold $N\in\Mfrak$, $(\mathfrak{q}\circ\pfrak_\alpha)(N)$ and $(\mathfrak{q}\circ\pfrak_\beta)(N)$ are both trivial.
  
\subsection{The notion of equivalence}
 Let $N\in\Mfrak$ and  $\partial N=-\partial^-\amalg\partial^+$. We say that $N'$ is obtained from $N$ by adding a trivial connected component if $N'$ is the disjoint union of $N$ with a component $S\times I$ for a closed oriented surface $S$. A particular special case is when $S$ is a union of $2$-spheres, when we write $N\doteq N'$. Given a Heegaard diagram $\Hcal=(\Sigma,\alphas,\betas)$ for $N$, the Heegaard diagram $\Hcal'=(\Sigma'=\Sigma\amalg S,\alphas,\betas)$ represents $N'$. It follows that $\pfrak_\alpha(N)=\pfrak_\alpha(N')$ and $\pfrak_\beta(N)=\pfrak_\beta(N')$ if $N\doteq N'$ (or if $N$ and $N'$ differ in a number of trivial components).  By a solid cylinder $D\times I$ in $N$, we  always mean a homeomorphic image of the standard solid cylinder with the property  that $D\times \partial^\pm I\subset \partial^\pm$ and $D\times I^\circ\subset N^\circ$, where $D$ is the standard $2$-disk. Consider a pair of disjoint solid cylinders $D_i\times I$ in $N$, for $i=1,2$. Remove $D_i^\circ\times I$ from $N$ and identify $\partial D_1\times I$ with $\partial D_2\times I$ using the obvious orientation reversing map.  The resulting manifold $N'$ has two collections of boundary components, denoted by $\partial^{+}N'$ and $\partial^{-}N'$, where $\partial^\pm N'$ is  obtained  from $\partial^{\pm}$  by adding a $1$-handle. Moreover, $C=\partial D_1\times I$, which is identified with $\partial D_2\times I$ in $N'$, is a cylinder in $N'$.

\begin{definition}
In the above situation, we say $N'$ is obtained from $N$  by {\emph{adding the cylinder}} $C$, $N$ is obtained from $N'$  by {\emph{removing the cylinder}} $C$, and write $N'\overset{C}{\ra} N$. We say $N$ {\emph{simplifies}} to $N'$, and write $N\leadsto N'$, if there is a sequence $N'=N_0,N_1,\ldots,N_n=N$ in $\Mfrak$ and the cylinders $C_i\subset N_i$  such that $N_i\xrightarrow{C_i}N_{i-1}$ for $i=1,\ldots,n$ and $N'\doteq N''$, where $N''$ is connected and {\emph{irreducible}} (which means that  all the boundary components of $N''$ are spheres and the closed $3$-manifold $\overline{N}''$ obtained from $N''$ by attaching $3$-disks to its boundary spheres is irreducible).  
\end{definition}
	
If $N'\overset{C}{\ra} N$, the balanced Heegaard diagrams $\Hcal=(\Sigma,\alphas,\betas)$ and $\Hcal'=(\Sigma',\alphas,\betas)$ for $N$ and $N'$ may be chosen so that $\Sigma'$ is obtained from $\Sigma$ by adding a $1$-handle, away from the curves in $\alphas$ and $\betas$. It follows, in particular, that  $\pfrak_\alpha(N)=\pfrak_\alpha(N')$ and $\pfrak_\beta(N)=\pfrak_\beta(N')$. 

\begin{definition}
 We say $N,N'\in\Mfrak$ are {\emph{equivalent}}, and write $N\sim N'$,  if $N'$ is obtained from $N$ by a finite sequence of adding/removing {trivial connected components} and adding/removing cylinders.   $\pfrak_\alpha,\pfrak_\beta:\Mfrak\ra \Pfrak$ induce  maps from  $\Mhat=\Mfrak/\sim$ to $\Pfrak$, which are also denoted by $\pfrak_\alpha,\pfrak_\beta:\Mhat\ra \Pfrak$, by slight abuse of notation.	
\end{definition}

If a balanced (finite) group presentation $P=\left( X|R\right)$ representing a class $P\in\Pfrak$ is {\emph{realized}} by a Heegaard diagram $\Hcal=(\Sigma,\alphas,\betas)$ (after fixing  the start points and appropriate orientations for the curves in $\alphas$ and $\betas$), every presentation obtained by composition, inversion, stabilization or replacement is also realized by a Heegaard diagram. For instance, if $P'$ is obtained from $P$ by a composition which replaces $b_i$ with $b_ib_j$, then $P'=P_{\Hcal'}$ where $\Hcal'=(\Sigma',\alphas,\betas')$ is obtained as follows. The surface $\Sigma'$ is obtained from $\Sigma$ by attaching a $1$-handle with feet near $p_i$ and $p_j$. The set $\betas'$ of disjoint simple closed curves is obtained from $\betas$ by replacing $\beta_i'$ for $\beta_i$, where $\beta_i'$ is obtained by stretching a small part of $\beta_i$ in a neighborhood of $p_i$ by an isotopy (a finger move) over the attached $1$-handle so that it arrives near $p_j$, and doing a handle-slide over $\beta_j$. Similarly, if $P'$ is obtained from $P$ by a replacement, a Heegaard diagram $\Hcal'$ corresponding to $P'$ may be constructed from $\Hcal$ by adding a $1$-handle to $\Sigma$ and doing a handle-slide among curves in $\alphas$.  This is not necessarily true, however,  when $P'$ is obtained from $P$ by a cancellation. 
The above partial correspondence suggests a further study of the notion of equivalence. In fact, the main purpose of this paper is taking a number of steps in this direction. With the above definitions and concepts in place, the main results of this paper are the following two theorems.

 \begin{theorem}\label{thm1}
 Suppose that a $3$-manifold $N\in\Mfrak$ simplifies to $N'\in\Mfrak$. If $C$ is a cylinder without punctures in $N$, then there is a simplification
\[N=N_n\xrightarrow{C_n=C}N_{n-1}\xrightarrow{C_{n-1}}\cdots
\xrightarrow{C_2}N_1\xrightarrow{C_1}N_0,\quad\text{with}\ 
N'\doteq N_0.\] 
  \end{theorem}

\begin{theorem}\label{thm2}
	If $N\in\Mfrak$  is equivalent to $S^2\times I$, it may be simplified to  some $N_0\doteq S^2\times I$.
\end{theorem}

If $N\in\Mfrak$ is equivalent to $S^2\times I$, it follows that both $\pfrak_\alpha(N)$ and $\pfrak_\beta(N)$ are the same as the EAC-equivalence class of the trivial group presentation and $N$ is a simple balanced $3$-manifold. It is then interesting to ask the following question.
\begin{question}
	Are there simple balanced $3$-manifolds which are not equivalent to $S^2\times I$?
\end{question}
For a potential counterexample $N\in\mathfrak{M}$, to show that $N$ is not equivalent to $S^2\times I$, it suffices (by Theorem \ref{thm2}) to show that $N$ does not include a non-trivial cylinder. This may be investigated using Heegaard Floer theory, and will be pursued in an upcoming sequel.

The notion and equivalence of {\emph{simplifiers}} are studied in Section~\ref{s-equiv}.  Two types of intersections of a cylinder $C$ with a simplifier $\Cscr$ are distinguished, and are investigated separately in Sections~\ref{s-type-1} and~\ref{s-type-2}, respectively. This study gives a proof of Theorem~\ref{thm1}, while Theorem~\ref{thm2} follows as a corollary.

\textbf{Acknowledgements.} It is a pleasure to thank my advisor, Eaman Eftekhary, for his continuous support, all substantial helps, discussions and suggestions during the course of this work, and reading the drafts carefully. I would like to thank the Institute for Research in Fundamental Sciences  (IPM) for providing a supportive research environment.

\section{Simplifiers for balanced \texorpdfstring{$3$}{Lg}-manifolds}\label{s-equiv}
\subsection{Balanced \texorpdfstring{$3$}{Lg}-manifolds, punctured cylinders and simplifiers}
  Let $N\in\Mfrak$ and  $\partial N=-\partial^-\amalg\partial^+$.  Let $C\subset N$ be a surface of genus zero with one boundary component $\partial^+C$ on $\partial^+$, one boundary component  $\partial^-C$ on $\partial^-$, and $C\setminus\partial^\pm C\subset N^\circ$. We call $C$ a \emph{punctured cylinder} and each curve in $\partial C\setminus \partial^\pm C$ is called  \emph{a puncture of $C$}.	If $C$ has no punctures, it is called a {\emph{cylinder}} and is usually identified as a homeomorphic image of  $S^1\times I$  with $S^1\times\partial^\pm I\subset\partial^\pm$ and $S^1\times I^\circ\subset N^\circ$. If $N\xrightarrow{C}N'$, associated with $C\subset N$ we obtain a pair of solid cylinders in $N'$, which are denoted by $D_1^C\times I$ and $D_2^C\times I$. Given a simple closed curve $l$ in the interior of a punctured cylinder $C$,  the complement $C\setminus\{l\}$ has two components. If one component contains $\partial^+C\amalg\partial^-C$,  $l$ is  the boundary of a \emph{punctured disk} $D_{l,C}$ on $C$. Otherwise, if $\partial^+C$ and  $\partial^-C$ are  in different components, $l$ is called an \emph{essential curve} on $C$.\\
	 
  Let $\Cscr=(C_1,\dots,C_n)$ be a sequence of punctured cylinders in $N$, with $\partial^\pm C_i=l_i^\pm\subset\partial^\pm$  for $i=1,\ldots,n$, such that  $l_1^\pm,\dots,l_n^\pm$ are disjoint simple closed curves, and the interiors of $C_i$ are disjoint. Set $\partial^\pm\Cscr=l_1^\pm\cup\cdots\cup l_n^\pm \subset\partial^\pm N$ and $\Cscr^i=(C_{i+1},\ldots,C_n)$. Further, assume that
  \begin{enumerate}[label=\textbf{S.\arabic*},ref=S.\arabic*]
  \setlength\itemsep{-0.3em}
   \item\label{2} The complement of $\partial^\pm \Cscr$ in $\partial^\pm$  is a collection of punctured spheres;
   \item\label{6}  The punctures of $\cup_{k\leq i}C_k$   are disjoint essential curves on $\Cscr^i$,  called the {\emph{generating curves}}. 
  \end{enumerate}
  Let $N\xrightarrow{C_n} N_{n-1}=N[C_n]$ and $\partial N_{n-1}=\partial_{n-1}^+\amalg\partial_{n-1}^-$. Then $\{l_i^\pm\}_{i=1}^{n-1}$ is also a collection of curves on $\partial_{n-1}^\pm$ and   $\partial_{n-1}^\pm\setminus\{l_i^\pm\}_{i=1}^{n-1}$ is a union of punctured spheres. Those punctures of $C_i$ which are essential curves on $C_n$,  are the boundaries of disjoint disks in $N_{n-1}\setminus\cup_{i=1}^{n-1}C_i$. Let $C_i[C_n]$ denote the punctured cylinder in $N_{n-1}$, obtained by attaching these disks to $C_i$. We usually abuse the notation and denote $C_i[C_n]$ by $C_i\subset N_{n-1}$. The sequence $\Cscr_{n-1}=(C_1[C_n],\ldots,C_{n-1}[C_n])$ of punctured cylinders in $N_{n-1}$ satisfies  \ref{2} and \ref{6}. Continue by removing $C_i$ from $N_i$ to arrive at $N_{i-1}=N[\Cscr^{i-1}]$ with $\partial N_{i-1}=-\partial_{i-1}^-\amalg\partial_{i-1}^+$, for $i=n,\dots,1$.  Moreover, we obtain a sequence $\Cscr_i=(C_1[\Cscr^i],\ldots,C_{i}[\Cscr^i])$ of punctured cylinders in $N_i$, which is usually denoted by $(C_1,\ldots,C_i)$ (by abuse of notation). The punctured cylinder $C_i=C_i[\Cscr^j]\subset N_j$ is obtained from $C_i=C_i[\Cscr^k]\subset N_k$ by attaching disks to some of the punctures, if $j<k$. The notation (and abuse of notation) set here will be used frequently through this paper. We  say $N_0=N[\Cscr]$ is obtained from $N$ by removing $\mathscr{C}=\left(C_i\right)_{i=1}^n$ and write $N\overset{\mathscr{C}}{\leadsto} N_0$. Finally, we assume:  
  \begin{enumerate}[label=\textbf{S.\arabic*},ref=S.\arabic*]
   \setcounter{enumi}{2}
   \item\label{7} $N_0$ is irreducible and  all the components of ${N_0}$, except possibly one, are $S^2\times I$s.
  \end{enumerate}
\begin{definition}\label{d-simplifying-structure}
We   call $\mathscr{C}=\left(C_i \right)_{i=1}^n$ a \emph{simplifier} for $N$ which {\emph{simplifies}} it to $N_0$, if \ref{2}-\ref{7} are satisfied.   The simplifiers $N\overset{\Cscr}{\leadsto}N_0$ and $N\overset{\Cscr'}{\leadsto}N_0'$  are called {\emph{equivalent}} if $N_0\doteq N_0'$.
  \end{definition}
	
\begin{remark}\label{r0}
Let $N$ simplify to $N_0$ using a simplifier $\mathscr{C}=(C_i)_{i=1}^n$. By attaching disks to $\partial^\pm$ along $l_i^\pm$, and capping the resulting sphere boundary components with $3$-disks, we obtain two solid tori  of genus $g$ which are denoted by $B_\mathscr{C}^\pm$. Then 
\[\overline{N}_\mathscr{C}=N\cup_{\partial^+} B_\mathscr{C}^+\cup_{\partial^-} B_\mathscr{C}^-\] 
is a closed 3-manifold. Since we attached disks to $l_n^\pm$, $C_n$ corresponds to a sphere $S_n^2$ in $\overline{N}_\mathscr{C}$. Further, the punctures on all the punctured cylinders $C_i$, $1\leq i<n $ which are essential curves on $C_n$ become the boundaries of disjoint disks in $\overline{N}_\mathscr{C}$. Repeating the same argument for $C_i$, with $i=n,n-1,\ldots,1$,  we conclude that each one of these cylinders  corresponds to a sphere $S_i^2$ in $\overline{N}_\mathscr{C}$. Removing $C_i$ corresponds to decomposing $\overline{N}_\mathscr{C}$ along the sphere $S_i^2$. Thus, \ref{7} is equivalent to the condition that $\overline{N}_\mathscr{C}$ is the union of some components homeomorphic to $S^3$  with a connected irreducible 3-manifold.
\end{remark}

\begin{remark}\label{r1}
Suppose that $N$ has a simplifier $\mathscr{C}=(C)$, and that all the connected components of $\partial^+$ and $\partial^-$ are spheres. Then  $D_1^C\times I$ and $D_2^C\times I$ are (necessarily) in two different components  of $N_0$. Since at least one of these components  is $S^2\times I$,  we find $N\doteq N_0$.
  \end{remark}
	
  \begin{remark}\label{r2} Let $\Cscr=(C_1,\ldots,C_n)$ be a simplifier for $N\in\Mfrak$ as above. Then
   \begin{enumerate}[ref=\ref{r2}-\arabic*]
   \setlength\itemsep{-0.3em}	
    \item\label{r2.1} $\partial^+\setminus\partial^+\Cscr$ and $\partial^-\setminus\partial^-\Cscr$ have the same number of components.
    \item\label{r2.2} Let $N^1,\ldots,N^k$ be the connected components of $N$ and $N_0^i$ be obtained from $N^i$ by removing the simplifier $\mathscr{C}\cap N^i$, for $i=1,\dots,k$. Then $N_0=N_0^1\amalg\dots\amalg N_0^k$.
    \item\label{r2.3} Suppose that  $D_2^{C_n}\times I$ is not in the component $N'$ of $N_{n-1}$ which contains $D_1^{C_n}\times I$. If $\partial N'$ is a pair of spheres, then $N=\overline{N}'\#(N_{n-1} \setminus N')$.
    \item\label{r2.4} 
    Suppose that $N=N'\#M\overset{\Cscr}{\leadsto} N_0$, where $M$ is a closed $3$-manifold and $\Cscr\cap M=\emptyset$. Then $\Cscr$ may also be regarded as a simplifier for $N'$. If $N'\overset{\Cscr}{\leadsto}N'_0$, we have and $N_0=N'_0\#M$.
\item\label{r2.5}  If $C_i$ is a cylinder in $\mathscr{C}$ without punctures, set $C'_j=C_j$ for $j<i$, $C'_j=C_{j+1}$ for $i\leq j\leq n-1$ and $C'_{n}=C_i$. Then $\mathscr{C}'=(C'_1,\ldots,C'_n)$ is  a simplifier for $N$ which is equivalent to $\mathscr{C}$.
   \end{enumerate}
\end{remark}

  \begin{proposition}\label{p5}
   Let $\mathscr{C}=\left(C_i\right)_{i=1}^n$ be a simplifier for a connected manifold $N\in\Mfrak$.  Then:
   \begin{itemize}[leftmargin=.8in]\setlength\itemsep{-0.3em}
    \item[$A_\mathscr{C}^i(1).$]If an essential curve on $C_i$ bounds a disk in $N$, then $l_i^\pm$ bounds a disk in $\partial_i^\pm$.
	\item[$A_\mathscr{C}(2).$]Each 2-sphere $S$ in $N$ is separating, i.e. $N\setminus S$ has $2$ connected components.
	\item[$A_\mathscr{C}(3).$]If a sphere $S\subset N^\circ$  separates $\partial^+$ from $\partial^-$, $N$ has sphere boundary components. 
   \end{itemize}
  \end{proposition}
	
  \begin{proof}
 We use an induction on $n=|\mathscr{C}|$. For $n=0$,  all claims follow since $N$ is irreducible. First, we prove $A_\mathscr{C}^n(1)$.  Let $l$ be an essential curve on $C_n$ such that $l=\partial D$ for a disk $D\subset N$. We may assume that $D$  cuts $C_n$ transversely and  that the curves in $D^\circ \cap C_n$ are not essential  on $C_n$. Choose $l'\in C_n\cap D$ such that   $D_{l',C_n}\cap D=l'$. $l'$ is also the boundary of a disk $D'\subset D$. Remove $D'$ from $D$ and replace it with $D_{l',C_n}$ to obtain a disk bounded by $l$, which has fewer intersections with $C_n$.  If we continue this process, we obtain a disk $D_l$ with $\partial D_l=l$ and $D_l^\circ\cap C_n=\emptyset$.\\
	
Let $D_i\times I=D_i^{C_n}\times I\subset N_{n-1}$. Then $l\subset C_n$ corresponds to the curves $l_i\subset\partial D_i\times I$, $i=1,2$. Since $D_l^\circ\cap C_n=\emptyset$, there is a disk in $N_{n-1}$ which corresponds to $D_l$, and may be denoted by $D_l$ as well, by slight abuse of notation. In such situations, we simply say that {\emph{a copy of $D_l$ exists in $N_{n-1}$}}. The boundary of this new disk is one of $l_1$ and $l_2$, say $l_1$.  $l_1$ is also the boundary of a disk $D'_l\subset D_1\times I$, and we may assume that $D_l'=D_1\times\{1/2\}$. Thus, $S=D_l\cup D'_l$ is a sphere in $N_{n-1}$, which is separating by the induction hypothesis. Since $D_1\times\{0\}$ and $D_1\times\{1\}$ are on two different sides of $S$, it follows that $S$ separates the negative boundary from the positive boundary. 
 Let us first assume that $N_{n-1}$ is connected. Since $S$  separates $\partial^+N_{n-1}$ from $\partial^-N_{n-1}$, $(\partial D_2\times I)\cap S\neq\emptyset$. Therefore, $C_n\cap D_l^\circ\neq\emptyset$, which is impossible. It thus follows that $N_{n-1}$ is disconnected. Let $D_1\times I$ be in the component $N'$ of $N_{n-1}$. Since $S$ separates $\partial^+N'$ and $\partial^-N'$, by the induction hypothesis $\partial^+N'$ is a sphere. Thus, $\partial D_1\times\{1\}$ bounds a disk in $S^+$, disjoint from $D_1^\circ\times\{1\}$. As a result, $l_n^+$ bounds a disk in $\partial^+$. Similarly, $l_n^-$ bounds a disk in $\partial^-$. This completes the proof of $A_\mathscr{C}^n(1)$.\\
		
   Next, we prove $A_\mathscr{C}^i(1)$.  Let $l$ be an essential curve on $C_i$, $i<n$, which is  the boundary of a disk $D$ in $N$ which cuts the cylinders transversely. We can remove the components in $D\cap C_n$ which are not essential on $C_n$ one by one (following the procedure introduced in the proof of $A_\mathscr{C}^n(1)$) and assume that all the curves in $D\cap C_n$ are essential  on $C_n$. If $D\cap C_n=\emptyset$, there is a copy of $D$ in $N_{n-1}$, and by the induction hypothesis $l_i^\pm$ bounds a disk in $\partial^\pm_i$ (which completes the proof of $A_\mathscr{C}^i(1)$). Otherwise $D\cap C_n\neq\emptyset$. Since all the curves in $D\cap C_n$ are the boundaries of disks in $D$,  $A_\mathscr{C}^n(1)$ implies that $l_n^\pm=\partial D^\pm$ where $D^\pm\subset\partial^\pm$ is a disk. Therefore,  $D_1\times I$ and $D_2\times I$ are in two different components  of $N_{n-1}$. The positive boundary of one of these components, say $N'$ which contains $D_1\times I$,  is a sphere. This implies that $\partial^-N'$ is a sphere as well  and we have $N=(N_{n-1}\setminus N')\#\overline{N}'$ (by Remark \ref{r2}). Moreover, we may assume that the copies of $D^\pm$ in $N$ are in $\partial^\pm N'$ and therefore $\partial (D_1\times\{0\})=\partial D^-$.\\
   
Then $S=C_n\cup D^+\cup D^-$ is a sphere in $N$, which is included in $N^\circ$ after a  perturbation. Cutting $N$ along $S$, and attaching two $3$-disks to the resulting boundary spheres, we obtain $N_{n-1}\setminus N'$ and $\overline{N}'$. The intersection $D\cap S$ consists of some closed curves. Let $l'\in D\cap S$ and choose the disk $D'\subset S$ such that $l'=\partial D'$ and $(D')^\circ\cap D=\emptyset$. We then have $l'=\partial D''$ for a disk $D''\subset D$. If we replace $D''\subset D$ with $D'$, the resulting disk has fewer intersections with $S$. By repeating this process, we obtain a disk $D_l$ with $\partial D_l=l$ and $D_l\cap S=\emptyset$. If $l_i^+$ is not the boundary of a disk in $\partial^+_i$, it follows that $C_i\subset\mathscr{C}\cap(N_{n-1}\setminus N')$. Since $l$ is an essential curve on $C_i$, we have $D_l\subset N_{n-1}\setminus N'$, which is impossible by the induction hypothesis. Therefore, $l_i^+$ bounds a disk in $\partial_i^+$, and similarly, $l_i^-$ bounds a disk in $\partial_i^-$. This completes the proof of $A_\mathscr{C}^i(1)$.\\

   Let $A_\mathscr{C}^m(2)$ be the claim $A_\mathscr{C}(2)$ for spheres $S$ in $N$ with $|S\cap C_n|=m$. We prove $A_\mathscr{C}^m(2)$ by (a second) induction on $m$. If $m=0$, there is a copy of $S$ in $N_{n-1}$, denoted again by $S$. By the induction hypothesis, $S$ is separating in $N_{n-1}$. If $N_{n-1}$ is disconnected,  $N\setminus S$ clearly has two connected components. Suppose otherwise that $N_{n-1}$ is connected. Since $C_n\cap S=\emptyset$, it follows that $(D_i\times I)\cap S=\emptyset$. Each $D_i\times I$, $i=1,2$, connects $\partial^+N_{n-1}$ to $\partial^-N_{n-1}$. Therefore, the two boundary components of $N_{n-1}$ are in the same component of $N_{n-1}\setminus S$. Therefore, $D_1\times I$ and $D_2\times I$ are in the same component of $N_{n-1}\setminus S$. If follows that $N\setminus S$ has two connected components.\\
			
Having settled the case $m=0$, suppose $m>0$.  If there is a curve $l$ in $S\cap C_n$,  such that $l=\partial D$ for a disk $D\subset C_n$, we may choose $l$ such that $D^\circ\cap S=\emptyset$. If $S\setminus l= D_1\cup D_2$, the two spheres $S_i=D_i\cup D$, $i=1,2$, are such that $|S_i\cap C_n|<m$, $i=1,2$. By the induction hypothesis, $N\setminus S_i$ has two connected components. Let $N\setminus S_i=A_i\cup B_i$, $i=1,2$. Let $S_2\subset A_1$ and $B_2\subset A_1$. Attaching $S_1$ and $S_2$ along $D$, we obtain $S$, which is separating since
\[N\setminus S=(B_1\cup B_2\cup D)\dot\cup(A_1\setminus\overline{B}_2).\]
		
Therefore, we may assume that each curve in $S\cap C_n$ is  essential  on $C_n$.  Since every curve in $S\cap C_n$ bounds a disk in $S$,  $l_n^\pm$ bounds a disk $D^\pm\subset\partial^\pm$ and $S'=C_n\cup D^+\cup D^-$ gives a sphere in $N^\circ$, after a slight perturbation.  Cut $N$ along $S'$, and  attach $3$-disks to the  resulting boundary spheres. We may choose $D^-$ so that the outcome is the union of $\overline{N}'$ and $N_{n-1}\setminus N'$, where $N'$ is a component of $N_{n-1}$ with spherical boundary components. Note that $m=|S\cap S'|=|S\cap C_n|>0$. If $m=1$, $l\in S\cap S'$ bounds a disk $D\subset S'$. $l$ cuts $S$ to two disks $D'_1$ and $D'_2$. Then $S_i=D'_i\cup D$, $i=1,2$ give two spheres, one in $\overline{N}'$  and one in $N_{n-1}\setminus N'$. By the induction hypothesis, $S_1$ and $S_2$ are separating in $\overline{N}'$  and  $N_{n-1}\setminus N'$.  As we argued in the previous paragraph, one deduces from here that $S$ is separating in $N$. This gives $A_\mathscr{C}^1(2)$. If $m>1$, choose $l\in S\cap S'$ such that it bounds a disk $D\subset S'$ with $D^\circ \cap S=\emptyset$. Let $S\setminus l=D_1'\cup D_2'$ and set  $S_i=D_i'\cup D$, $i=1,2$. Then $S_1$ and $S_2$ are two spheres with $|S_i\cap C_n|<m$, $i=1,2$. By the induction hypothesis, $S_1$ and $S_2$ are separating and as we argued in the previous paragraph, it is deduced from that $S$ is also separating.\\

For $A_\mathscr{C}(3)$, let $S$ separate $\partial^+$ and $\partial^-$. $C_n$ intersects $S$ in an essential curve on $C_n$. Let
\[A_S=\{\gamma\in C_n\cap S\ |\ \gamma=\partial D,\ \ D\subset C_n\ \text{is\ a\ disk}\}\quad\text{and}\quad B_S=(S\cap C_n)\setminus A_S.\]
Choose $l\in A_S$ so that it bounds a disk $D$ on $C_n$, so that $D^\circ\cap S=\emptyset$. Let $l=\partial D'$ for a disk $D'\subset S$, such that $\partial^+$ and $\partial^-$ are on one side of the sphere $D\cup D'$. Let $S_1$ be obtained from $S$ by replacing $D'$ with $D$. It is clear that $S_1$ separates $\partial^+$ and $\partial^-$. We have $|A_{S_1}|<|A_S|$ and $B_{S_1}\neq\emptyset$. By repeating this process, we obtain a sphere $\overline{S}=S_k$ which separates $\partial^+$ and $\partial^-$, such that $|A_{\overline{S}}|=\emptyset$ and $B_{\overline{S}}\neq\emptyset$. By $A_\Cscr^n(1)$, the latter property implies  that there are  disks $D^\pm\subset\partial^\pm$ with $\partial D^\pm=l_n^\pm$. Then $S'=C_n\cup D^+\cup D^-$ is a sphere in $N^\circ$, after a slight perturbation. Cut $N$ along $S'$ and attach two $3$-disks to the resulting boundary spheres. We may choose $D^-$ so that the outcome is the union of  $\overline{N}'$ and $N_{n-1}\setminus N'$, where $N'$ is a component of $N_{n-1}$ with spherical boundary components. We prove $\partial^\pm$ are spheres by induction on $h=|B_{\overline{S}}|=|\overline{S}\cap S'|>0$. If $h=1$, $\overline{S}\cap S'=\{l\}$ and $l$ is the common boundary of the disks $D\subset S$ and $D'\subset S'$, whereas  $D\cap\overline{N}'\neq\emptyset$. If $S''$ is obtained from $\overline{S}$ by replacing $D'$ with $D$, $S''\cap S'=\emptyset$ and $S''$ is a sphere in $N_{n-1}\setminus N'$ which separates $\partial^+(N_{n-1}\setminus N')$ and $\partial^-(N_{n-1}\setminus N')$. Therefore, by the induction hypothesis, $\partial^\pm(N_{n-1}\setminus N')$  are spheres. This implies that $\partial^\pm$ are spheres as well.\\

Finally when $h>1$, choose $l\in \overline{S}\cap S'$ so that it bounds a disk $D\subset S'$ with $D^\circ\cap \overline{S}=\emptyset$. We have $l=\partial D'$ where  $D'\subset \overline{S}$ is a disk, and $D'\cap\overline{N}'\neq\emptyset$. If $S''$ is obtained from $\overline{S}$ by replacing $D'$ with $D$,  $|S''\cap S'|<|\overline{S}\cap S'|$ and $S''$ separates $\partial^+$ and $\partial^-$. Therefore, by the induction hypothesis, $\partial^\pm$ are spheres. This completes the proof of the proposition. 
\end{proof}
	
\begin{definition}
Let $N$ be a manifold with a simplifier $\mathscr{C}=\left(C_i\right)_{i=1}^n$. We call $\mathscr{C}$ a {\emph{reduced}} simplifier if no $l_i^\pm=\partial^\pm C_i$ is  the boundary of a disk in $\partial_i^+=\partial^\pm N[\Cscr^i]$, for  $i=1,\ldots,n$.
\end{definition}
\subsection{Adding a cylinder to a manifold with a simplifier}
\begin{remark}\label{r-sliding}
Fix a simplifier $\mathscr{C}=\left(C_i\right)_{i=1}^n$ for $N\in\Mfrak$ and the indices $0<i_1<i_2<i_3\leq n$. For $j=1,2$, let $l_j$   be a puncture on $C_{i_j}$, which is also generating curves on $C_{i_{j+1}}$. Let $l_2$ be adjacent to $l_1$ in the sense that there is a curve on $C_{i_2}$, with one boundary component on $l_1$ and one on $l_2$, such that this curve is disjoint from the generating curves on $C_{i_2}$ (Figure \ref{sliding}).\\
		\begin{figure}
			\def\svgwidth{15cm}
			\begin{center}
				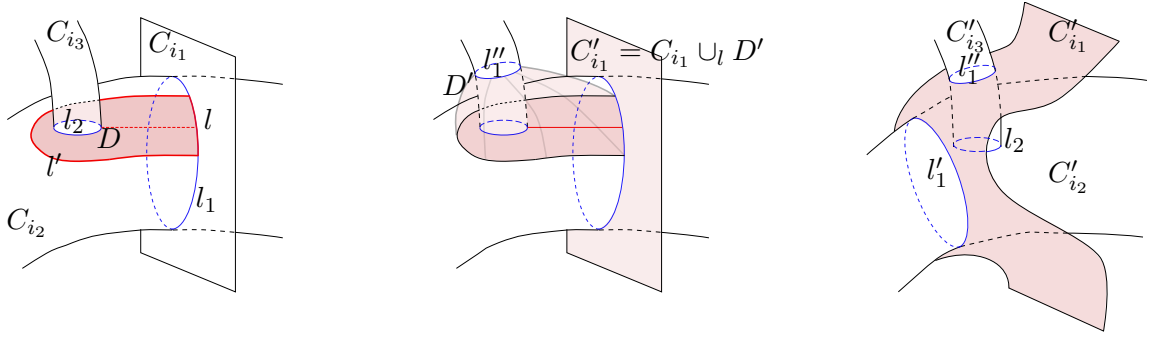
				\caption{The blue curves are the punctures, or  equivalently, the generating curves (left). $C_{i_1}'$ is obtained by attaching $D'$ to $C_{i_1}$ along the arc $l$ (middle), and a small perturbation (left).}
				\label{sliding}
			\end{center}
		\end{figure}
		
Choose a punctured disk $D\subset C_{i_2}$ such that $l_2$ is its only puncture and $\partial D=l\cup l'$ where $l\subset l_1$ and $l'\subset C_{i_2}$ are arcs, so that $l'$ is disjoint from the generating curves, as illustrated in Figure~\ref{sliding}. Let $D'$ be a disk obtained by applying an isotopy to the interior and the puncture of $D$ which moves it away from $C_{i_2}$ and has one puncture $l_1''$ on $C_{i_3}$, such that $\partial D'=\partial D$ and $(D')^\circ\cap\Cscr=\emptyset$. Let $C'_{i_1}$ be obtained by attaching $D'$ to $C_{i_1}$ along $l$, followed by a small perturbation (Figure \ref{sliding}). We say $C'_{i_1}$ is obtained from $C_{i_1}$ by {\emph{sliding}} the generating curve $l_1$ over the puncture $l_2$. Let $\mathscr{C}'$ be obtained from $\mathscr{C}$ by replacing $C'_{i_1}$ for $C_{i_1}$.  It is clear that $\mathscr{C}'$ is a simplifier equivalent to $\mathscr{C}$. If $\mathscr{C}$ is reduced, so is $\mathscr{C}'$. 
\end{remark}
	
\begin{lemma}\label{l10}
Let $N$ be a manifold with a simplifier $\mathscr{C}=\left(C_i\right)_{i=1}^n$. Let $D\subset C_i$ be a punctured disk such that $\partial D$ is disjoint from the generating curves on $C_i$. Let $D'$ be a disk in $N$ with $\partial D'=\partial D$ and $(D')^\circ\cap \mathscr{C}=\emptyset$. Let $C_i'$ be obtained from $C_i$ by removing $D$ and replacing it with $D'$, and $\mathscr{C}'$ be obtained from $\mathscr{C}$ by replacing $C_i'$ for $C_i$. Then $\mathscr{C}'$ is a simplifier for $N$ which is  equivalent to $\mathscr{C}$. If $\mathscr{C}$ is reduced, so is $\mathscr{C}'$.
\end{lemma}
\begin{proof}
By Remark  \ref{r2.2}, we can assume that $N$ is connected. In the manifold $N_i$, $D$ is a disk without punctures and there is a copy of $D'$ in $N_i$, which is again denoted by $D'$. Thus, $S=D\cup D'$ is a sphere in $N_i$ which is disjoint from $\mathscr{C}_i$, after a slight perturbation. By Proposition \ref{p5}, $S$ is separating. Since $S\cap\mathscr{C}_i=\emptyset$, it is clear that one component of $N_i\setminus S$ does not contain $\partial^+_i\amalg\partial^-_i$. Therefore, if we cut $N_i$ along $S$, and attach two $3$-disks $B_1$ and $B_2$ to the resulting boundary spheres, we obtain the manifolds $N'_i$ and $M$, where $M$ is a closed manifold and $\mathscr{C}_i\cap M=\emptyset$. By Remark \ref{r2.4}, $\mathscr{C}_i$ may also be regarded as a  simplifier $\mathscr{C}_i|_{N_i'}$ on $N'_i$. If $N'_0=N_i'[\mathscr{C}_i|_{N_i'}]$, we have $N_0=N'_0\# M$.	Note that $N'_0$ may be disconnected, so the connected sum is performed between $M$ and a connected component $N_0^1$ of $N'_0$. Therefore, 
\begin{equation}\label{eq1}
N_0=(N_0^1\# M)\amalg(N'_0\setminus N_0^1).
\end{equation}

$N'_i$ also inherits a simplifier from  $\mathscr{C}'_i$, which is denoted by $\mathscr{C}'_i|_{N_i'}$. Note that $\mathscr{C}_i|_{N_i'}$ and $\mathscr{C}'_i|_{N_i'}$  are equivalent, since  by moving $C_i$ through $B_1$, $C_i'$ is obtained.  If $N_0''=N_i[\Cscr_i']$, from $N_i=N'_i\# M$, we have $N_0''=N'_0\# M$, where $M$ is connected to the component $N_0^2$ of $N'_0$:
\begin{equation}\label{eq2}
N_0''=(N_0^2\# M)\amalg(N'_0\setminus N_0^2).
\end{equation}

If $M=S^3$, it is clear that $\mathscr{C}$ and $\mathscr{C}'$ are the same after an isotopy. So, let $M\neq S^3$. Since all the components of $N_0$, except possibly one, are $S^3$ and $N_0$ is irreducible, by (\ref{eq1}) we have $\overline{N}_0\cong M$ and  $\overline{N}'_0\cong S^3$. By (\ref{eq2}), $N_0''\doteq M$ and $\mathscr{C}_i'$ is a simplifier for $N_i$, which is equivalent to $\mathscr{C}_i$. Therefore, $\mathscr{C}$ and $\mathscr{C}'$ are equivalent simplifiers for $N$. 
\end{proof}
	
\begin{lemma}\label{l6}
Let $N'\xrightarrow{C} N$ and  $\mathscr{C}=\left(C_i\right)_{i=1}^n$ be a reduced simplifier for $N$. Let the solid cylinders $D_1\times I,D_2\times I\subset N$  associated with $C$ be such that $D_i\times\partial^\pm I$ is disjoint from $\partial^\pm\Cscr$. Then $N'$ admits a simplifier $\mathscr{C}'=(\Cscr'',C)$, such that $\Cscr''$ induces the simplifier $\left(C_i'\right)_{i=1}^n$ on $N$, which is reduced and equivalent to $\Cscr$. Moreover, $\partial^\pm C_i'=l_i^\pm=\partial^\pm C_i$ for $i=1,\ldots,n$. If $D_1\times I$ and $D_2\times I$ are disjoint from $C_n$, we can choose $C_n'$ equal to $C_n$. 
\end{lemma}
\begin{proof}
 Moving $\Cscr$ by an isotopy, we can assume that $D_i\times I$ cuts $\Cscr$ transversely,  away from  the punctures on the cylinders. Let $B_\Cscr(j_0)$ be the claim that every curve $l\in(\partial D_i\times I)\cap C_{n-j}$  is closed and bounds a punctured disk $D_l$ on $C_{n-j}$ away from the generating curves on $C_{n-j}$, for $0\leq j\leq j_0$. Since $\Cscr$ is reduced and $(\partial D_i\times I)\cap C_n$,  consists of closed curves, by Proposition~\ref{p5} $B_\Cscr(0)$ is satisfied {after removing the generating curves from $(\partial D_i\times I)\cap C_n$}. Suppose that $B_\Cscr(j)$ is satisfied for $j<j_0$ and let $l\in (\partial D_i\times I)\cap C_{n-j_0}$.  Since $D_i\times\partial^\pm I$ is disjoint from $\partial^\pm \Cscr$, every potential boundary point of $l$ is on a puncture $l_j$ of $C_{n-j_0}$, which  is  a generating curves on  $C_{n-j}$ for some $j<j_0$. Then, there is a curve $l'\in C_{n-j}\cap(\partial D_i\times I)$ which intersects the generating curve $l_j$ on $C_{n-{j}}$. This contradiction implies that all  the curves in $C_{n-{j_0}}\cap(\partial D_i\times I)$ are closed. Since $l_{n-{j_0}}^+$ is not the boundary of a disk in $\partial_{n-{j_0}}^+$, all  the curves in $C_{n-{j_0}}\cap(\partial D_i\times I)$ bound punctured disks in $C_{n-{j_0}}$ by Proposition~\ref{p5}.\\

If $l\in C_{n-{j_0}}\cap(\partial D_i\times I)$  bounds a punctured disk $D_l\subset C_{n-{j_0}}$, we can remove each generating curve $l'$ on $C_{n-{j_0}}$ from each $D_{l}$ by sliding it on the punctures of $D_l$. Using this process, we obtain an equivalent reduced simplifier $\mathscr{C}^1=\big(C^1_i\big)_{i=1}^n$ such that $B_{\Cscr^1}(j_0)$ is satisfied  (see Remark \ref{r-sliding}). By repeating this process, we obtain an equivalent reduced simplifier $\mathscr{C}^2=\big(C^2_i\big)_{i=1}^n$ such that $B_{\Cscr^2}(n-1)$ is satisfied. Furthermore, we have $\partial^\pm C^2_j=l_j^\pm$. Let 
		$$A_{\mathscr{C}^2}=\{l\in(D_i\times I)\cap C_j^2\ |\ i=1,2,\quad 1\leq j\leq n,\quad l=\partial D_l\quad\text{and}\quad D_l\subset\partial D_i\times I\},$$
		where $D_l$ are disks. Let $l\in A_{\Cscr^2}$ be such that $D_l^\circ\cap C_k^2=\emptyset$, $1\leq k\leq n$. Let $l=\partial D'$, $D'\subset C_j^2$. Using Lemma \ref{l10}, we can replace $D'$ with $D_l$ to obtain an equivalent reduced simplifier $\mathscr{C}^3$ such that $|A_{\mathscr{C}^3}|<|A_\mathscr{C}|$. So, if we repeat this process, we obtain an equivalent reduced simplifier  $\mathscr{C}^4=\big({C}_j^4\big)_{j=1}^n$ with $A_{\mathscr{C}^4}=\emptyset$. Thus, each curve in $(\partial D_i\times I)\cap\mathscr{C}^4$ is an essential curve on $\partial D_i\times I$. Note that if $C_j^4\cap(D_i\times I)\neq\emptyset$, then ${C}_j^4\cap(\partial D_i\times I)\neq\emptyset$ (otherwise, ${C}_j^4\subset D_i\times I$, especially $\partial^+C_j^4\subset D_i\times\{1\}$ which is impossible by our assumption on $\mathscr{C}^4$). Let $D_1^j,\ldots,D_{k_j}^j$ be punctured disks in $C_j^4$ such that $\partial D_l^j\in {C}_j^4\cap(\partial D_i\times I)$, and all the curves in ${C}_j^4\cap(\partial D_i\times I)$ are in $D_1^j\cup \cdots\cup D_{k_j}^j$. Let $C'_j$ be obtained by removing $(D_l^j)^\circ$ from ${C}_j^4$. Then, some generating curves which were punctures of $D_l^j$ are removed from some cylinders with indices greater that $j$. Then $\Cscr'=\big(C_1',\ldots,C_n',C\big)$ is a simplifier for $N$ with the desired properties.
\end{proof}
	
\begin{remark}\label{r7}
Let $\mathscr{C}=\left(C_i\right)_{i=1}^n$ be a simplifier for $N$, and $N'$ be a connected component of $N$ such that $\mathscr{C}\cap N'=\emptyset$. Thus, $N'$ is irreducible and the boundary of $N'$ consists of two spheres. Set $M=(N\setminus N')\# \overline{N}'$, where $\overline{N}'$ is connected (away from $\Cscr$) to one of the connected components of $N\setminus N'$ in the above connected sum.  So, there is a copy of $\mathscr{C}$ in $M$, denoted again by $\mathscr{C}$. Let  $M_n=M$ and $M_{i-1}$ denote the manifold obtained from $M_i=(N_i\setminus N')\# \overline{N}'$ by removing $C_i$.  All the connected components of $N_0$, except possibly one, are $S^2\times I$ and $N'$ is a connected component of $N_0$. Thus, $N_0\doteq M_0$, ${M_0}$ is irreducible and $\mathscr{C}$ is a simplifier for $M$.
\end{remark}
	
\begin{lemma}\label{l8}
Given a simplifier $\mathscr{C}$ for $N\in\Mfrak$, there is an equivalent reduced simplifier $\mathscr{C}'$ with $\partial^\pm\Cscr'\subset\partial^\pm\Cscr$.
\end{lemma}
\begin{proof}
 Let $\mathscr{C}=\left(C_i\right)_{i=1}^n$. We use an induction on $n$. If $n=0$, $\mathscr{C}$ is reduced, so assume $n>0$.  Since $\Cscr_{n-1}$ is a simplifier for $N_{n-1}$, by the induction hypothesis, there is  a reduced simplifier $\mathscr{C}'$ for $N_{n-1}$ which is equivalent to $\mathscr{C}$ and satisfies the conditions of the lemma. If $\partial^+C_n$ does not bound a disk in $\partial^+$, by adding $C_n$ to $\Cscr'$ we obtain a reduced simplifier for $N$ by Lemma \ref{l6}. \\
		
Now, let $\partial^+C_n$ be the boundary of a disk in $\partial^+$. By Remark \ref{r2.3}, $N=(N_{n-1}\setminus N')\# \overline{N}'$ where $N'$ is a component of $N_{n-1}$ with spherical boundary components. By the induction hypothesis, $N_{n-1}$ has a reduced simplifier $\mathscr{C}'=(C'_i)_{i=1}^m$ satisfying the conditions of the lemma. From Remark \ref{r2.2}, $\mathscr{C}'\cap N'$ is a simplifier for $N'$. Since the boundary components of $N'$ are spheres, we have $\mathscr{C}'\cap N'=\emptyset$. By Remark \ref{r7}, $\mathscr{C}'$ is a simplifier for $N$ and if $N'_0$ is obtained from $N$ by removing $\mathscr{C}'$, we have $N'_0\doteq N_0$. Thus, $\mathscr{C}$ and $\mathscr{C}'$ are equivalent.
\end{proof}
	
\begin{corollary}\label{c9}
Let $N'\xrightarrow{C}N$, where $N$ has a simplifier. Then $N'$ has a reduced simplifier. 
\end{corollary}
\begin{proof}
		This follows directly from Lemma \ref{l6} and Lemma \ref{l8}.
\end{proof}
	
\begin{lemma}\label{l11}
		Let  $\mathscr{C}=\left(C_i\right)_{i=1}^n$ be a simplifier for $N\in\Mfrak$ and $l\subset \partial^\pm$ be a closed curve which is disjoint from $\partial^\pm\Cscr$ and bounds a disk in $N$. Then  $l$ bounds a disk in $\partial^\pm$.
\end{lemma}
\begin{proof}
		By Lemma \ref{l8} and  Remark \ref{r2.2}, we may assume that $N$ is connected and $\Cscr$ is reduced. We prove the lemma by  induction on $n=|\mathscr{C}|$. Let $l\subset\partial^+$ be the boundary of a disk  $D$ in $N$ such that $D^\circ\subset N^\circ$. If $n=0$, it is clear that $l$ bounds a disk in $\partial^+$ (since the boundary components are spheres). If $n=1$ and $l$ is not the boundary of a disk in $\partial^+$, then there is a cylinder $C\subset\partial^+$ with  $\partial C=l\amalg l_1^+$. Thus, $l_1^+=\partial (D\cup C)$, which  is impossible by Proposition \ref{p5}. \\
		
		Suppose   $n>1$. $D\cap C_n$ consists of some closed curves (we assume $D$ intersects $\mathscr{C}$ transversely). By Proposition \ref{p5}, all the curves in $D\cap C_n$ are the boundaries of disks in $C_n$. Let $l'\in D\cap C_n$ be the boundary of a disk $D'\subset C_n$ such that $(D')^\circ \cap D=\emptyset$. $l'$ bounds a disk $D''\subset D$. Remove $D''$ from $D$ and replace it with $D'$. Denote the resulting disk with $D_1$. We have $\partial D_1=l$ and $|D_1\cap C_n|<|D\cap C_n|$. By repeating this process, we can obtain a disk $D_2$ such that $\partial D_2=l$ and $D_2\cap C_n=\emptyset$. Thus, $l$ bounds a disk  in $N_{n-1}$ as well. By the induction hypothesis, $l$ bounds a disk $D_3\subset\partial^+_{n-1}$. If $D_i^{C_n}\times\{1\}\cap D_3=\emptyset$ for $i=1,2$, there is a copy of  $D_3$ in $N$, denoted again by $D_3$, and the lemma is proved.\\
		
		 Otherwise, at least one of $D_i^{C_n}\times\{1\}$, say $D_1^{C_n}\times\{1\}$, is inside $D_3$. Then $S=D_2\cup D_3$ is a sphere which is separating by Proposition \ref{p5}. $\partial^+_{n-1}$ and $\partial^-_{n-1}$ cannot have components on the two sides of $S$. Let us first assume that $N_{n-1}$ is connected. Since $n>1$, if $\partial^+_{n-1}$ and $\partial^-_{n-1}$ are on the two sides of $S$, $C_{n-1}$ should intersect $S$ in some closed curves, and at least one of these closed curves is an essential curve on $C_{n-1}$. Since each curve on $S$ bounds a disk, by Proposition \ref{p5}, this cannot happen.	We may thus assume that $N_{n-1}$ has two connected components. Let $S$ be in the component $N'$ of $N_{n-1}$.  Since $l_n^+$ is not the boundary of a disk in $\partial^+$, the two connected components of $N_{n-1}$ have non-spherical boundary components. So $N'\cap\mathscr{C}\neq\emptyset$. We can assume $C_{n-1}\subset N'$ and as before, $S$ cannot separate $\partial^+_{n-1}$ from $\partial^-_{n-1}$.\\

		The connected component of  $N_{n-1}\setminus S$ which does not contain the boundary of $N_{n-1}$, is called the interior of $S$. $D_1^{C_n}\times I$ enters  the interior of $S$ in a neighborhood of $D_1^{C_n}\times\{1\}$. Thus, in order to reach the negative boundary (i.e. $\partial^-_{n-1}$), it should intersect $S$ in some other disks, which are necessarily on $D_2$. The boundary of at least one of these disks is an essential curve on $\partial D_1^{C_n}\times I$. But, $D_2$ is disjoint from $C_n$. This contradiction  completes the proof of the lemma.
	\end{proof}
	
	Let $N$,  $\mathscr{C}=\left(C_i\right)_{i=1}^n$ and $l_i^\pm=\partial^\pm C_i$ be as before. Let $D$ be a punctured disk in $C_k$, for some $1\leq k\leq n$, with $\partial D=l_1\cup l_2$ where $l_1\subset l_k^+$, $l_2^\circ\subset C_k^\circ$, and $l_2$ is disjoint from the generating curves on $C_k$. Let $D'$ be a punctured disk in $N$ such that  $(D')^\circ\cap \mathscr{C}=\emptyset$ and its punctures are essential curves on $C_j$, for $j>k$. Furthermore, we assume that these punctures are disjoint from the generating curves and that $\partial D'=l_1'\cup l_2$  where $l'_1\subset\partial^+$  is disjoint from $l_i^+$, $1\leq i\leq n$ (Figure \ref{figl12}-left). Suppose that $C'_k$ is obtained from $C_k$ by removing $D$ and replacing it with $D'$, and  $\mathscr{C}'$ is obtained from $\mathscr{C}$ by replacing $C_k$ with $C_k'$.

\begin{figure}
			\def\svgwidth{9cm}
			\begin{center}
				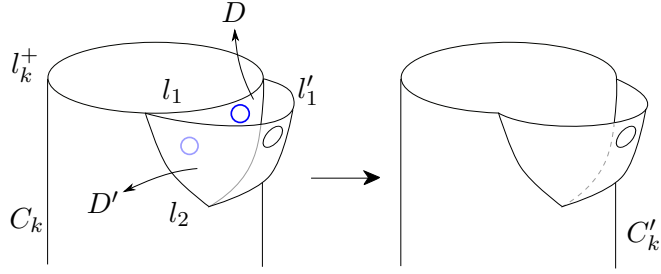
				\caption{ {$C'_k$ is obtained from $C_k$ by removing $D$ and replacing it with $D'$. The curves in blue denote the punctures on $C_k$ and $C'_k$.
				}}
				\label{figl12}
			\end{center}
		\end{figure}
			
\begin{lemma}\label{l12}
		Let $N$, $\mathscr{C}$ and $\mathscr{C}'$ be as above. Then, $\mathscr{C}'$ is a simplifier for $N$ which is equivalent to $\mathscr{C}$. If $\mathscr{C}$ is reduced, so is $\mathscr{C}'$.
\end{lemma}
\begin{proof}
		In $N_k$, $D$ and $D'$ are disks without punctures. Thus, $l_1\cup l_1'$ bounds the disk $D\cup D'\subset N_k$. Furthermore, $l_1\cup l_1'\subset\partial^+_k$ and is disjoint from $l_i^+$, $1\leq i\leq k$. From Lemma \ref{l11}, $l_1\cup l_1'$ bounds a disk $D''\subset\partial_k^+$. Set $S=D\cup D''\cup D'$. $S$ is a sphere disjoint from $C_k$. Thus, $S$ cannot separate $\partial^+_k$ and $\partial^-_k$. The connected component of  $N_k\setminus S$ which has $S$ on the boundary  and does not contain $\partial^+_k\amalg\partial^-_k$ is called the interior of $S$. $D''$ is disjoint from $C_i$, $1\leq i\leq k$. Otherwise, $C_i$ intersects $S$ in $D''$ and enters the interior of $S$. So, to reach $\partial^-_k$, $C_i$ should intersect $S$ somewhere in $S\setminus D''$. Since $D$ and $D'$ are  disjoint from $C_i$, this cannot happen.\\
		
		In $N_k$, remove $D$ from $C_k$ and replace it with $D'\cup D''$, to obtain $C'_k$. By Lemma \ref{l10}, $\mathscr{C}$ and the new collection $\mathscr{C}'$ of (punctured) cylinders are two equivalent simplifiers for $N_k$. Since the punctures of $D'$ are essential curves on $C_i$, for $i>k$, $\mathscr{C}'$  is also a simplifier in $N$ and, as stated above, it is equivalent to $\mathscr{C}$. Note that $=\partial^+C'_k$ is obtained by sliding $l_k^+$ over $l_i^+$, $i>k$.
\end{proof}
	
	With $N$ and $\mathscr{C}$ as above,  let $C$ be a cylinder in $N$ with 	$l^\pm=\partial^\pm C\subset\partial^\pm$,  $C^\circ\subset N^\circ$ and  $l^\pm\cap \partial^\pm\Cscr =\emptyset$. So, there are subsurfaces $\Sigma_{l^\pm}\subset\partial^\pm$ such that $l^\pm$ are among the  boundary components of $\Sigma_{l^\pm}$ and the other boundary components of $\Sigma_{l^\pm}$ are in $\partial^\pm\Cscr$. Let us further assume that {$l_i^+\subset \Sigma_{l^+}$ if and only if $l_i^-\subset\Sigma_{l^-}$}. Let $C$ be such that if $l_i^+\subset\Sigma_{l^+}$, then $C_i$ is a cylinder without punctures. Furthermore, if $C$ intersects $C_j$, then $l_j^+\cap\Sigma_{l^+}=\emptyset$ and $C\cap C_j$ consists of essential curves on $C$ which are the boundaries of (punctured) disks in $C_j$ and are disjoint from the generating curves on $C_j$. We call $\Sigma_{l^\pm}$ the surfaces associated with $l^\pm$. Define $\mathscr{C}'=\left(C_i'\right)_{i=1}^{n+1}$ by setting $C'_{n+1}=C$, while $C'_i$ is obtained as follows for $i\leq n$: if $D_j^i$, $1\leq j\leq k_i$, are (punctured) disks in $C_i$ such that $\partial D_j^i\subset C_i\cap C$, and all the curves in $C_i\cap C$ are in one of $D_j^i$, $1\leq j\leq k_i$, then, $C'_i$ is obtained from $C_i$ by removing $(D_j^i)^\circ$, $1\leq j\leq k_i$. It is clear that $\partial^+C'_i=l_i^+$. Let $N'_{i}=N[C_{i+1}',\ldots,C_{n+1}']$.
	
\begin{lemma}\label{l13}
		Let $N$, $\mathscr{C}$ and $\mathscr{C}'$ be as above. Then $\mathscr{C}'$ is a simplifier for $N$ and is equivalent to $\mathscr{C}$. If $l_i^+$ bounds a disk in $\partial^+N_i'$, it also bounds a   disk in $\partial^+_i$.
\end{lemma}
\begin{proof}
		If $n=0$, the claim follows from Remark \ref{r1}. For $n>0$, we use an induction on $k=|\Sigma_{l^+}\cap\partial^+\Cscr|$. If $k=0$, $l^\pm$ bounds a disk $D^\pm\subset\partial^\pm$ where $D^\pm\cap\partial^\pm\Cscr=\emptyset$. Then $S=D^+\cup C\cup D^-$ is a separating sphere. By Remark \ref{r2.2}, we can assume $N$ is connected. $\partial^+$ and $\partial^-$ are on the same side of $S$; otherwise, each cylinder $C_i$ intersects $C$ in at least one essential curve on $C_i$ which is not possible.\\
		
		If $N'=N[C]$, it follows that $\partial^\pm N'=\partial_C^\pm\amalg S^\pm$, where $S^\pm$ are spheres which include a copy of $D^\pm$, denoted again by $D^\pm\subset S^\pm$. Identify $\partial D^+$ with $\partial D_1^{C}\times\{1\}$ and assume that $S^+=D^+\cup(D_1^{C}\times\{1\})$, without loss of generality. Let us first assume that $\partial D^-$ is identified with $\partial D_2^{C}\times\{-1\}$ and $S^-=D^-\cup(D_2^{C}\times\{-1\})$. Then, $D^-\cup\partial_C^-$ and $D^+\cup\partial_C^+$ are on different sides of $S$ in  $N$ (Figure \ref{figl13-2}),  which is not possible. Therefore, $\partial D^-$ is identified with $\partial D_1^{C}\times\{-1\}$, $N'$ is disconnected and $N'=N''\cup M$, where $M$ has two sphere boundary components. Similar to Remark \ref{r2.3}, we find $N=N''\# \overline{M}$.\\

		\begin{figure}
			\def\svgwidth{7.5cm}
			\begin{center}
				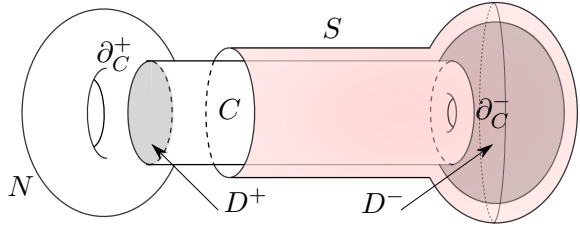
				\caption{If $\partial D^-$ is identified with $\partial D_2^{C}\times\{-1\}$, $D^-\cup\partial_C^-$ and $D^+\cup\partial_C^+$ are on two different sides of $S$ in the manifold $N$.}
				\label{figl13-2}
			\end{center}
		\end{figure}
		
		Set $m=|S\cap\mathscr{C}|$. Let $l'\in S\cap\mathscr{C}$ be the boundary of a disk $D'\subset S$ with $(D')^\circ\cap\mathscr{C}=\emptyset$. If $l'\subset C_i$, it bounds a punctured disk $D''\subset C_i$ (note that if $C_i$ intersects $S$, it intersects $C\subset S$). From Lemma \ref{l10}, if we remove $D''$ from $C_i$ and replace it with $D'$, we obtain an equivalent simplifier $\mathscr{C}_1$ with $|\mathscr{C}_1\cap S|<m$. Repeating this process, we obtain an equivalent simplifier $\mathscr{C}''$ which is disjoint from $S$. Thus $\mathscr{C}''\cap\overline{M}=\emptyset$. From Remark \ref{r2.4}, $\mathscr{C}''$ is a simplifier for $N''$ and if $N''_0=N''[\Cscr'']$, we have $N_0=N_0''\# \overline{M}$. One should note that $\mathscr{C}'=\mathscr{C}''$ in $N''$. If $N'_0=N[\Cscr']$, we have $N'_0=N_0''\cup M$. If $\overline{M}=S^3$, we have $N'_0\doteq N_0$, while in case $\overline{M}\neq S^3$ we find $N_0''=S^3$. Thus, $N'_0\doteq N_0$. This means that $\mathscr{C}'$ is a simplifier for $N$ and is equivalent to $\mathscr{C}$. Suppose that $l_i^+$ does not bound a disk in $\partial^+_i$. By the above discussion, $N'_i=N_i''\cup M$ and $N_i=N_i''\# \overline{M}$, where $N'_i=N[(\Cscr')^i]$. So, it is clear that $l_i^+$ does not bound a disk in $\partial^+N'_i$.\\
		
		Let us now assume that  $k>0$. According to Remark \ref{r2}, we can assume $l_n^+\subset\Sigma_{l^+}$. In $N_{n-1}$, $\Sigma_{l^+}$ changes to a subsurface of $\partial^+_{n-1}$, again denoted by $\Sigma_{l^+}$, which is the surface associated with $l^+$ in $\partial^+_{n-1}$. Furthermore, $|\Sigma_{l^+}\cap(\cup _{i=1}^{n-1}l_i^+)|<k$. Therefore, by the induction hypothesis, there is a simplifier $\mathscr{C}'=(C'_1,\ldots,C'_{n-1},C)$ in $N_{n-1}$ equivalent to $\mathscr{C}_{n-1}$. Adding $C_n$ to $N_{n-1}$, we obtain a simplifier $\mathscr{C}''=(C''_1,\ldots,C_n'',C)$ with $C''_n=C_n$, and $C'_i=C''_i[C_n]$, $i<n$. By the induction hypothesis, each $C'_i$, $i<n$, is obtained from $C_i$ by removing $(D_j^i)^\circ$. So, it is clear that each $C''_i$ is also obtained from $C_i$ by removing $(D_j^i)^\circ$ in $N$.
\end{proof}
\begin{remark}\label{r14}
Let $\mathscr{C}=\left(C_i\right)_{i=1}^n$ be a reduced simplifier as before and $\Sigma_{l^+}\subset\partial^+$ be a cylinder with $\partial\Sigma_{l^+}=l^+\cup l_n^+$, $C'_n\subset C_n$ be a subcylinder with $\partial C'_n=l^+_n\cup l'$ (where $l'$ is an essential curve on $C_n$) and $C'_n$ does not intersect the generating curves and  $C'\subset N$ be a cylinder with $\partial C'=l^+\cup l'$, $(C')^\circ\subset N^\circ$ and $(C')^\circ\cap\mathscr{C}=\emptyset$.     Let $C$ be the cylinder obtained from $C_n$ by replacing $ C'_n$ with $ C'$ (c.f. the construction before Lemma \ref{l13}).  Let $\mathscr{C}'=(C_1,\ldots,C_n,C)$ and $\mathscr{C}''=(C_1,\ldots,C_{n-1},C)$. It is clear that $\mathscr{C}''$ is a reduced simplifier. By Lemma \ref{l13}, $\mathscr{C}'$ is equivalent to $\mathscr{C}$. In $N'_{n-1}=N[C]$, $l^\pm_n$ are the boundaries of disks $D^\pm$, disjoint from $l_j^\pm$, $1\leq j<n$. By Lemma \ref{l13}, $\Cscr'[C]$ is equivalent to $\Cscr''[C]$ in  $N'_{n-1}$. Thus,  $\mathscr{C}'$ is equivalent to $\mathscr{C}''$ in $N$. Hence, $\mathscr{C}$ is equivalent to $\mathscr{C}''$. 
	\end{remark}

	\subsection{Nice intersections} 	
	We assume that $N$ is  connected and has a reduced simplifier $\mathscr{C}=\left(C_i\right)_{i=1}^n$,  unless stated otherwise. 
Let $C$ be a cylinder without punctures in $N$ with  $l^\pm=\partial^\pm C$  which intersects $\mathscr{C}$ transversely. $C\cap C_i$ is a $1$-dimensional submanifold  of $C$ and $C_i$. Each component of this submanifold can be a closed $1$-manifold or a $1$-manifold with boundary. We refer to the boundary components of a $1$-manifold with boundary as its {\emph{legs}}. $l\in C\cap C_i$ is called a \emph{simple bordered curve}, or a \emph{SBC} for short, if it has one leg on  $l^+$ and one on $l^-$. We call $C\cap\mathscr{C}$ of type $\textrm{I}$ if it does not contain any SBCs and we call $C\cap\mathscr{C}$ of type $\textrm{II}$ otherwise. It  is clear that if there is a SBC in $C\cap \mathscr{C}$,  none of the curves in $C\cap\mathscr{C}$ are essential curves on $C$. Furthermore, if a curve in $C\cap\mathscr{C}$ is an essential curve on $C$,  none of the curves in $C\cap\mathscr{C}$ are SBCs.

\begin{definition}\label{d14}
		We say that the (transverse) intersection of a cylinder $C$ with a punctured cylinder $C_i$ in  $\Cscr=\left(C_i\right)_{i=1}^n$  is {\emph{almost nice}} if every generating curve on $C_i$ separates the punctures on $C_i$ from $\partial_i^-$ and each component $l\in C\cap C_i$ which intersects the generating curves on $C_i$,  has at least one leg  on $\partial^-$. Moreover, if $l$ has exactly only one leg on $\partial^-$, then it intersects each generating curve on $C_i$ exactly once. The intersection of $C$ with $C_i$ is called {\emph{nice}} if it is almost nice and each component $l\in C\cap C_i$ which intersects the generating curves on $C_i$,  has exactly one leg on $\partial^-$, while there are no punctured disks or separating punctured cylinders  $D\subset C_i$ with $l^\circ\subset D^\circ$ and $\partial D\cap C\cap C_i^\circ=\emptyset$. Here, by a separating punctured cylinder $D$ on $C_i$ we mean a punctured cylinder so that in $C_i\setminus D^\circ$, $\partial_i^+$ and $\partial_i^-$ are in different connected components. 
\end{definition}

\begin{figure}
\def\svgwidth{9.5cm}
\begin{center}
			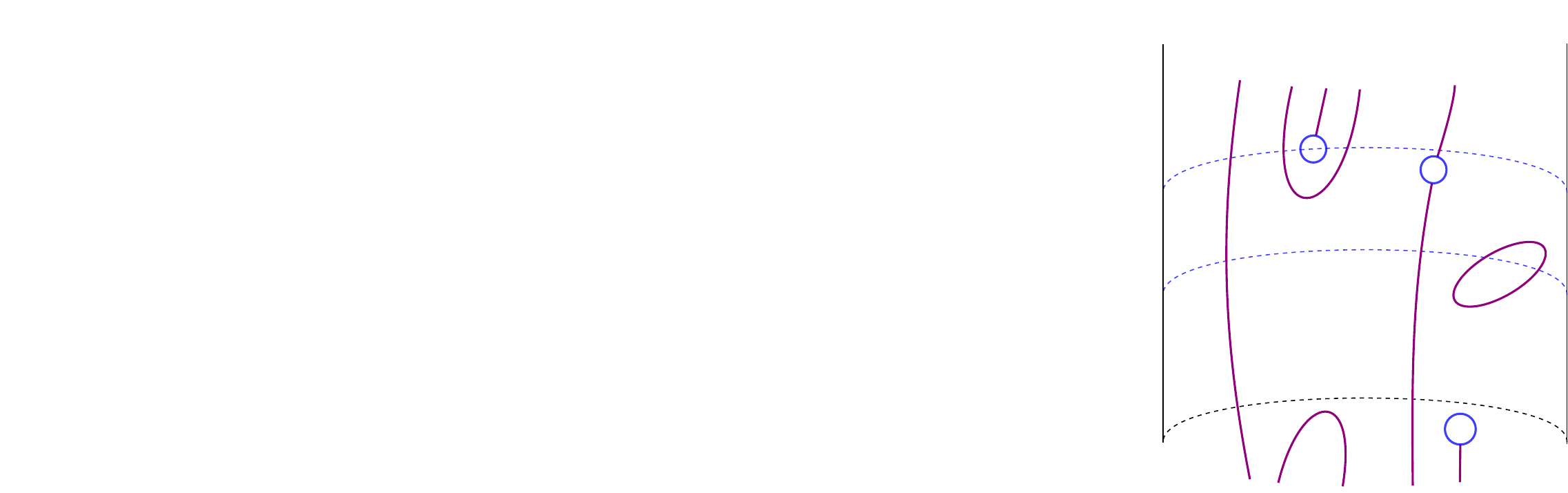\quad\quad\quad\quad
\caption{Left: $C\cap C_i$ is almost nice, but not nice. Right: $C\cap C_i$ is  nice.}
\label{fig5}
\end{center}
\end{figure}

\begin{lemma}\label{r16}
Let $C$ and $\Cscr=\left(C_i\right)_{i=1}^n$ be as before. Suppose that the intersection of $C$ with $C_j$ is nice and of type $\textrm{I}$, for $j_0\leq j\leq n$. Then for $j_0\leq j\leq n$, every curve $l\in C\cap C_j$ is disjoint from the generating curves on $C_j$ and has no legs on the punctures of $C_j$. 
\end{lemma}
\begin{proof}
We use reverse induction on $j$. If $l\in C\cap C_n$ intersects the generating curves on $C_n$, $l$ has exactly one leg on $\partial^-$. Since $C_n$ has no punctures, the other leg of $l$ is on $\partial^+$, which means $l$ is a SBC. Since $C\cap C_n$ is of type $\textrm{I}$, this cannot happen. Suppose now that the claim is true for every $l\in C\cap C_k$, for every $j_0\leq j<k\leq n$. To prove the inductive step, let $l\in C\cap C_j$ intersect the generating curves on $C_j$. So, $l$ has exactly one leg on $\partial^-$. Since $C\cap C_j$ is of type $\textrm{I}$, the other leg of $l$ is on a puncture $l_k$ of $C_j$ where $l_k$ is a generating curve on $C_k$, for some $k>j$. Thus, there is a curve $l'_k\in C\cap C_k$ such that $l'_k$ intersects the generating curve $l_k$  on $C_k$.  By the inductive hypothesis, this is impossible. A similar argument proves that there is  no $l\in C\cap C_j$ such that  $l$ has a leg on a puncture of $C_j$ (Figure \ref{figr16}), completing the proof by induction. 
\end{proof}		
		
		\begin{figure}
	\def\svgwidth{10cm}
		\begin{center}
			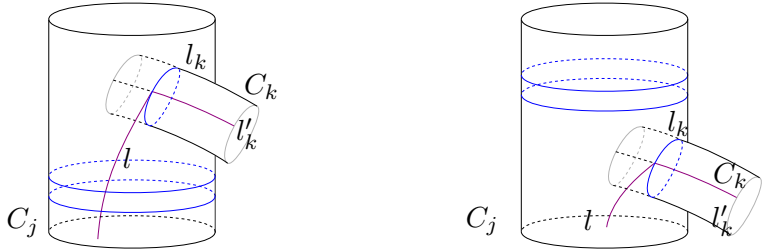
	\caption{$l\in C\cap C_j$ intersects the (blue) generating curves on $C_j$  and has one leg on $\partial^-$ and one on a puncture $l_k$ of $C_j$ (left).  $l\in C\cap C_j$ has a leg on a puncture $l_k$ of $C_j$. $l_k$ is a generating curve on a $C_k$ and there is a curve $l'_k\in C\cap C_k$ which intersects the generating curve $l_k$ (right). }
	\label{figr16}
	\end{center}
	\end{figure}

	\begin{lemma}\label{l17}
		Suppose that $C\cap C_k$ is nice, for $k>j$. There is an equivalent  reduced simplifier $\mathscr{C}'=\{C'_i\}_{i=1}^n$, with nice intersections with $C$ and of the same type as $\mathscr{C}$, such that $C'_k$ is the same as $C_k$  for $k>j$. 
	\end{lemma}
	
\begin{figure}
	\def\svgwidth{15cm}
		\begin{center}
			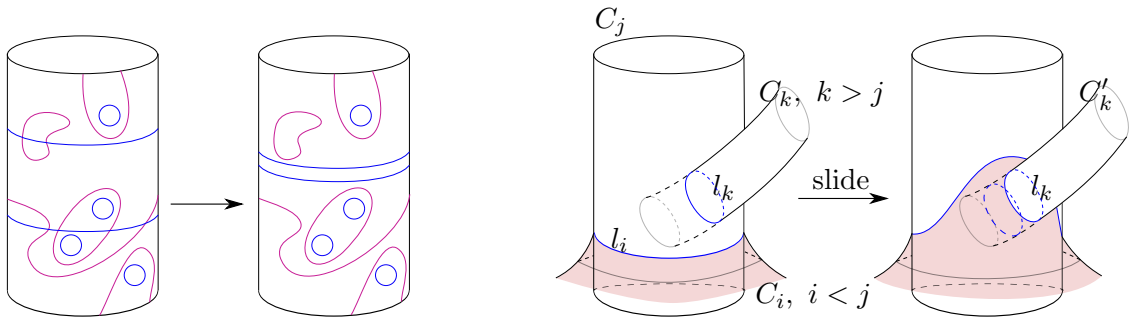
	\caption{After sliding the generating curves over the punctures, one may assume that the generating curves are disjoint from $C\cap C_j$ (left). Sliding the generating curve $l_i$ of $C_j$ on the puncture $l_k$, creates more generating curves on $C_k$, which are arbitrarily close to $l_k$ (right).}
	\label{figl17-2}
	\end{center}
\end{figure}
			
	\begin{proof}
Let us first assume that $C\cap\mathscr{C}$ is of type $\textrm{I}$. According to Lemma \ref{r16}, each $l\in C\cap C_k$, for $k>j$, is disjoint from the generating curves on $C_k$. Since the punctures of $C_j$ are generating curves on $C_k$, for $k>j$, there is no curves $l\in C\cap C_j$ with a leg on a puncture of $C_j$. By sliding the generating curves  on the punctures and a slight perturbation, we can separate the generating curves on $C_j$ from $C\cap C_j$. Therefore, the requirements of Definition \ref{d14} are satisfied (Figure \ref{figl17-2}-left). Sliding the generating curves on the punctures of $C_j$ may create more generating curves on $C_k$, for $k>j$. The new generating curves can be chosen arbitrarily close to the old generating curves on $C_k$ (Figure \ref{figl17-2}-Left). Thus, $C\cap C_k$, for $k>j$, remains nice. If we repeat this process for all the punctured cylinders $C_i$, for $i<j$, we obtain a reduced simplifier $\mathscr{C}'$ which is equivalent to $\mathscr{C}$ and satisfies the conditions of the lemma.\\

		Suppose now that  $C\cap\mathscr{C}$ is of type $\textrm{II}$. As for the intersections of type $\textrm{I}$, after sliding the generating curves of $C_j$ over its punctures, we may assume that every curve $l\in C\cap C_j$ which intersects the generating curves has precisely one leg on $\partial_j^-$ (Figure~\ref{figl17-4}). Again, note that by sliding the generating curves on the punctures of $C_j$, some new generating curves are created on $C_k$, for $k>j$, which can be chosen arbitrarily close to the old generating curves on $C_k$. Thus, $C\cap C_k$ remains nice for $k>j$. Repeating this process  for all the punctured cylinders $C_i$ with $i<j$, we obtain a reduced simplifier that satisfies the conditions of the lemma.  
\end{proof}
\begin{figure}[b]
	\def\svgwidth{9cm}
		\begin{center}
			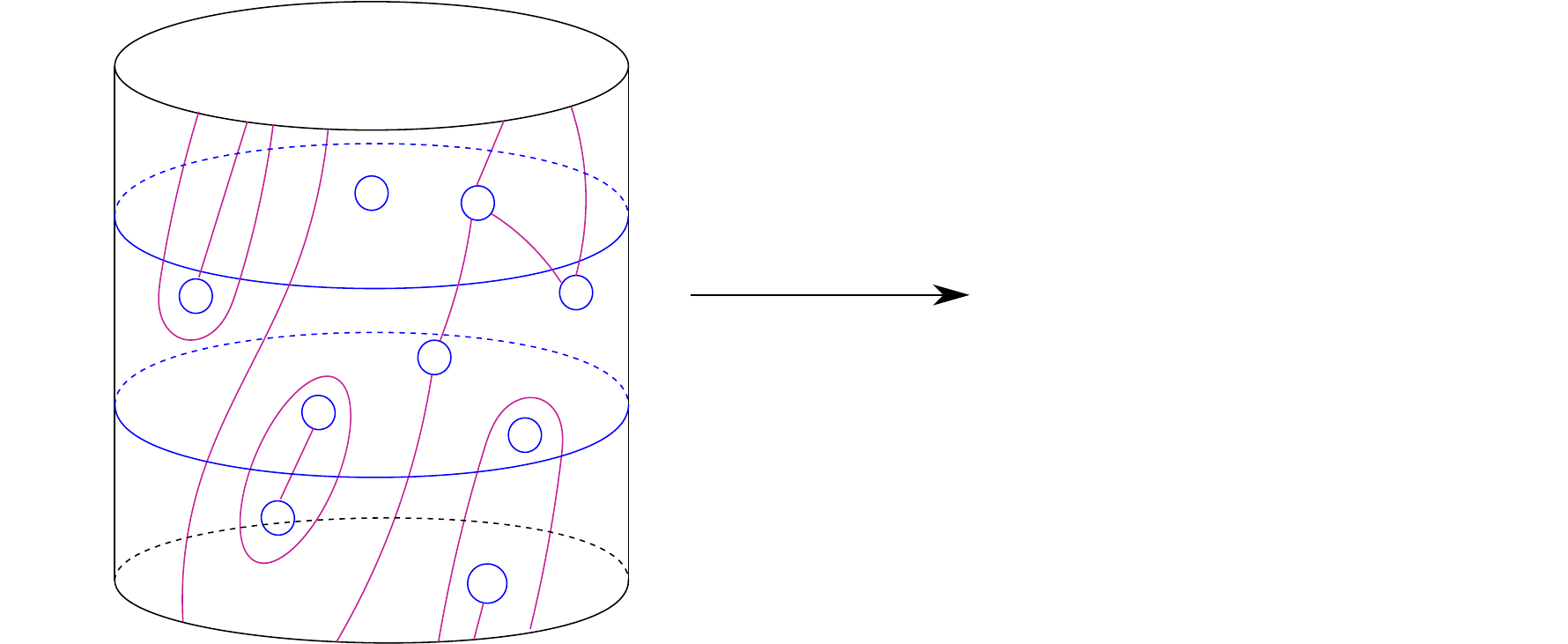
	\caption{After sliding the generating curves on $C_j$ over its punctures,  every curve in $C\cap C_j$ which cuts the generating curves has precisely one leg on $\partial^-_j$.}
	\label{figl17-4}
	\end{center}
\end{figure}

\section{Intersections of type \texorpdfstring{$\textrm{I}$}{Lg}}\label{s-type-1}
\subsection{Removing non-closed intersections and the spheres in \texorpdfstring{$N$}{Lg}}
Let  $\Cscr=\left(C_i\right)_{i=1}^n$ be a simplifier for $N\in\Mfrak$ as before, and $C$ be a cylinder in $N$ so that the transverse intersection of $C$ with $\Cscr$ is  of type $\textrm{I}$. In this section, we prove that $C$ sits in a simplifier $\mathscr{C}'$ for $N$ which is equivalent to $\Cscr$. 
	By Lemma \ref{l17}, we can assume that the intersection of $C$ with $\mathscr{C}$ is nice and by Lemma \ref{r16}, each curve in $C\cap\mathscr{C}$ is either closed, or has both legs on $\partial^+$, or has both legs on $\partial^-$. 

\begin{lemma}\label{l18}
If  the intersection of  $C$ with the reduced simplifier $\mathscr{C}$ is nice and of type $\textrm{I}$, there is an equivalent reduced simplifier which intersects $C$ nicely and their intersection only contains closed curves.
\end{lemma}
\begin{proof}
Suppose that $l\in C\cap C_k$ has nonempty boundary with $\partial l\subset\partial^+$. Let $D_l\subset C$ be a disk with $\partial D_l=l\cup l'$ and $l'\subset l^+=\partial^+C$. Choose $l$ such that $D_l\cap \mathscr{C}$ only consists of closed curves. Let $l_i\subset D_l^\circ\cap C_i$ be such that $l_i=\partial D_i$ where $D_i\subset C$ is a disk and $D_i^\circ\cap\mathscr{C}=\emptyset$ (Figure \ref{F1}). According to Proposition \ref{p5}, $l_i\subset C_i$ bounds  a punctured disk $D'_i\subset C_i$. By Lemma \ref{l10}, we can replace $D_i'$ with $D_i$ to obtain an equivalent reduced simplifier whose intersections with $C$ is a proper subset of the intersections of $\Cscr$ with $C$. By repeating this process, we  obtain an equivalent reduced simplifier $\mathscr{C}'=\{C_i'\}_{i=1}^n$ such that the intersection $C\cap\mathscr{C}'\subset C\cap \Cscr$ is of type $\textrm{I}$ and $D_l^\circ\cap\mathscr{C}'=\emptyset$. There is a disk $D_l'\subset C_k'$ such that $\partial D_l'=l\cup l''$, where $l''\subset l_k^+$. By Lemma \ref{l12}, we can replace $D'_l$ with $D_l$ to obtain an equivalent  simplifier with less non-closed intersections with $C$ (in comparison with the number of non-closed intersections in $C\cap\mathscr{C}$). By repeating this process, we  obtain an equivalent nice reduced simplifier $\mathscr{C}''$ such that $C\cap\mathscr{C}''$ does not contain curves with both legs on $\partial^+$. By a similar argument, we can remove all the curves in $C\cap\mathscr{C}''$ with two boundary components on $\partial^-$. 
\end{proof}
	
\begin{figure}[b]
		\def\svgwidth{5cm}
	\begin{center}
		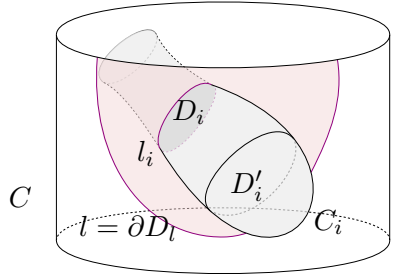
		
		\caption{ $l\in C\cap C_k$, $\partial l\subset\partial^+$ and $l$ is part of the boundary of a disk $D_{l}\subset C$ such that $(D_l)^\circ\cap\mathscr{C}$ only consists of closed curves. $l_i\in D_l\cap C_i$ bounds a disk
		$D_{i}\subset C$ and a punctured disk $D'_i\subset C_i$ with $(D_{i})^\circ\cap\mathscr{C}=\emptyset$.}
		\label{F1}
		\end{center}
\end{figure}

\begin{lemma}\label{l19}
		Let $N$ be a manifold with a reduced simplifier $\mathscr{C}=\left(C_i\right)_{i=1}^n$ which simplifies it to $N_0\doteq S^2\times I$.  Then every sphere $S\subset N$ which does not separate $\partial^+$ and $\partial^-$, bounds a $3$-disk. 
\end{lemma}
\begin{proof}
		We use an induction on $n$. If $n=0$,  $N=N_0\doteq S^2\times I$ and the claim is clear. To prove the inductive step, we can assume that $S\cap C_n$ consists of some closed curves and that $S\subset N^\circ$. Set $m=|S\cap C_n|$.  We prove the inductive step  by (a second) induction on $m$. If $m=0$, there is a copy of $S$ in $N_{n-1}$ which is also denoted  by $S$.  Since $S\cap C_n=\emptyset$, $S$ cannot separate $\partial^+_{n-1}$ and $\partial^-_{n-1}$. Thus, by the induction hypothesis, $S=\partial B$, where $B$ is a $3$-disk in $N_{n-1}$. Since $S\cap C_n=\emptyset$, we have $B\cap(D_i^{C_n}\times I)=\emptyset$ for $i=1,2$. Therefore, $S$ bounds a disk in $N$.\\
		
		Now, suppose that $m>0$. Each curve $l\in S\cap C_n$ bounds a disk in $S$. By Proposition \ref{p5}, $l$ bounds a disk $D$  in $C_n$. Choose $l$ such that $D^\circ\cap S=\emptyset$. Cut $S$ along $D$ and attach copies of $D$ to the  resulting boundaries to obtain the spheres $S_1$ and $S_2$ (having $D$ in common) with $|S_i\cap C_n|<m$, for $i=1,2$. $S_1$ and $S_2$ cannot separate $\partial^+$ and $\partial^-$. Otherwise, $C_n$ intersects $S_1$ or $S_2$ in at least one essential curve on $C_n$. Such a curve bounds a disk on $S_1$ or $S_2$ which is impossible by Proposition \ref{p5}. By the induction hypothesis, $S_i=\partial B'_i$, for $i=1,2$, where $B'_1$ and $B'_2$ are $3$-disks. If $(B'_1)^\circ\cap(B'_2)^\circ=\emptyset$, we have $S=\partial(B'_1\cup B'_2)$. If $B'_1\subset B'_2$, we have $S=\partial(B'_2\setminus(B'_1)^\circ)$.
Thus, $S$ bounds a $3$-disk in $N$. This completes the second induction, and thus the inductive step of the first induction.
\end{proof}

\begin{lemma}\label{l20}
With $N\in\Mfrak$ and  $\Cscr$  as before, let $S_1,\dots,S_k$ be disjoint spheres in $N^\circ$ which do not bound disks in $N$ and do not separate $\partial^+$ from $\partial^-$. Then, these spheres are parallel.
\end{lemma}
\begin{proof}
		We use an induction on $n$. If $n=0$, $N_0=N$. Since $\overline{N}$ is irreducible, each $S_i$ bounds a $3$-disk $D_i\subset\overline{N}$. Since $S_i$ does not separate $\partial^+$ and $\partial^-$ and  does not bound a $3$-disk in $N$, $\partial^+\amalg\partial^-\subset D_i$. Therefore, $D_i$ are not disjoint and $S_1,\ldots,S_k$ are parallel. \\
		
		When $n>0$, we can assume that $S_1,\ldots, S_k$ intersect $C_n$ transversely. We prove the inductive step by (a second) induction on $m=\sum_{i=1}^k|S_i\cap C_n|$. If $m=0$, there is a copy of each $S_i$  in $N_{n-1}$,  yet denoted by $S_i$. If  $S_i$ bounds a $3$-disk $B\subset N_{n-1}$, then $\partial^+_{n-1}\amalg\partial^-_{n-1}$ is in  $N_{n-1}\setminus B$. Since $S_i$ does not intersect $D_1^{C_n}\times I$ and $D_2^{C_n}\times I$, it follows that $B$ does not cut them either. Therefore, $S_i$ bounds a $3$-disk in $N$ as well, which is not possible. It thus follows that $S_1,\ldots, S_k$ do not bound disks in $N_{n-1}$.  Furthermore, $S_i$ does not separate $\partial^+_{n-1}$ from $\partial^-_{n-1}$, for  $1\leq j\leq k$. Otherwise, at least one of $D_1^{C_n}\times I$ and $D_2^{C_n}\times I$(say the first one) intersects $S_i$ in an essential curve on $\partial D_1^{C_n}\times I$, which cannot happen. The only problem in using the induction hypothesis is that $N_{n-1}$ may be  disconnected. If this is the case, since $N$ is connected, $N_{n-1}$ has two connected components, denoted $N_{n-1}^1$ and $N_{n-1}^2$. Set $\mathscr{C}^i=\mathscr{C}\cap N_{n-1}^i$, $i=1,2$. By Remark \ref{r2.2}, $\mathscr{C}_{n-1}\cap N_{n-1}^i$ is a simplifier for $N_{n-1}^i$. If $N_0^i=N_{n-1}^i[\mathscr{C}_{n-1}\cap N_{n-1}^i]$, we have $N_0=N_0^1\amalg N_0^2$. Thus, we may assume that $N^1_0\doteq S^2\times I$ . If one of the spheres, say  $S_j$, is in $N_{n-1}^1$, $S_j$ bounds a $3$- disk in $N_{n-1}^1$ by Lemma \ref{l19}, which  is not possible. Thus,  $S_1,\ldots,S_k$ are all in the same component of $N_{n-1}$. Therefore, by the induction hypothesis $S_1,\ldots,S_k$ are parallel in $N_{n-1}$. By Proposition \ref{p5}, each $S_i$ is a separating sphere in $N_{n-1}$. Thus, $\partial^+_{n-1}$ and $\partial^-_{n-1}$ are on one side of each $S_i$, which is called its interior. The region bound between any two spheres $S_i$ and $S_j$ (which is homeomorphic to $S^2\times I$) is thus disjoint from $D_1^{C_n}\times I$ and $D_2^{C_n}\times I$, since $S_i$ and $S_j$ are disjoint from these solid cylinders. Therefore, $S_i$ and $S_j$ are also parallel in $N$.\\

		Let us now assume that  $m>0$. Each $l\in S_i\cap C_n$ bounds a disk $D_l$ in  $C_n$. Choose $l$ such that $D_l^\circ\cap S_j=\emptyset$ for $j=1,\ldots,k$. Cut $S_i$ along $D_l$ and attach two copies of $D_l$ to the two resulting circular boundaries. We then obtain the spheres $S_i^1$ and $S_i^2$ such that
$$|S_i^1\cap C_n|+|S_i^2\cap C_n|+\sum_{j\neq i}|S_j\cap C_n|<m.$$
		$S_i^1$ and $S_i^2$ do not separate $\partial^+$ and $\partial^-$; otherwise $C_n$ intersects $S_i^1$ and $S_i^2$ in at least one essential curve on $C_n$ which is not possible by Proposition \ref{p5}. Since $S_i$ does not bound a $3$-disk, at least one of $S_i^1$ and $S_i^2$ (say $S_i^1$) does not bound $3$-disk.   By the induction hypothesis,  $S_1,\ldots,S_{i-1},S_i^1,S_{i+1},\ldots,S_k$ are parallel. Since $S_i$ does not bound a $3$-disk, $S_i^2$ is not parallel to  $S_i^1$. It follows that $S_i^2$ bounds a $3$-disk in $N$, since otherwise, by the induction hypothesis $S_i^1$ and $S_i^2$ are parallel. From here, it follows that $S_i$ is parallel with the other spheres.
\end{proof}

\subsection{Proof of the main theorem for intersections of type \texorpdfstring{$\mathbf{\mathrm{I}}$}{Lg}}
\begin{remark}\label{r21} If  $N_{n-1}$ is disconnected, it has two components $N_{n-1}^1$ and $N_{n-1}^2$. We may assume that $D_i^{C_n}\times I$ is in the component $N_{n-1}^i$.  First, suppose that $\mathscr{C}\cap N_{n-1}^i=\emptyset$. Then the boundary of $N_{n-1}^i$ consists of two spheres and $(\partial_{n-1}^+\cap N_{n-1}^1)\setminus(D_i^{C_n}\times\{1\})$ is a disk $D_i'$. It then follows that there is a copy of $D_i'$ in $N$ with $l_n^+=\partial D_i'$ which is not possible. Thus, $\mathscr{C}\cap N_{n-1}^i\neq\emptyset$, for $i=1,2$. In particular, when $n=1$, this argument implies that $N_0$ is connected.
\end{remark}	
\begin{proposition}\label{p23}
		If the cylinder $C$ in $N\in\Mfrak$ cuts the reduced simplifier  $\Cscr=\left(C_i\right)_{i=1}^n$ transversely and of type   $\textrm{I}$,  there is a simplifier  equivalent to $\mathscr{C}$ which includes $C$.
\end{proposition}
\begin{proof}
		 We use an induction on $n$. If $n=0$, $\mathscr{C}'=\{C\}$ is a simplifier equivalent to $\mathscr{C}$ (see Remark \ref{r1}). For the inductive step, by Lemmas \ref{l17} and \ref{l18} we can assume that $C\cap\mathscr{C}$ is nice and consists of closed curves. Given $l\in C\cap C_i$, let us first assume that $l=\partial D_l$, for a disk $D_l\subset C$. By Proposition \ref{p5}, $l$ bounds a punctured disk $D_l'$ in $C_i$. Moreover, $l$ can be chosen so that $D_l^\circ\cap\mathscr{C}=\emptyset$. By Lemma \ref{l10}, we can replace $D_l'$ in $C_i$ with $D_l$ to obtain an equivalent reduced simplifier with  fewer  intersections with $C$ bounding disks on $C$. Repeating this process, we may assume that all the curves in $C\cap\mathscr{C}$ are essential curves on $C$.\\ 
		
Next, assume that $l^+=\partial^+C$ bounds a  disk $D^+\subset\partial^+$. If $l\in C\cap C_i$ (which implies that $l$ is essential on $C$), let  $C_l\subset C$ be a cylinder with two boundary components $l$ and $l^+$. Then $l$ bounds the disk $D_+=C_l\cup D^+$. $l\subset C_i$ bounds a (punctured) disk in $C_i$ and $D^+\cap\{l_i^+\}_{i=1}^n=\emptyset$. By Lemma \ref{l13}, there is an equivalent reduced simplifier which includes $C$. We may thus assume that $l^+$ does not bound a disk in $\partial^+$, and similarly, $l^-=\partial^-C$ does not bound a disk in $\partial^-$.\\

 Let $l\in C\cap C_n$ (which is essential on $C$) bound a disk $D_l\subset C_n$ with $C\cap D_l^\circ=\emptyset$. Let $C_l\subset C$ be a sub-cylinder with two boundary components given by $l$ and $l^+$. Thus, $l^+$ bounds the disk  $D_l\cup C_l$. By Lemma \ref{l11}, $l^+$ also bounds a disk in $\partial^+$, which is not possible. It follows from this observation that all the curves in $C\cap C_n$ are essential curves on $C_n$. Choose $l\in C\cap C_n$  such that $C_l^\circ\cap C_n=\emptyset$ and set $C_+=C_l$. There is a copy of $C_+$ in $N_{n-1}$, which is again denoted by $C_+$. In $N_{n-1}$, $l$ bounds a disk $D_+$ which is disjoint from $C_i$, $1\leq i\leq n-1$. Thus, $l^+$ bounds the disk $D_+\cup C_+$ in $N_{n-1}$ (Figure \ref{figp23-2}-left). Again by Lemma \ref{l11}, $l^+$ bounds the disk $D^+\subset\partial^+_{n-1}$. Similarly, $l^-$ bounds a disk $D^-\subset\partial^-_{n-1}$. If $D^+$ is disjoint from both $D_1^{C_n}\times\{1\}$ and $D_2^{C_n}\times\{1\}$, it follows that there is a copy of $D^+$ in $\partial^+$, also denoted $D^+$, such that $l^+=\partial D^+$. This means that $l^+$ bounds a disk in $\partial^+$, which contradicts our assumption.\\

Therefore, $D^+$ includes at least one of $D_1^{C_n}\times\{1\}$ and $D_2^{C_n}\times\{1\}$. Let us first assume that it includes both of the aforementioned disks (Figure \ref{figp23-2}-left). If $n=1$, $l^+$ bounds a disk in $\partial^+$ which is not possible. Therefore $n>1$, and we may choose the disk $D_+\subset D_1^{C_n}\times I$ which bounds $l$ as before. The sphere $S_+=D^+\cup D_+\cup C_+$ in $N_{n-1}$ is then separating. If $N_{n-1}$ is connected, $S_+$ cannot separate $\partial^+_{n-1}$ and $\partial^-_{n-1}$, since otherwise, each $C_i$  intersects $S_+$ in some curves and at least one of them is essential on $C_i$. This cannot happen since all the curves on $S_+$ bound disks. On the other hand, if $N_{n-1}$ is disconnected, $S_+$ is included in a component $N_{n-1}^1$ of $N_{n-1}$. Then  $\mathscr{C}\cap N_{n-1}^1\neq\emptyset$, by Remark \ref{r21}. As we argued in the previous case, it follows that  $S_+$ cannot separate $\partial^+_{n-1}\cap N_{n-1}^1$ and $\partial^-_{n-1}\cap N_{n-1}^1$. Therefore $\partial^+_{n-1}$ and $\partial^-_{n-1}$ are on the same side of $S_+$. The component of  $N_{n-1}\setminus S_+$ which contains $\partial^+_{n-1}\cup\partial^-_{n-1}$ is called the exterior of $S_+$ and the other component which has $S_+$ on its boundary is called the interior $S_+$. $D_2^{C_n}\times I$ enters the interior of $S_+$ in a neighborhood of $\partial^+_{n-1}$. To reach $\partial^-_{n-1}$, it should intersects $S_+$ again on $D_+\cup C_+$. But $C_+^\circ\cap C_n=\emptyset$, so $C_+^\circ\cap (D_2^{C_n}\times I)=\emptyset$. Furthermore, $D_+\cap (D_2^{C_n}\times I)=\emptyset$. This rules out the possibility that $D^+$ includes both of $D_1^{C_n}\times\{1\}$ and $D_2^{C_n}\times\{1\}$.\\

\begin{figure}
			\def\svgwidth{15cm}
			\begin{center}
			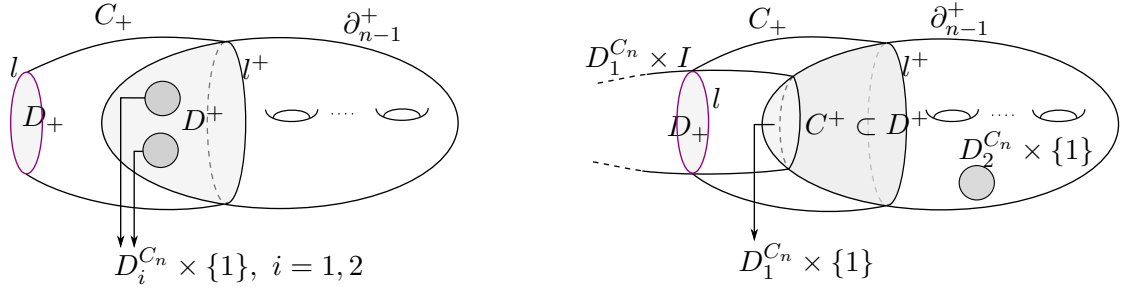
			\caption{In  $N_{n-1}$,  $l$ bounds a disk $D_+$ and $l^+$ bounds the disk $D^+$, where $D^+$ is a disk in $\partial^+_{n-1}$.  $D_i^{C_n}\times\{1\}$  are both in $D^+$ for $i=1,2$ (left). $D_1^{C_n}\times\{1\}\subset D^+$, $D^+\cap D_2^{C_n}\times\{1\}=\emptyset$ and  $\partial C^+=l^+\cup(\partial D_1^{C_n}\times\{1\})$ (right).}
			\label{figp23-2}
			\end{center}
\end{figure}
				
		The above two observations imply that precisely one of $D_1^{C_n}\times\{1\}$ and $D_2^{C_n}\times\{1\}$ is in $D^+$. Since $D_1^{C_n}\times I$ intersects $S_+$ in $D_+$, $D_1^{C_n}\times\{1\}\subset D^+$ (Figure \ref{figp23-2}-right).   $S_+=D^+\cup D_+\cup C_+$ cannot separate $\partial^+_{n-1}$ and $\partial^-_{n-1}$. This was shown above when $n>1$. If $n=1$, by Remark  \ref{r21}, $N_0$ is connected. Now, if $S_+$ separates $\partial^+_0$ and $\partial^-_0$, $D_2^{C_n}\times I$ intersects $S_+$ in at least one essential curve on $\partial D_2^{C_n}\times I$. $C_+\cap C_1=\emptyset$ and $D^+\cap (D_2^{C_n}\times\{1\})=\emptyset$. Thus, $D_2^{C_n}\times I$ intersects $D_+$ which is not possible. There is a cylinder $C^+\subset\partial^+_{n-1}$ with two boundary components given by $l^+$ and $(\partial D_1^{C_n})\times\{1\}$. Since $(C^+)^\circ\cap D_i^{C_n}\times I=\emptyset$ for $i=1,2$, there is a copy of $C^+$ in $\partial^+$ such that $\partial C^+=l^+\cup l_n^+$.\\

\begin{figure}[!h]
		\def\svgwidth{15cm}
		\begin{center}
		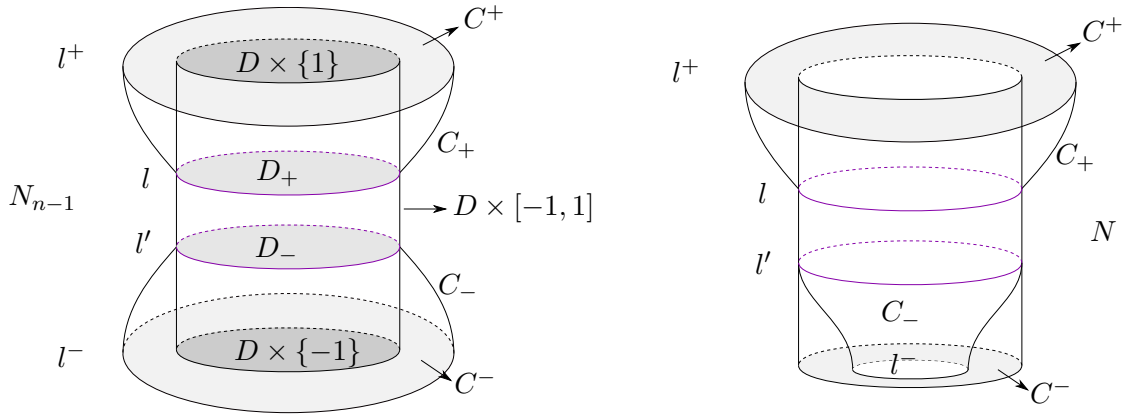
		\caption{$\partial C^-=l^-\cup(\partial D_1^{C_n}\times\{0\})$. $S_+=D_+\cup C_+\cup D'_+$ and $S_-=D_-\cup C_-\cup D'_-$ are disjoint spheres in $N_{n-1}$, where $D'_\pm=(D_1^{C_n}\times\{\pm 1\})\cup C^+$ (left). If $\partial C^-=l^-\cup(\partial D_2^{C_n}\times\{-1\})$,  $C_+$ and $C_-$ are as illustrated in $N$ (right).}
		\label{figp23-4}
		\end{center}
\end{figure}		
		
Similarly, if for a curve $l'\in C\cap C_n$ and a sub-cylinder $C_-\subset C$ we have $\partial C_-=l^-\cup l'$ and $C_-^\circ\cap C_n=\emptyset$, there is a cylinder $C^-\subset\partial^-$ with $\partial C^-=l^-\cup l_n^-$ (note that we may have $l=l'$).  $C^-$ is a cylinder in $\partial^-_{n-1}$ and $\partial C^-=l^-\amalg (\partial D_i^{C_n}\times\{-1\})$ for $i=1$ or $i=2$.\\ 			
			
If $i=1$, $C_+$ and $C_-$ are two disjoint cylinders in $N_{n-1}$, where each one has a boundary component on $\partial D_1^{C_n}\times I$ (Figure \ref{figp23-4}-left). Set $S_-=D_-\cup C_-\cup D'_-$, where $D_-$ is a disk in $D_1^{C_n}\times I$ which is bounded by $l'$ and $D'_-=(D_1^{C_n}\times \{-1\})\cup C^-$.  As discussed for  $S_+$, $S_-$ is a sphere which does not separate $\partial^+_{n-1}$ and $\partial^-_{n-1}$. $S_+$ and $S_-$ are in the same component $N'$ of $N_{n-1}$. Suppose that $S_+$ and $S_-$ do not bound $3$-disks in $N'$. By Lemma \ref{l20}, $S_+$ and $S_-$ are parallel in $N'$. On the other hand, $\partial^+_{n-1}$ is adjacent to $S_+$ and $\partial_{n-1}^-$ is adjacent to $S_-$. Since $\partial_{n-1}^+$ and $\partial_{n-1}^-$ are not in the region $B\simeq S^2\times I$  bounded between $S_+$ and $S_-$, they are on two different components of $N'\setminus B$. Therefore, $S_+$ and $S_-$ separate $\partial_{n-1}^+$ and $\partial_{n-1}^-$. As a result, at least one of $S_+$ or $S_-$, say $S_+$ bounds  a $3$-disk.\\
			
If $i=2$,	take $D_-\subset D_2^{C_n}\times I$ to be a disk with $\partial D_-=l'$. Set $S_-=D_-\cup C_-\cup D'_-$, where $D'_-=(D_2^{C_n}\times\{-1\})\cup C^-$ is a disk.  $S_-$ does not separate $\partial^+_{n-1}$ and $\partial^-_{n-1}$, as before. If $N_{n-1}$ is connected, $S_+$ and $S_-$ are in the same component. Similar to the case $i=1$, it follows that one of $S_+$ and $S_-$, say $S_+$, bounds a $3$-disk. On the other hand, if $N_{n-1}$ is disconnected, $S_+$ and $S_-$ are in two different components $N_{n-1}^1$ and $N_{n-1}^2$ of $N_{n-1}$, respectively. If $N_0^i=N_{n-1}^i[\mathscr{C}_{n-1}\cap N_{n-1}^i]$ for $i=1,2$, we have $N_0=N_0^1\amalg N_0^2$. So, at least for one of ${N_0^1}$ or ${N_0^2}$, say ${N_0^1}$, $N_0^1\doteq S^2\times I$. Lemma \ref{l19} implies that $S_+$ bounds a $3$-disk.\\
				
The above argument allows us to assume that $S_+=\partial B$, where $B$ is a $3$-disk. Let $\bar{C}_n\subset C_n$ be a sub-cylinder with two boundary components $l_n^+$ and $l$. Move the cylinders $C_i$ out of $B$ by an isotopy, for $1\leq i\leq n-1$. Then, all the intersections of $C_i$ with $C_+$ are removed for $i<n$. It is also clear that by moving the generating curves,  we can remove them from $\bar{C}_n$. Let $C'_n$ be obtained from $C_n$ by replacing $\bar{C}_n$ with $C_+$. By Remark \ref{r14}, $\left(C_1,\ldots, C_{n-1},C_n'\right)$ is a reduced simplifier, equivalent to $\mathscr{C}$ such that $|C\cap C'_n|<|C\cap C_n|$. By repeating this process, we obtain an equivalent reduced simplifier $\Cscr'=\left(C_1,\ldots, C_{n-1},C_n''\right)$ such that $C\cap C''_n=\emptyset$. Thus, there is a copy of $C$ in $N[C''_n]$, which is again denoted by $C$. Moreover the intersection of $C$ with $\Cscr'_{n-1}$  is of type $\textrm{I}$. By the induction hypothesis, there is a reduced simplifier, which is equivalent to $\Cscr'_{n-1}$ and includes $C$. Adding $C_n''$ to this reduced simplifier (in the expense of creating some new punctures) we obtain a simplifier $\mathscr{C}''$ in $N$ which includes $C$. In fact, since $C$ is disjoint from $C_n''$, $D_i^{C_n''}\times I$  are disjoint from $C$. This completes the proof of the proposition.
\end{proof}

\section{Intersections of type \texorpdfstring{$\textrm{II}$}{Lg}}\label{s-type-2}	
Throughout this section, $\mathscr{C}=\left(C_i\right)_{i=1}^n$ is a reduced simplifier for $N$ and  the intersection of the cylinder $C$ with $\mathscr{C}$ is transverse and of type $\textrm{II}$. The goal is to obtain an equivalent simplifier which includes $C$. The number of generating curves on $C_i$ is denoted by $g_i(\mathscr{C})$.
	
\subsection{Removing closed intersections}	
\begin{lemma}\label{l24}
	Suppose that $C\cap C_i$ is almost nice for $1\leq i\leq n$. Then there is an equivalent reduced simplifier $\mathscr{C}'=\{C'_i\}_{i=1}^n$ with $C\cap C_i'\subset C\cap C_i$ and $g_i({\mathscr{C}'})\leq g_i(\mathscr{C})$, such that  each curve in $C\cap\mathscr{C}'$ has nonempty boundary. In particular, if $C\cap C_i$ is nice then  $C\cap C_i'$ is nice as well. 
	
\end{lemma}
\begin{proof}
		Let $l\in C\cap C_i$ be a closed curve. Since $C\cap\mathscr{C}$ is of type $\textrm{II}$, $l$ bounds a disk $D_l$ on $C$. By Proposition \ref{p5}, $l$ bounds a (punctured) disk $D_l'$ on $C_i$. Choose $l\in C\cap C_i$ such that  $D_l^\circ\cap \mathscr{C}$ does not contain any closed curves. Since $C\cap C_i$ is (almost) nice, $l$ is disjoint from the generating curves on $C_i$. Thus, there is no curve in $C\cap\mathscr{C}$ with a leg on $l$. We also have $D_l^\circ\cap\mathscr{C}=\emptyset$. In fact, if $l_k\in C\cap C_k$ and $l_k^\circ\subset D_l^\circ$,  $l_k$ has nonempty boundary. It is clear that $\partial l_k\cap(\partial^+\amalg\partial^-)=\emptyset$. So, there is a puncture $l_{k'}\subset C_k$ such that $\partial l_k\cap l_{k'}\neq\emptyset$. Then $l_{k'}$ is a generating curve on a punctured cylinder $C_{k'}$ with $k'>k$. Therefore, there is a curve $l'_{k'}\in C\cap C_{k'}$ such that  $l'_{k'}\cap l_{k'}\neq\emptyset$ (Figure \ref{figl24}). By the definition of (almost) nice intersections, $l'_{k'}$ has a leg on $\partial^-$. This is not possible since $l'_{k'}\subset D_l$. So, $D_l^\circ\cap\mathscr{C}=\emptyset$. By Lemma \ref{l10}, replacing $D_l'$ with $D_l$  we obtain an equivalent reduced simplifier  $\mathscr{C}'$ such that $C\cap\mathscr{C}'$ is of type $\textrm{II}$, $C\cap C_i^1$ is (almost) nice, and $\Cscr'$ has fewer  closed intersections with $C$. Repeating this process, we obtain the desired equivalent reduced simplifier.
\end{proof}
		
\begin{figure}[b]
	\def\svgwidth{6cm}
		\begin{center}
			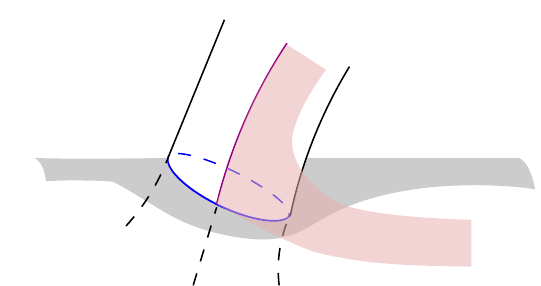
	\caption{$l_k\in C\cap C_k$ has a leg on the puncture $l_{k'}$ of $C_k$. This puncture is a generating curve on a $C_{k'}$.  $l'_{k'}\in C\cap C_{k'}$  intersects $l_{k'}$.}
	\label{figl24}
	\end{center}
\end{figure}

\begin{remark}\label{r-a-l24}

If $\mathscr{C}$ is  not reduced, but $C\cap C_i$ is (almost) nice and closed curves in $C\cap C_i$ bound  disks in $C$ and  punctured disks in $C_i$ for $i=1,\ldots,n$, the above proof shows that there is an equivalent simplifier (not necessarily reduced) that satisfies the properties stated in Lemma \ref{l24}.
\end{remark}

\subsection{Removing the intersections with both legs on \texorpdfstring{$\boldsymbol{\partial^+}$}{Lg}}	
\begin{lemma}\label{l25}
If   $C\cap C_i$ is almost nice, for $1\leq i\leq n$, there is an equivalent reduced simplifier $\mathscr{C}'=\{C'_i\}_{i=1}^n$  such that $C\cap C_i'\subset C\cap C_i$, $g_i({\mathscr{C}'})\leq g_i(\mathscr{C})$, and each curve in $C\cap\mathscr{C}'$ has nonempty boundary which is not a subset of $\partial^+$. In particular, if $C\cap C_i$ is nice then $C\cap C_i'$ is nice as well.
\end{lemma}
\begin{proof}
By Lemma~\ref{l24}, we may assume that all the curves in $C\cap\mathscr{C}$ have nonempty boundary. Suppose that $l\in C\cap C_i$ has both legs on $\partial^+$. There is a disk $D_l\subset C$ such that $\partial D_l=l\cup l'$, with $l'\subset l^+=\partial^+C$.  Choose $l$ such that $D_l$ does not contain another curve in $C\cap\mathscr{C}$ with both legs on $\partial^+$. Since $C\cap C_i$ is (almost) nice, $l$ is disjoint from the generating curves on $C_i$. Thus,  no  curve in $C\cap\mathscr{C}$ has a leg on $l$. \\
		
	By the choice of $D_l$, if $l_k\in C\cap C_k$ and $l_k^\circ\subset D_l^\circ$,   at least one leg of $l_k$ is on a puncture $l_{k'}$ of $C_k$, which  is a generating curve on a $C_{k'}$ (with $k'>k$). Thus, there is a curve $l'_{k'}\in C\cap C_{k'}$ such that $l'_{k'}\cap l_{k'}\neq\emptyset$ (Figure \ref{figl24}). Since the intersection is (almost) nice, $l'_{k'}$ has a leg on $\partial^-$.  This is not possible since $(l'_{k'})^\circ\subset D_l^\circ$. Therefore, $D_l^\circ\cap\mathscr{C}=\emptyset$. There is a punctured disk $D_l'\subset C_i$ such that $\partial D_l'=l\cup l''$, where $l''\subset l_i^+=\partial^+ C_i$. By Lemma \ref{l12}, replacing $D_l$ with $D_l'$, we obtain an equivalent reduced simplifier $\mathscr{C}'$ such that $C\cap\mathscr{C}'$ has fewer curves (in comparison with $C\cap\mathscr{C}$) with both legs on $\partial^+$. Furthermore, $C\cap\mathscr{C}'$ is of type $\textrm{II}$ and $C\cap C_i'$ is (almost) nice. Repeating this process, we obtain an equivalent reduced simplifier with the desired properties.
\end{proof}
	
\begin{remark}\label{r-a-l25}
	If $\mathscr{C}$  is not reduced, but $C\cap C_i$ is (almost) nice and does not contain any closed curves for $i=1,\ldots,n$, the above argument implies that there is an equivalent simplifier (not necessarily reduced) which satisfies the properties stated in Lemma \ref{l25}.
\end{remark}
	
	Let $l\in C\cap C_i$ be  a SBC and $l'$ be a curve component in $C\cap C_i$, such that $\partial l'\cap l\neq\emptyset$. The curve $l'$ is then of one of the following types (Figure \ref{SBC}):
	\begin{align*}
	\textsf{i}.&\  \partial l'\subset l,&
	\textsf{ii}.&\  \partial l'\cap l^+\neq\emptyset,&
	\textsf{iii}.&\  \partial l'\cap l^-\neq\emptyset&&\text{and}&
	\textsf{iv}.&\  \partial l'\cap(C^\circ\setminus l)\neq\emptyset.
	\end{align*}
Next, we refine a reduced simplifier with properties stated in  Lemma~\ref{l25}, to construct an equivalent simplifier $\mathscr{C}'$  so that $C\cap \Cscr'$ does not contain curves of types $\textsf{i}$ and $\textsf{ii}$. 

	\begin{figure}
	\def\svgwidth{3.5cm}
		\begin{center}
			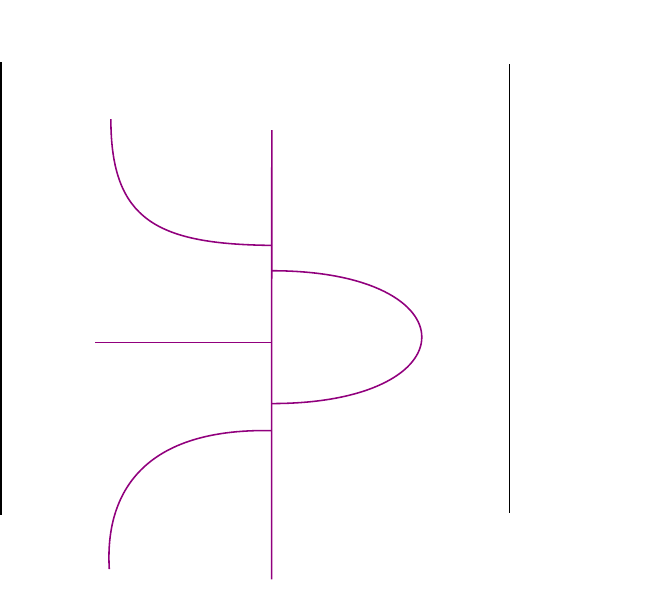
	\caption{$l\in C\cap\mathscr{C}$ is a SBC. The four types of intersection of curves in $C\cap\mathscr{C}$ with at least one leg on $l$ are illustrated.}
	\label{SBC}
	\end{center}
	\end{figure}

	\subsection{Removing  a special intersection of type \texorpdfstring{$\textsf{i}$}{Lg}}\label{3.1}
	Suppose that $C\cap C_i$ is (almost) nice for $i=1,\ldots,n$, and does not contain closed curves or curves with both legs on $\partial^+$. Let $l\in C\cap C_i$ be a SBC and $l_k\in C\cap C_k$ be such that $\partial l_k\subset l$. There is a disk $D_{k}\subset C$ such that $\partial D_{k}=l_k\cup\bar{l}$, where $\bar{l}\subset l$. Choose $l_k$ such that $D_{k}^\circ$ does not contain intersections of type $\textsf{i}$. $l_k\subset C_k$ has legs on the punctures of $C_k$. Thus, $l_k$ does not intersect the generating curves on $C_k$. So, there is no  curve in $C\cap \mathscr{C}$ with legs on $l_k$.\\

	On the other hand, $D_{k}^\circ\cap\mathscr{C}=\emptyset$. In fact, if there is some $l_j\in C\cap C_j$ with $l_j^\circ\subset D_{k}^\circ$, then $\partial l_j\neq\emptyset$ and by the choice of  $l_k$, we have $\partial l_j\nsubseteq l$. Thus, at least one leg of $l_j$ is in $D_{k}^\circ$ and on a puncture $l_{j'}$ of $C_j$. $l_{j'}$ is a generating curve on a punctured cylinder $C_{j'}$, with $j'>j$. Therefore, there is a second curve $l'_{j'}\in C\cap C_{j'}$ which intersects $l_{j'}\subset C_{j'}$. So, $\partial l'_{j'}\cap\partial^-\neq\emptyset$. This cannot happen since $(l'_{j'})^\circ\subset(D_k)^\circ$.\\
	
	Since $l$ intersects each generating curve on $C_i$ exactly once, the two legs of $l_k$ are on  different punctures $l_{k}^1$ and $l_{k}^2$ of $C_k$, which are generating curves on $C_i$. Let $\bar{C}\subset C_i$ be a subcylinder with $\partial\bar{C}=l_{k}^1\cup l_{k}^2$. Since the intersection of $C$ with $\Cscr$ is almost nice, $\bar{C}$ has no punctures. However, it may include some generating curves, which are the intersections with cylinders $C_s$ with $s<i$, as illustrated in Figure~\ref{special1-2}, where a neighborhood of $D_{k}\cup\bar{C}$ is pictured. \\
	
	\begin{figure}
	\def\svgwidth{16cm}
		\begin{center}
			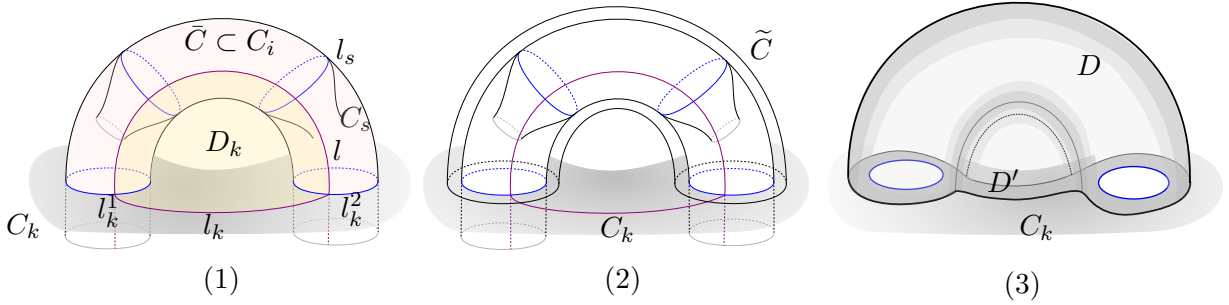
	\caption{The legs of $l_k\in C\cap C_k$ are on $l_{k}^1\cap l$ and $l_{k}^2\cap l$, where $l_{k}^1$ and $l_{k}^2$ are punctures on $C_k$ and generating curves on $C_i$. The boundary of the subcylinder $\bar{C}\subset C_i$ is  $l_{k}^1\amalg l_{k}^2$ (1).  $\widetilde{C}$ is a parallel copy of $\bar{C}$ (2). $D$ is obtained by cutting $\widetilde{C}$ along  $\widetilde{C}\cap D_{k}$, and attaching parallel copies of $D_{k}$ to the cut edges (3). $\partial D$ bounds the punctured disk $D'\subset C_k$.}
	\label{special1-2}
	\end{center}
	\end{figure}

	Let $\widetilde{C}$ be a parallel copy of $\bar{C}$ such that $\widetilde{C}^\circ\cap \mathscr{C}=\emptyset$ and $\partial \widetilde{C}\subset C_k^\circ$ (Figure \ref{special1-2}-(2)). By cutting  $\widetilde{C}$ along its intersection with $D_{k}$, and gluing two parallel copies of $D_{k}$ along the cut edges, we obtain a disk $D$ with $D^\circ\cap\mathscr{C}=\emptyset$. $\partial D\subset C_k$ bounds a punctured disk $D'\subset C_k$, which contains  the two punctures $l_{k}^1$ and $l_{k}^2$ (Figure \ref{special1-2}-(3)). Furthermore, $\partial D$ is disjoint from the generating curves on $C_k$.  Lemma \ref{l10} implies that by replacing $D'$ with $D$ we obtain an equivalent reduced simplifier $\mathscr{C}'=(C_1,\ldots,C_{k-1},C_k',C_{k+1},\ldots,C_n)$, where  two of the punctures of $C_k$ are removed in $C_k'$. These punctures are generating curves on $C_i$. Thus, $g_i({\mathscr{C}'})<g_i(\mathscr{C})$. \\

Moreover, $C\cap C_k'$ is almost nice. To see this, first note that for $l'\in C\cap C_k$ with a leg on $l_{k}^1$ (or $l_{k}^2$), there is a curve $\bar{l}\in C\cap C_i$ such that a leg of $l'$ is on $\bar{l}$. $\bar{l}\cap\bar{C}$ consists of curves with legs on $l_{k}^1$ or $l_{k}^2$. Thus, there is a curve $\widetilde{l}$ in $\bar{l}\cap\bar{C}$ with one leg given by $l'\cap \bar{l}$. The other leg of $\widetilde{l}$ corresponds to a curve $l''\in C\cap C_k$ (possibly with $l'=l''$).  In Figure \ref{special1-4}, if $l'=l'_1$, then we have $l'=l''$, while for $l'=l'_2$ we have $l'\neq l''$. The boundaries of $l'$ and $l''$, which are on $l_{k}^1$ or $l_{k}^2$, are identified in $C'_k$, giving a new intersection curve in $C\cap C'_k$. Let $\bar{l}\in C\cap C'_k$ intersect the generating curves on $C'_k$. If $\bar{l}$ is in $C\cap C_k$,  $\bar{l}$ has a leg on $\partial^-$. If $\bar{l}$ has exactly one leg on $\partial^-$, it intersects each generating curve exactly once. On the other hand, if $\bar{l}\notin C\cap C_k$,  there are curves $l',l''\in C\cap C_k$ such that $\bar{l}$ is obtained by identifying a leg of $l'$ with a leg of $l''$, as described above. Thus, $l'$ or $l''$ should intersect the generating curves. Since $C\cap C_k$ is almost nice, $l'$ or $l''$ has a leg on $\partial^-$. So, $\bar{l}$ also has a leg on $\partial^-$. If $\bar{l}$ has only one leg on $\partial^-$, then only one of $l'$ or $l''$, say $l'$, intersects the generating curves. By definition, $l'$ intersects all the generating curves in exactly one point. Thus, $\bar{l}$ intersects  all the generating curves in exactly one point. Since the punctures of $C'_k$ are a subset of the punctures of $C_k$,  it follows that $C\cap C_k'$ is almost nice.
		
\begin{figure}
	\def\svgwidth{14cm}
		\begin{center}
			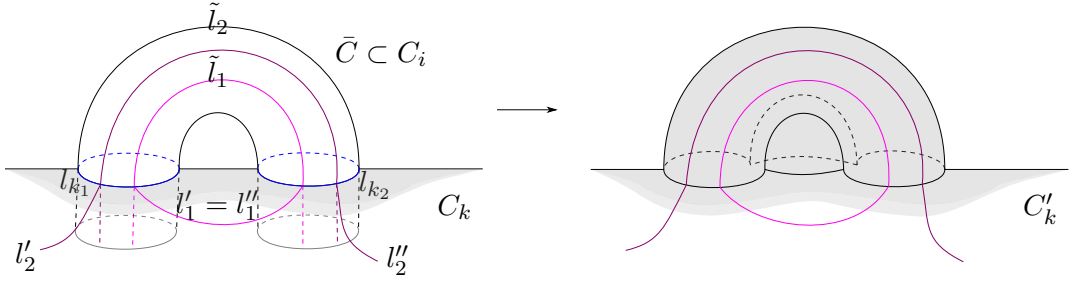
	\caption{$l'_i,l''_i\in C\cap C_k$ have a leg on $\tilde{l}_i\in C\cap \bar{C}$ for $i=1,2$. $l'_1=l''_1$ and $l'_2\neq l''_2$ (left).  $l'_1$ gives a closed curve in $C\cap C'_k$ and $l'_2$ and $l''_2$  give one curve in $C\cap C'_k$ (right).
	}
	\label{special1-4}
	\end{center}
\end{figure}

	\subsection{Removing a special intersection of type \texorpdfstring{$\textsf{ii}$}{Lg}}\label{3.2}
	Suppose that $C\cap C_i$ is (almost) nice and does not contain closed curves or curves with both legs on $\partial^+$ for $i=1,\ldots, n$. Let $l\in C\cap C_i$ be a SBC and $l_k\in C\cap C_k$ have one leg on $l$ and one on $\partial^+$. Then there is a disk $D_{k}\subset C$ such that $\partial D_{k}=l_k\cup\bar{l}\cup\bar{l}^+$, where $\bar{l}\subset l$ and $\bar{l}^+\subset l^+$, see Figure \ref{Special-C}-(1). Suppose that $D_k^\circ$ does not include curves of type $\textsf{i}$ or $\textsf{ii}$. The curve $l_k\subset C_k$ has one leg on $l_k^+$ and one on a puncture of $C_k$. Thus, $l_k$ does not intersect the generating curves on $C_k$ and there is no curve in  $C\cap\mathscr{C}$ with legs  on $l_k$. On the other hand, $D_{k}^\circ\cap C=\emptyset$. In fact, if $l_j\in C\cap C_j$ and $l_j^\circ\subset D_{k}^\circ$, then $l_j$ has nonempty boundary, and by our assumption on $l_k$, at least one leg of $l_j$ is on a puncture $l'_j$ of $C_j$ and in $D_{k}^\circ$. The curve $l'_j$ is a generating curve on some $C_{j'}$, with $j'>j$. So, there is a curve $l'_{j'}\in C\cap C_{j'}$ which intersects  $l'_{j}$ on $C_{j'}$. Therefore, one leg of $l'_{j'}$ is on $\partial^-$. This cannot happen since $l'_{j'}\subset D_{k}^\circ$.\\

	\begin{figure}
	\def\svgwidth{15cm}
		\begin{center}
			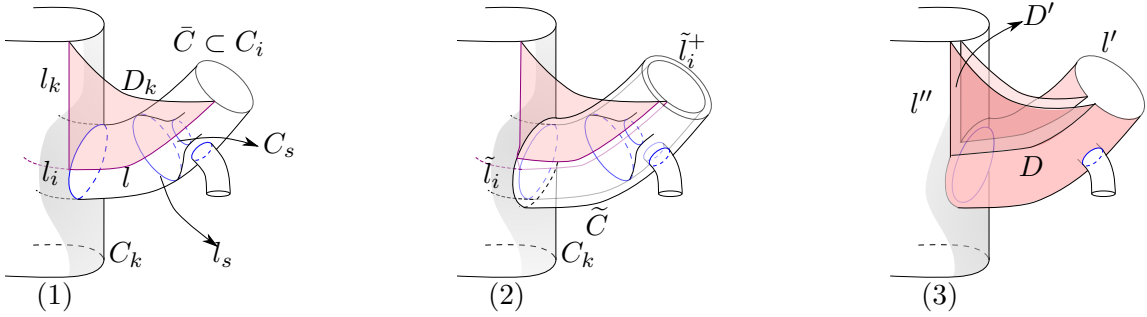
	\caption{The boundary of $l_k\in C\cap C_k$ is on $\partial^+$ and  $l_{i}\cap l$, with $l_{i}$ a puncture on $C_k$ and a generating curve on $C_i$. The boundary of the subcylinder $\bar{C}\subset C_i$ is $l_{i}\amalg l_i^+$ (1).  $\widetilde{C}$  is a parallel copy of $\bar{C}$ (2). $D$ is obtained by cutting $\widetilde{C}$ along $\widetilde{C}\cap D_k$ and gluing parallel copies of  $D_{k}$ to the cut edges (3). $\partial D=l'\cup l''$ where  $l'\subset\partial^+$ and $l''\subset C_k$.}
	\label{Special-C}
	\end{center}
	\end{figure}	
	
	The curve $l_k\subset C_k$ has one leg on $l_k^+$ and one on a puncture $l_i$ of $C_k$, where $l_i$ is a generating curve on $C_i$. Let $\bar{C}\subset C_i$ be a punctured subcylinder with $\partial \bar{C}=l_i^+\amalg l_i$.  Let $\widetilde{C}$ denote a parallel copy of $\bar{C}$  such that $\partial\widetilde{C}=\tilde{l}_i^+\cup\tilde{l}_i$, where $\tilde{l}_i^+\subset\partial^+$  is disjoint from $l_j^+$, for $j=1,\ldots, n$, and $\tilde{l}_i\subset C_k$ is disjoint from $l_i$ and the generating curves on $C_k$. The punctures of $\widetilde{C}$ are essential curves on $C_j$, for $j>i$, while $\widetilde{C}^\circ\cap\mathscr{C}=\emptyset$ (Figure \ref{Special-C}-(2)).\\

	By cutting $\widetilde{C}$ along its intersection with $D_k$ and gluing parallel copies of  $D_{k}$ to the cut edges, a punctured disk $D$ is obtained with $\partial D=l'\cup l''$, where $l'\subset\partial^+$, $l''\subset C_k$ and $D^\circ\cap\mathscr{C}=\emptyset$. The punctures of $D$ are essential curves on $C_j$, for $j>i$, and $l''$ is disjoint from the generating curves on $C_k$. There is a punctured disk $D'\subset C_k$ with  a single puncture $l_i$ and  $\partial D'=l''\cup l^*$, where $l^*\subset l_k^+$ (Figure \ref{Special-C}-(3)). By Lemma \ref{l12}, if we replace $D'$ with $D$, we obtain a punctured cylinder $C_k'$ and an equivalent reduced simplifier $\mathscr{C}'=\left(C_1,\ldots,C_{k-1},C_k',C_{k+1},\ldots,C_n\right)$, with  $g_i({\mathscr{C}'})=g_i(\mathscr{C})-1$ and   $g_j({\mathscr{C}'})=g_j(\mathscr{C}_j)$ for $j<i$. Since the new generating curves on $C_j$ may be chosen sufficiently close to (some) old generating curves for $j>i$, if $C\cap C_j$ is nice, so is $C\cap C_j'$.\\

Moreover, $C\cap C'_k$ is almost nice. To see this, first note that for each curve $l'\in C\cap C_k$ with a leg on $l_i$, there is a curve $\tilde{l}\in C\cap C_i$ such that a leg of $l'$ is on $\tilde{l}$ (Figure~\ref{fig-Special-D}-(1)). The intersection $\tilde{l}\cap\bar{C}$ consists of curves with legs on $l_i$, or $l_i^+$, or on a puncture of $\bar{C}$. Thus, there is a curve $\tilde{l}_1$ in $\tilde{l}\cap\bar{C}$ with one leg given by $l'\cap \tilde{l}$. If the other leg of $\tilde{l}_1$ is on $l_i^+$ or on a puncture of $\bar{C}$, there is a corresponding curve in $C\cap C'_k$,  denoted by $l^*$, such that $l^*$ has a leg on $\partial^+$ or on a puncture of $C'_k$ (Figure~\ref{fig-Special-D}-(2)). Then $l'\in C\cap C_k$ intersects the generating curves iff $l^*\in C\cap C'_k$  intersects the generating curves. Suppose that $l^*\in C\cap C'_k$ intersects the generating curves on $C'_k$. Since $C\cap C_k$ is (almost) nice, $l'$ has a leg on $\partial^-$. If $l'$ has exactly one leg on $\partial^-$, it intersects each generating curve on $C_k$ exactly once. As a result $l^*$  intersects each generating curve on $C_k'$ exactly once, and the claim follows.\\

If the other leg of $\tilde{l}_1$ is on $l_i$, there is a curve $l''\in C\cap C_k$ (possibly with $l''=l'$) which shares this leg (Figure \ref{fig-Special-D}-(3)). The boundaries of $l'$ and $l''$  on $l_i$ are identified in $C'_k$, giving a curve $l^*$ in $C\cap C'_k$ (Figure~\ref{fig-Special-D}-(4)). If $l^*$ intersects the generating curves on $C'_k$,  $l'$ or $l''$, say $l'$,  intersects the generating curves on $C_k$. Since $C\cap C_k$ is (almost) nice, $l'\subset C_k$ has one leg on $\partial^-$. Thus, so does $l^*$. If $l^*$ has exactly one leg on $\partial^-$, the same is true for $l'$. Since $C\cap C_k$ is (almost) nice,   $l'$  intersects the generating curves on $C_k$ exactly once. Hence, so does $l^*$. This completes the proof of the  claim.

\begin{figure}
	\def\svgwidth{14cm}
	\begin{center}
	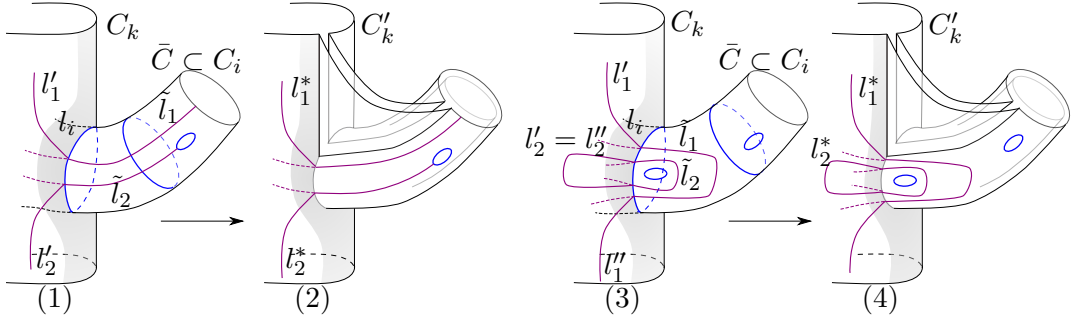
	\caption{ $l'_i, l''_i\in C\cap C_k$  have boundary on $\tilde{l}_i\cap l_i$, for $i=1,2$. $\tilde{l}_1$ and $\tilde{l}_2$ have one leg  on $\partial^+$ and the other one on a puncture of $\bar{C}$, respectively (1).  $l'_1$ (resp., $l'_2$) gives a curve $l_1^*$ (resp., $l_2^*$) in $C\cap C'_k$ with one leg on $\partial^+$ (resp., on a puncture of $C'_k$) (2). $l'_1\neq l''_1$, while $l'_2=l''_2$ (3).  $l'_1$ and $l''_1$ give a curve $l^*_1$ in $C\cap C'_k$ and $l'_2$ gives a closed curve $l^*_2$ in $C\cap C'_k$ (4). }
	\label{fig-Special-D}
	\end{center}
\end{figure}

\subsubsection{Reducing the number of generating curves}\label{3.3}
Suppose that $\mathscr{C}=\left(C_i\right)_{i=1}^n$ is a simplifier such that the following are satisfied: 
\begin{itemize}[leftmargin=.6in]\setlength\itemsep{-0.3em}
\item[$B_\mathscr{C}(i).1$] $C\cap C_j$ is almost nice for $1\leq j<i$ and is nice for $i\leq j\leq n$ ;
\item[$B_\mathscr{C}(i).2$]\label{z3} No curve in $C\cap\mathscr{C}$ is closed or has both legs on $\partial^+$;
\item[$B_\mathscr{C}(i).3$]\label{z4}If $l\in C\cap C_j$ is a SBC and $l'\in C\cap C_t$ is of type $\textsf{i}$ or $\textsf{ii}$ with $j,t+1>i$, then  $\partial l'\cap l=\emptyset$;
\item[$B_\mathscr{C}(i).4$]\label{z5}If a curve in $C\cap C_j$, for $j\geq i$, intersects the generating curves on $C_j$,  it is a SBC.
\end{itemize}
	
	Let $m_k(\mathscr{C})$ denote the number of curves of type $\textsf{i}$ or $\textsf{ii}$ with a leg on SBCs in $C\cap C_k$ and $h\geq i$ be the smallest index with  $m_h(\mathscr{C})\neq 0$. Let $l\in C\cap C_h$ be a SBC and $l_k\in C\cap C_k$ be of type $\textsf{i}$ or $\textsf{ii}$ such that $\partial l_k\cap l\neq\emptyset$ and $D_{k}^\circ\subset C$ includes no intersection curves of type $\textsf{i}$ or $\textsf{ii}$. Since $h\geq i$, then $k<i$ by definition and $B_\mathscr{C}(i).3$. By the discussion of $\S$\ref{3.1} and $\S$\ref{3.2}, we obtain an equivalent reduced simplifier $\Cscr'=\{C_i'\}_{i=1}^n$ which satisfies $B_{\Cscr'}(i).1$, while $g_j({\Cscr'})=g_j({\mathscr{C}})$ for $i\leq j<h$, and $g_h({\Cscr'})<g_h({\mathscr{C}})$. \\
	
We next show that $\Cscr'$ satisfies $B_{\Cscr'}(i).3$.  For a SBC $\tilde{l}\in C\cap C_j'$ with $j>i$, and an intersection $l''\in C\cap  C_t'$ of type $\textsf{i}$ or $\textsf{ii}$ with  $\partial l''\cap l\neq\emptyset$ we have $C\cap C_j=C\cap  C_j'$. For $t=k$, we already observed that $t=k<i$. If $t\neq k$, $C\cap C_t'=C\cap C_t$ (since $ C_t'$ is the same as $C_t$ possibly with a different set of generating curves).  So, $l''\in C\cap C_t$ and by $B_{\mathscr{C}}(i).3$, $t<i$.\\
	
Finally, we show that $B_{\Cscr'}(i).4$ is satisfied. Let $l\in C\cap  C_j'$, intersect the generating curves on $ C_j'$ for some $j\geq i$. If  $C_j'=C_j$, then $l$ is a SBC, since $\mathscr{C}$ satisfies $B_\mathscr{C}(i).4$. On the other hand, if $C_j'$  is obtained from $C_j$ by removing some generating curves, then $l\in C\cap C_j$ also intersects the generating curves on $C_j$ (note that $C\cap C_j=C\cap  C_j'$). Since $\mathscr{C}$ satisfies $B_\mathscr{C}(i).4$, it follows again that $l$ is a SBC. Finally, if $C_j'$ is obtained from $C_j$ by adding some generating curves, then these generating curves are arbitrarily close to the ones on $C_j$. Thus, $l$ also intersects the generating curves on $C_j$. Since $\mathscr{C}$ satisfies $B_\mathscr{C}(i).4$, it follows again that $l$ is a SBC.\\
	
$C\cap\Cscr'$ may include closed curves or curves with both legs on $\partial^+$ (Figure \ref{fig-special2-4}).  Lemma \ref{l25} gives an equivalent reduced simplifier  $\mathscr{C}''=\{C_j''\}_{j=1}^n$ such that $B_{\mathscr{C}''}(i).1$ and $B_{\mathscr{C}''}(i).2$ are satisfied,  $g_j({\mathscr{C}''})\leq g_j(\mathscr{C})$ for $i\leq j< h$, and $g_h({\mathscr{C}''})<g_h(\mathscr{C})$. Moreover, $B_{\mathscr{C}''}(i).3$ is satisfied;  if $\tilde{l}\in C\cap C_j''$ is a SBC for $j>i$ and $l''\in C\cap C_t''$ is a curve of type $\textsf{i}$ or $\textsf{ii}$ with a leg on $\tilde{l}$, $\tilde{l}\in C\cap C_j'$ and $l''\in C\cap C_t'$ by  Lemma \ref{l25}. Since  $B_{\Cscr'}(i).3$ is satisfied, we have $t<i$. Similarly, $B_{\mathscr{C}''}(i).4$ is satisfied. Note that we may have $m_k({\mathscr{C}''})\neq m_k(\mathscr{C})$ (Figure \ref{fig-special2-5}).

\begin{figure}
\def\svgwidth{15cm}
\begin{center}
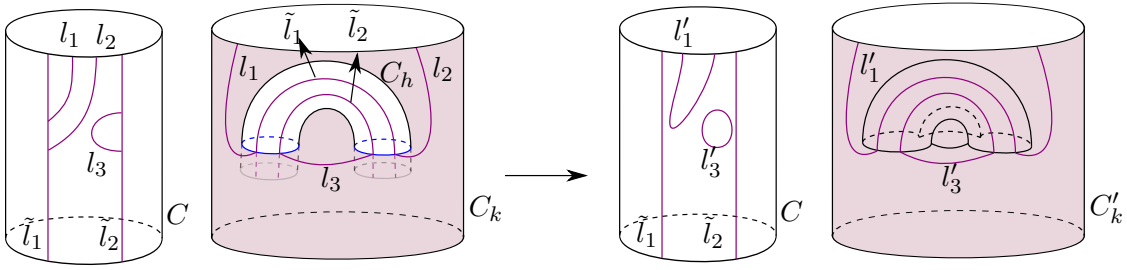
\caption{$l_1,l_2\in C\cap C_k$ are of type $\textsf{ii}$ with one leg on the SBC $\tilde{l}_1\in C\cap C_h$, while $l_3\in C\cap C_k$ is of type $\textsf{i}$ with both legs on the SBC $\tilde{l}_2\in C\cap C_h$ (left).  $l_1$ and $l_2$ give $l'_1\in C\cap  C'_k$ with both legs on $\partial^+$, while $l_3$ gives $l'_3\in C\cap C'_k$ (right). 
}
\label{fig-special2-4}
\end{center}
\end{figure}

\begin{figure}
\def\svgwidth{15cm}
\begin{center}
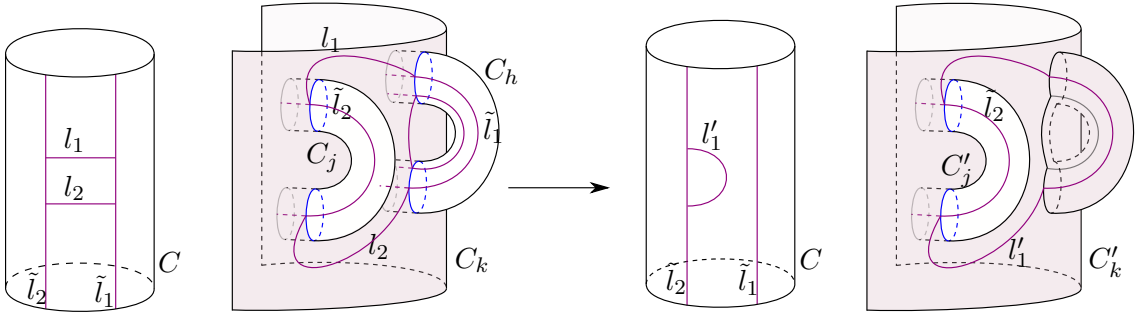
\caption{ The two legs of either of $l_1,l_2\in C\cap C_k$ is on the SBCs  $\tilde{l}_1\in C\cap C_h$ and  $\tilde{l}_2\in C\cap C_j$ (left).  $l_1$ and $l_2$ are changed to  $l'_1\in C\cap C'_k$, which is of type $\textsf{i}$ (right). }
\label{fig-special2-5}
\end{center}
\end{figure}	

\section{Removing intersections of type \texorpdfstring{$\textrm{II}$}{Lg}}
Having fixed the cylinder $C\subset N$, for a simplifier $\Cscr$,  let $B_\Cscr$ be  the following statement:\\

$\mathbf{B_\Cscr:}\quad$ {\emph{$\Cscr$ is reduced,  $C\cap \mathscr{C}$ is nice of type $\textrm{II}$ and if $l\in C\cap\mathscr{C}$, then $l$ is not closed, $\partial l\not\subset \partial^+$, $l$ is not of type $\textsf{i}$ or $\textsf{ii}$, and if $l$ intersects the generating curves, it is a SBC}}.  
\subsection{Removing type \texorpdfstring{$\textsf{i}$}{Lg} and type \texorpdfstring{$\textsf{ii}$}{Lg} intersections}
\begin{proposition}\label{p26}	
Let $N$ be a manifold with a reduced simplifier $\mathscr{C}=\left(C_i\right)_{i=1}^n$ and $C$ be a cylinder in $N$ such that $C\cap\mathscr{C}$ is nice and of type $\textrm{II}$. There is an equivalent simplifier $\mathscr{C}'=\big(C'_j\big)_{i=1}^n$ such that $B_{\Cscr'}$ is satisfied.
\end{proposition}
\begin{proof}
 Let $B_\mathscr{C}(i)$ be the statement that "{\emph{there is a reduced simplifier $\mathscr{C}'=\{C'_j\}_{j=1}^n$, equivalent to $\mathscr{C}$ and with nice intersections with $C$ such that $B_{\mathscr{C}'}(i).2,3,4$ (from $\S$\ref{3.3}}) are satisfied}". We may assume that $n>0$, and note that $B_\mathscr{C}(1)$ implies the proposition. In fact, let $l\in C\cap C_j$ be a SBC. If $l'\in C\cap C_t$ is such that $\partial l'\cap l\neq\emptyset$ and $l'$ is of type $\textsf{i}$ or $\textsf{ii}$, by $B_{\mathscr{C}}(1).3$ we have $j=1$. Since there is no generating curves on $C_1$, there is no curve $l'\in C\cap C_t$ with $\partial l'\cap l\neq\emptyset$. Therefore, $C\cap\mathscr{C}$ does not contain curves of type $\textsf{i}$ or $\textsf{ii}$. \\

We prove $B_\mathscr{C}(i)$ by reverse induction on $i$. If $i=n$, using Lemma \ref{l25} we obtain an equivalent reduced simplifier $\mathscr{C}'=\{C'_i\}_{i=1}^n$ such that $C\cap\mathscr{C}'$ is nice and  $B_{\mathscr{C}'}(n).2$ is satisfied. $B_{\mathscr{C}'}(n).3$ is vacuously true. Let $l\in C\cap C_n'$ intersect the generating curves on $C'_n$. Since $C\cap C'_n$ is nice, $l$ has exactly one leg on $\partial^-$. Since $C'_n$ has no punctures, the other leg  of $l$ is on $\partial^+$, i.e. $l$ is a SBC. Thus, $B_{\mathscr{C}'}(n).4$ is satisfied, implying $B_\mathscr{C}(n)$.\\ 

Let us assume that  $B_\mathscr{C}(i)$ is satisfied and $\mathscr{C}'$ be the corresponding equivalent simplifier. We conclude $B_{\mathscr{C}}(i-1)$ as follows. Using \S \ref{3.3}, we obtain a sequence of the simplifiers $\mathscr{C}^{(k)}$ with $\mathscr{C}^{(0)}=\Cscr'$ so that $B_{\mathscr{C}^{(k)}}(i)-1,2,3,4$ are satisfied and $g_i({\mathscr{C}^{(k)}})\geq g_i({\mathscr{C}^{(k+1)}})\geq0$, for $d\geq 0$. Choose $d_i$ such $g_i({\mathscr{C}^{(d)}})=g_i({\mathscr{C}^{(d+1)}})$ for $d\geq d_i$. It follows that $m_i({\mathscr{C}^{(d)}})=0$, for all $d\geq d_i$. In fact, if $m_i({\mathscr{C}^{(d)}})\neq 0$ for some $d\geq d_i$, then the smallest index $h\geq i$ with $m_h({\mathscr{C}^{(d)}})\neq 0$ is $i$. Thus by \S \ref{3.3}, $g_i({\mathscr{C}^{(d+1)}})<g_i({\mathscr{C}^{(d)}})$, which is not possible. Therefore for $d\geq d_i$,  the smallest index $h\geq i$ with $m_h({\mathscr{C}^{(d)}})\neq0$, is at least $i+1$, $g_{i+1}({\mathscr{C}^{(d)}})\geq g_{i+1}({\mathscr{C}^{(d+1)}})$, and there is some $d_{i+1}$ so that $g_{i+1}({\mathscr{C}^{(d)}})= g_{i+1}(\mathscr{C}^{(d+1)})$ for $d\geq d_{i+1}$. \\

Similarly, we can show that  $m_{i+1}({\mathscr{C}^{(d)}})=0$ for all $d\geq d_{i+1}$. Repeating this process, we obtain a sequence $d_i\leq d_{i+1}\leq\dots\leq d_n$ such that $g_j(\mathscr{C}^{(d)})=g_j(\mathscr{C}^{(d_j)})$ and  $m_j({\mathscr{C}^{(d)}})=0$ for $j=i,\ldots,n$ and $d\geq d_j$. As a result, $m_j({\mathscr{C}^{(d)}})=0$ for all $d\geq d_n$ and $j\geq i$. Then $B_{\mathscr{C}^{(d_n)}}(i).1,2,3,4$ are satisfied. Lemma \ref{l17} gives an equivalent reduced simplifier $\Cscr''=(C_1'',\ldots,C_n'')$, such that for $j>i-1$, $C_j''$ is the same as $C_j^{(d_n)}$ (i.e. the $j$th punctured cylinder in $\Cscr^{(d_n)}$), possibly with more generating curves which are arbitrarily close to the old ones. Furthermore, $C\cap \Cscr''$ is nice. Note that by construction $C\cap C_j^{(d_n)}=C\cap C_j''$ for $j\geq i$. \\

To prove $B_{\Cscr''}(i-1).3$, let $l\in C\cap C_j''$ be a SBC with $j\geq i$ and $l'\in C\cap \tilde{C}_k'$ with $\partial l'\cap l\neq\emptyset$ be of type $\textsf{i}$ or $\textsf{ii}$. The boundary of $l'$, away from $\partial^+$, is on the punctures of $C_k''$, which are generating curves on $C_j''$. Such generating curves do not exist in $C_j^{d_n}$ (since $m_j({\mathscr{C}^{(d_n)}})=0$). Thus, $k<i-1$. \\

Finally, we prove $B_{\Cscr''}(i-1).4$. For $j\geq i-1$, let $l\in C\cap C_j''$ intersect the generating curves on $C_j''$. If $j\geq i$, from $C\cap C_j^{(d_n)}=C\cap C_j''$, it follows that $l\in C\cap C_j^{(d_n)}$. Since the generating curves on $C_j''$ are the same as, or close to, the generating curves on $C_j^{(d_n)}$, the curve $l\in C\cap C_j^{(d_n)}$ also intersects the generating curves on $C_j^{(d_n)}$. By  $B_{\mathscr{C}^{(d_n)}}(i).4$, $l$ is a SBC. We may thus assume that $j=i-1$. If $l$ is not a SBC,  since $C\cap C_{i-1}''$ is nice and $l$ intersects the generating curves on $C_{i-1}''$, it follows that $l$ has exactly one leg on $\partial^-$. The other leg of $l$ is on a puncture of $C_{i-1}''$. Then, there is a curve $l'\in C\cap C_{i-1}''$ with one leg on $\partial^+$ and one on a puncture $l_t$ of $C_{i-1}''$. $l_t$ is a generating curve on a $C_t''$, for some $t>i-1$. Thus, there is a curve $l'_t\in C\cap C_t''$ that intersects $l_t$. Since $t>i-1$, $l'_t$ is a SBC. However, $\partial l'\cap l'_t\neq\emptyset$, the curve $l'$ is of type $\textsf{ii}$, and $l'\in C\cap C_{i-1}''$. This contradicts $B_{\Cscr''}(i-1).3$, and proves our claim.\\

Applying  Lemma \ref{l25}, we obtain an equivalent reduced simplifier $\Cscr^*$ such that $B_{\Cscr^*}(i-1).2$ is satisfied. A similar argument as in $\S$\ref{3.3} proves that  $B_{\Cscr^*}(i-1).3,4$ are also satisfied. This gives $B_\mathscr{C}(i-1)$ and completes the proof of the (reverse) induction.
\end{proof}

\subsection{Intersections of type \texorpdfstring{$\textsf{iii}$}{Lg}}\label{3.5}
\begin{lemma}\label{l27}
Let $B_\Cscr$ be satisfied for $\mathscr{C}=\left(C_i\right)_{i=1}^n$. Then there is an equivalent simplifier $\mathscr{C}'=\left(C'_i\right)_{i=1}^n$  such that $B_{\Cscr'}$ is satisfied, while  $C\cap \Cscr'$ does not contain curves with both legs on $\partial^-$. Moreover,  $g_i({\mathscr{C}'})\leq g_i(\mathscr{C})$ and $C\cap C'_i\subset C\cap C_i$ for all $1\leq i\leq n$.
\end{lemma}
\begin{proof}
	Let $l\in C\cap C_i$ be such that $\partial l\subset\partial^-$. There is a disk $D^-\subset C$ with $\partial D^-=l\cup l'$, where $ l'\subset l^-$. Choose $l$ such that $(D^-)^\circ\cap\mathscr{C}$ does not contain curves with both legs on $\partial^-$. Since $C\cap\mathscr{C}$ is nice, $l$ is disjoint from the generating curves on $C_i$. Thus, there is no curve in $C\cap\mathscr{C}$ with a leg on $l$. We claim that $(D^-)^\circ\cap\mathscr{C}=\emptyset$. In fact, if $l_k\in C\cap C_k$ and $l_k^\circ\subset(D^-)^\circ$, by the choice of $D^-$ and the fact that $l_k$ has nonempty boundary, at least one leg of $l_k$ is on a puncture $l_{k'}$ of $C_k$, which  is a generating curve on  $C_{k'}$, for some $k'>k$. Thus, there is a curve $l'_{k'}\in C\cap C_{k'}$ that intersects the generating curve $l_{k'}\subset C_{k'}$. Hence, $l'_{k'}$ is a SBC. In particular, $l'_{k'}$ has a leg on $\partial^+$. Since $(l'_{k'})^\circ\subset(D^-)^\circ$, this cannot happen.\\

	There is a disk $\bar{D}\subset C_i$ such that $\partial \bar{D}=l\cup l''$, with $ l''\subset l_i^-$. Using Lemma \ref{l10}, if we replace $\bar{D}$ with $D^-$, we obtain an equivalent reduced simplifier $\mathscr{C}''=\big(C_i''\big)_{i=1}^n$ such that $C\cap\mathscr{C}''$ has fewer curves with both legs on $\partial^-$. It is clear that  $g_i(\Cscr'')\leq g_i(\Cscr)$ and $C\cap C_i''\subset C\cap C_i$. This latter fact implies that  $B_{\Cscr''}$ is satisfied.  By repeating this process, we obtain an equivalent reduced simplifier $\mathscr{C}'=\left(C_i'\right)_{i=1}^n$, with the desired property.
	\end{proof}

	Suppose that $\mathscr{C}=\left(C_i\right)_{i=1}^n$ satisfies $B_\Cscr$ and $C\cap\Cscr$ does not contain curves with boundary on $\partial^-$.  Let $n_i(\mathscr{C})$ denote the number of type $\textsf{iii}$ intersections which have a leg on a SBC in $C\cap C_i$, for $1\leq i\leq n$. Let $l\in C\cap C_i$ be a SBC and $l_k\in C\cap C_k$ have one leg on $l$ and one on $\partial^-$. There is a disk $D_{k}\subset C$ with $\partial D_{k}= l_k\cup \bar{l}^-\cup \bar{l}$, where $\bar{l}^-\subset l^-$ and $\bar{l}\subset l$. Suppose that $D_{k}^\circ$ does not contain any intersections of type $\textsf{iii}$. The curve $l_k\subset C_k$ has nonempty boundary with one leg on $l_k^-$ and one on a puncture $l_i$ of $C_k$, where $l_i$ is a generating curve on $C_i$. Thus, $l_k$ does not intersect the generating curves on $C_k$. So there are no curves in $C\cap\mathscr{C}$ with legs on $l_k$. We claim that $D_{k}^\circ\cap \mathscr{C}=\emptyset$.  In fact, let us assume that $l_j\in C\cap C_j$ and $l_j^\circ\subset(D_{k})^\circ$. Then $l_j$ has nonempty boundary and by the choice of $l_k$, at least one leg of $l_j$ is in $(D_{k})^\circ$, so it is on a puncture $l_{j'}$ of $C_j$, which  is a generating curve on a $C_{j'}$ with $j'>j$. Thus, there is a curve $l'_{j'}\in C\cap C_{j'}$ that intersects the generating curve $l_{j'}\subset C_{j'}$. Therefore, $l'_{j'}$ is a SBC  and has a leg on $\partial^+$. Since $l'_{j'}\subset D_{k}^\circ$, this cannot happen. \\

	\begin{figure}
\def\svgwidth{15cm}
\begin{center}
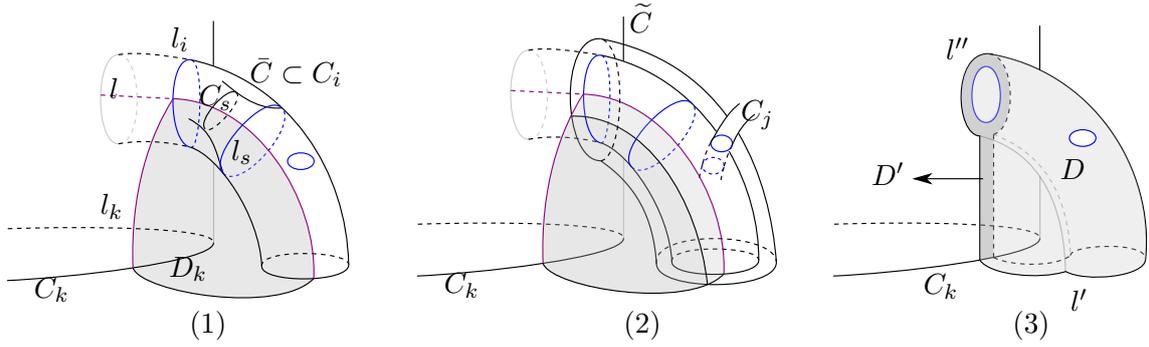
\caption{$l_k\in C\cap C_k$ has one leg on $\partial^-$ and one leg  on $l_{i}\cap l$, with $l_{i}$  a puncture on $C_k$ and a generating curve on $C_i$. The boundary of the subcylinder $\bar{C}\subset C_i$ is $l_{i}\amalg l_i^-$ (1).  $\widetilde{C}$ is a parallel copy of $\bar{C}$. Its punctures are essential curves, disjoint from the generating curves on $C_j$, for $j>i$ (2). $D$ is obtained by cutting $\widetilde{C}$ along  $\widetilde{C}\cap D_{k}$ and gluing parallel copies of $D_{k}$ to the cut edges (3).  $\partial D=l'\cup l''$ where  $l'\subset\partial^+$ and $l''\subset C_k$. }
\label{fig-type-3-3}
\end{center}
\end{figure}
	
	Let $\bar{C}\subset C_i$ be a punctured subcylinder with \ boundary \ $l_i^-\amalg l_i$. Let $l_s$ be a generating curve in $(\bar{C})^\circ$ which is a puncture on a punctured cylinder $C_s$ with $s<i$. A neighborhood of $l_s$ and $D_{k}\cup \bar{C}$ is illustrated in Figure \ref{fig-type-3-3}-(1). Let $\widetilde{C}$ denote a parallel copy of $\bar{C}$ such that $\widetilde{C}^\circ\cap\mathscr{C}=\emptyset$, one boundary component of $\widetilde{C}$ is on $C_k$, disjoint from $l_i$ and the generating curves on $C_k$, and the other boundary component of $\widetilde{C}$ is on $\partial^-$, disjoint from $l_j^-$ for $1\leq j\leq n$ (Figure \ref{fig-type-3-3}-(2)).\\

The punctures of $\widetilde{C}$ are essential curves on $C_j$ and disjoint from the generating curves, for $j>i$. By cutting $\widetilde{C}$ along  $\widetilde{C}\cap D_{k}$ and gluing two parallel copies of $ D_{k}$ to the cut edges, we obtain a punctured disk $D$ with $\partial D=l'\cup l''$, where $l'\subset\partial^-$ and $l''\subset C_k$, such that $(D)^\circ\cap\mathscr{C}=\emptyset$. The punctures of $D$  are essential curves on $C_j$, for $j>i$. There is  a punctured disk $D'\subset C_k$ with a single puncture $l_i$ and with  $l''\subset\partial D'$ (Figure \ref{fig-type-3-3}-(3)). Applying Lemma \ref{l10}, by replacing $D'$  with $D$ we obtain an equivalent reduced simplifier  $\mathscr{C}'=\left(C_1,\ldots,C_{k-1},C_k',C_{k+1},\ldots, C_n\right)$.
\begin{lemma}\label{new-lemma}
For $\Cscr'$ as above, $B_{\Cscr'}$ is satisfied, $g_j(\mathscr{C}')=g_j(\mathscr{C})$ for $j<i$ and $g_i(\mathscr{C}')<g_i(\mathscr{C})$. 
\end{lemma}
\begin{proof}		
We first show that  $C\cap C_k'$ is nice. Let $l'\in C\cap C_k$ have a leg on $l_i$. There is a curve $\tilde{l}\in C\cap C_i$ which  intersects  $l_i$, and is thus a SBC. So $l'$  changes to a curve $l''$ in $C'_k$ with at least one leg on $\partial^-$. Since $l'$ is not a SBC, it does not intersect the generating curves on $C_k$. Thus $l''$ is disjoint from the generating curves on $C'_k$. If $l''\in C\cap C_k'$ is a curve which is not in $C\cap C_k$ and does not intersect the generating curves, then $l''$ is either obtained as above, or  $l''\in C\cap \bar{C}$. In both cases, it follows that $l''$ is disjoint from the generating curves on $C'_k$. Therefore, if $l''\in C\cap C'_k$ intersects the generating curves on $C'_k$, $l''\in C\cap C_k$. Since the generating curves on $C_k'$ are the same as the generating curves on $C_k$, $l''$ intersects the generating curves on $C_k$ as well. Therefore, $l''$ is a SBC. This argument proves that $C\cap C'_k$ is nice. \\

It also follows from the above argument that if $l\in C\cap\Cscr'$ intersects the generating curves, it is a SBC. Moreover, it follows that $C\cap\Cscr'$ does not include closed curves or curves with both legs on $\partial^+$. The remaining non-trivial claim is that $C\cap\Cscr'$ does not include intersections of type $\textsf{i}$ or $\textsf{ii}$. Suppose otherwise, that $l\in C\cap C'_j$ is a SBC and $l'\in C\cap C'_t$, with $t<j$,  has a leg on $l$ and is of type $\textsf{i}$ or $\textsf{ii}$. If $t\neq k$, then  $C'_t=C_t$, and $l'\in C\cap C_t$. Since $l$ is a SBC, by the above argument (for showing $C\cap C_k'$ is nice), we have  $l\in C\cap C_j$. Thus, $l'$ is of type $\textsf{i}$ or $\textsf{ii}$ in $\mathscr{C}$ which is not possible. As a result, we are lead to assume $t=k$. Since there are no type $\textsf{i}$ or $\textsf{ii}$ curves in $C\cap\mathscr{C}$,  $l'\notin C\cap\mathscr{C}$. By construction, either $l'$ has a leg on $\partial^-$, which means that $l'$ is not of type $\textsf{i}$ or $\textsf{ii}$, or $l'$ corresponds to a curve in $C\cap\bar{C}$. In the latter case, $l'\in C\cap C_i$, which is  not possible since $\mathscr{C}$ does not contain intersections of types $\textsf{i}$ and $\textsf{ii}$. 
\end{proof}

Note that we may have $n_j({\mathscr{C}'})\neq n_j(\mathscr{C})$.  Moreover, $C\cap\mathscr{C}'$ may contain curves both legs on $\partial^-$. Figure \ref{fig-type-3-6} illustrates two examples. Nevertheless, we may apply Lemma \ref{l27} to obtain an equivalent reduced simplifier $S_{\textsf{iii}}(\Cscr,i,k)=\mathscr{C}''$ such that $C\cap\mathscr{C}''$ does not contain curves with both legs on $\partial^-$. By Lemma \ref{l27} and the argument of Lemma~\ref{new-lemma},  we further find $g_j({\mathscr{C}''})\leq g_j(\mathscr{C})$ for $j<i$ and $g_i({\mathscr{C}''})<g_i(\mathscr{C})$. 
	
	\begin{figure}
	\def\svgwidth{14cm}
	\begin{center}
	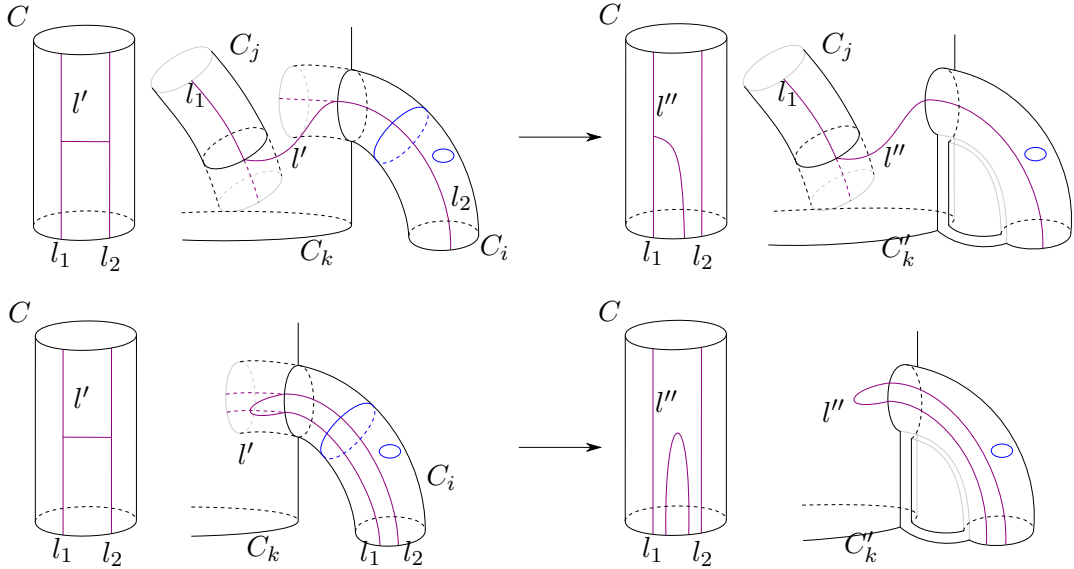
	\caption{ $l'\in C\cap C_k$ has legs on the SBCs $l_1\in C\cap C_j$ and $l_2\in C\cap C_i$ (top-left and bottom-left). $l'$ changes to  curve $l''\in C\cap C'_k$ which is of type $\textsf{iii}$ and has one leg on $l_1$ (top-right), or has both legs on $\partial^-$ (bottom-right).}
	\label{fig-type-3-6}
	\end{center}
	\end{figure}

\begin{proposition}\label{p28}
	Suppose that $B_\Cscr$ is satisfied for the simplifier $\mathscr{C}$ and that $C\cap\Cscr$ does not contain curves with both legs on $\partial^-$. Then there is an equivalent simplifier $\mathscr{C}'$ such that $B_{\Cscr'}$ is satisfied, and $C\cap \mathscr{C}'$  does not contain curves of type $\textsf{iii}$ or curves with both legs on $\partial^-$. 
\end{proposition}
\begin{proof}
	Let $\mathscr{C}^{(d)}=\{C_j^{(d)}\}_{j=1}^n$, $d\geq0$, be a sequence of simplifiers with $\Cscr^{(0)}=\Cscr$, where $\mathscr{C}^{(d+1)}$ is obtained from $\mathscr{C}^{(d)}$ as follows. Let $h_d$ be the smallest index  with  $n_{h_d}({\mathscr{C}^{(d)}})\neq0$ (since $C_1^{(d)}$ includes no generating curves, $h_d\geq2$). Let $l\in C\cap C_{h_d}^{(d)}$ be a SBC and $l'\in C\cap C_k^{(d)}$ have one leg on $l$ and one on $\partial^-$. Suppose that the disk $D_{l'}^\circ\subset C$ (constructed earlier) includes no curves of type $\textsf{iii}$. By the discussion of $\S$\ref{3.5}, we obtain the simplifier $\mathscr{C}^{(d+1)}=S_{\textsf{iii}}(\Cscr^{(d)},h_d,k)$. Note that $B_{\Cscr^{(d)}}$ is satisfied and $C\cap \mathscr{C}^{(d)}$  does not contain  curves with both legs on $\partial^-$. Moreover, we have $g_2(\mathscr{C}^{(d)})\geq g_2(\mathscr{C}^{(d+1)})\geq0$ for $d\geq 0$. Therefore, there is an integer $d_2$ such that $g_2(\mathscr{C}^{(d)})= g_2(\mathscr{C}^{(d+1)})$ for $d\geq d_2$. By the discussion after the proof of Lemma~\ref{new-lemma}, it follows that $n_2({\mathscr{C}^{(d)}})= 0$ for $d\geq d_2$. Consequently, $h_d>2$ for $d\geq d_2$. By repeating this process, we obtain a sequence $d_2\leq d_3\leq \cdots$ of integers so that for all $d\geq d_i$ we have $g_i(\Cscr^{(d)})=g_i(\Cscr^{(d_i)})$ and $n_i({\Cscr{(d)}})=0$. For $d\geq d_n$, we find $n_i({\Cscr^{(d)}})=0$ for  $i=1,\ldots,n$, completing the proof.
\end{proof}

\subsection{Semi-reduced simplifiers and sliding cylinders over one-another}\label{ss-5-7-1}	
	Let $\mathscr{C}=\left(C_i\right)_{i=1}^n$ be obtained from Proposition \ref{p28}. We claim that {\emph{there is an equivalent simplifier $\mathscr{C}'=\left(C'_i\right)_{i=1}^{n'}$ so that  $C\cap C'_i= \emptyset$ unless $C'_i$ has no punctures}}. Let $N_i=N[\Cscr^i]$ and $N_i'=N[(\Cscr')^i]$.  If $C_i$ has no punctures and $C\cap C_i=\emptyset$,  we can assume $i=n$. To prove the claim for $(N,\Cscr,C)$, it then suffices to prove it for $(N_{n-1},\Cscr_{n-1},C)$.  Therefore, we are down to showing the claim when $C$ intersects all the non-punctured cylinders in $\Cscr$.\\
	
	Suppose that $l\in C\cap C_i$ is not a SBC. Then the legs of $l$ are on the punctures of $C_i$. Let  $l_k$ be such a puncture, which is a generating curve on  $C_k$, for some $k>i$. This gives a curve $l'_k\in C\cap C_k$ which intersects the generating curves on $C_k$, and is a SBC by Proposition \ref{p28}. Since $C\cap\mathscr{C}$ has no curves of types $\textsf{i}-\textsf{iii}$, the other leg of $l$ is on another SBC $l''_j\in C\cap C_j$ (Figure \ref{fig-semi-reduced-2}-left). If $l''_j=l'_k$,  there is only one SBC  (Figure \ref{fig-semi-reduced-2}-right).

	\begin{figure}
	\def\svgwidth{13cm}
	\begin{center}
	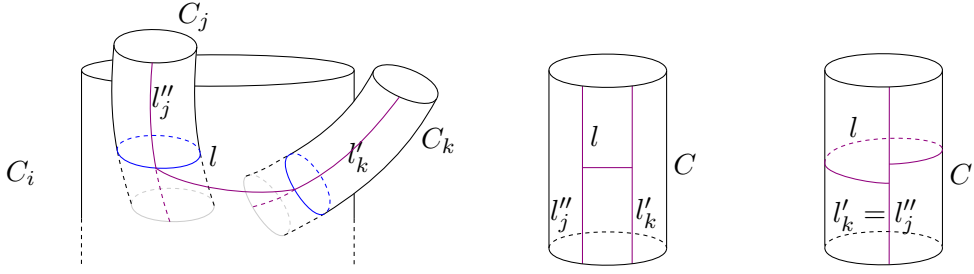
	\caption{$l\in C\cap C_i$ has both legs on two SBCs $l'_k\in C\cap C_k$ and $l''_j\in C\cap C_j$ (left). The cases $l''_j\neq l'_k$ and $l''_j=l'_k$ are illustrated as well (middle and right). }
	\label{fig-semi-reduced-2}
	\end{center}
	\end{figure}

	\begin{definition}\label{d-semi-redu}
	A simplifier $\mathscr{C}=\left(C_i\right)_{i=1}^n$ is called {\emph{semi-reduced}} if $l_i^+=\partial^+C_i$ does not bound a disk in $\partial^+_i$ for each $C_i$ which has punctures and  curves on $C\cap\Cscr$  are SBCs or of type $\textsf{iv}$.
	\end{definition}

\begin{figure}[b]
	\def\svgwidth{16cm}
	\begin{center}
	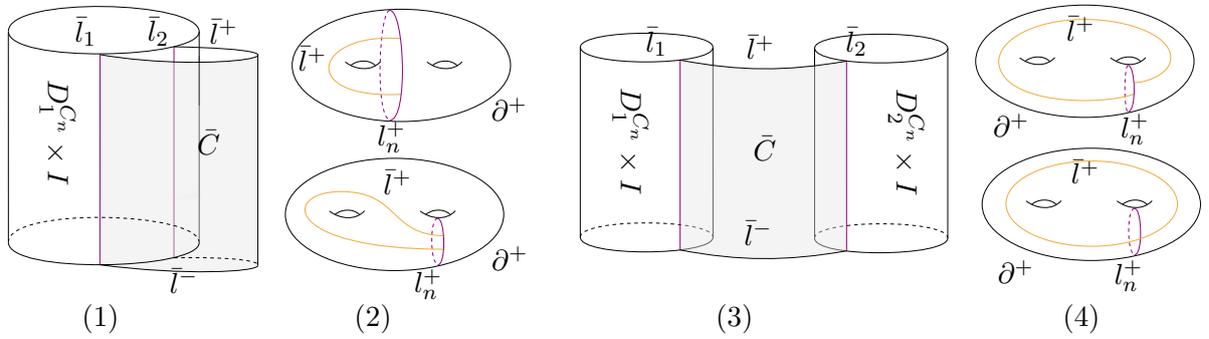
	\caption{$(1)$ and $(3)$:  (one of) the solid cylinders associated with $C_n$ in $N_{n-1}$ and a copy of $\bar{C}\subset C$ with $\partial \bar{C}=\bar{l}_1\cup\bar{l}^+\cup\bar{l}_2\cup\bar{l}^-$ are illustrated in $N_{n-1}$. The curves $\bar{l}_1$ and $\bar{l}_2$ are identified with the SBCs $l_1$ and $l_2$ in $N$. $(2)$ and $(4)$: $\bar{l}^+$ and $l^+_n$ are illustrated on $\partial^+$. }
	\label{fig-semi-reduced-5-1}
	\end{center}
	\end{figure}	
		
	The simplifier obtained by Proposition \ref{p28} is clearly semi-reduced. We deform a given simplifier to obtain an equivalent semi-reduced simplifier. Let $\mathcal{C}_{\mathscr{C}}$ denote the set of non-punctured cylinders in $\mathscr{C}$ and $\mathcal{C}'_{\mathscr{C}}$ denote the set of all the punctured cylinders in $\mathscr{C}$ such that their punctures are the generating curves on the cylinders in $\mathcal{C}_{\mathscr{C}}$. It is clear that if $\mathcal{C}'_{\mathscr{C}}=\emptyset$, $\mathcal{C}_\mathscr{C}=\mathscr{C}$.
	\begin{definition}\label{d29}
	Two SBCs  $l_i\in C\cap C_{j_i}$, for $i=1,2$   are called {\emph{adjacent}} in $C$ (respectively, in $C_{j_1}$ if $j_1=j_2$) if $C\setminus\{l_1,l_2\}$ (respectively, $C_{j_1}\setminus\{l_1,l_2\}$) has a component without any SBCs.
	\end{definition} 

	 If $l_1,l_2\in C\cap C_i$ are two adjacent SBCs in $C$, it is not necessary for $l_1$ and $l_2$ to be adjacent in $C_i$. If  $C_i\in\mathcal{C}_\mathscr{C}$,  Remark \ref{r2.5} implies that we can assume $i=n$.  Let $\bar{C}$ be the closure of the component of $C\setminus\{l_1,l_2\}$ with no SBCs, which corresponds to a rectangle in $N_{n-1}$  with $\partial \bar{C}=\bar{l}_1\cup\bar{l}^+\cup\bar{l}_2\cup\bar{l}^-$. Here $\bar{l}^\pm\subset\partial^\pm N_{n-1}$, while  $\bar{l}_1$ and $\bar{l}_2$ give $l_1$ and $l_2$ in $N$, respectively. Two cases may be distinguished: case $1$ when both  $\bar{l}_1$ and $\bar{l}_2$ are on one of  $D_i^{C_n}\times I$, say  $D_1^{C_n}\times I$ (Figure \ref{fig-semi-reduced-5-1}-left), and case $2$ when $\bar{l_i}\subset D_i^{C_n}\times I$ for $i=1,2$ (Figure \ref{fig-semi-reduced-5-1}-right). 	 Let $L=\bar{C}\cap\mathscr{C}=\{l'_{i_1},\dots,l'_{i_k}\}$ where $l'_{i_s}\in C\cap C_{i_s}$, $1\leq s\leq k$. In case  $2$, it is clear that the two legs of $l'_{i_s}$ are on two different punctures of $C_{i_s}$. In fact,  one leg of $l'_{i_s}$ in $N_{n-1}$ is on an essential curve $l''_1\subset\partial D_1\times I$ and one is on an essential curve $l''_2\subset(\partial D_2)\times I$. $l''_1$ and $l''_2$ are generating curves on $C_n$. If $l''_1=l''_2$ on $C_n$, then $C_{i_s}$ is not a surface of genus zero in $N$.\\

	\begin{figure}[!h]
	\def\svgwidth{6cm}
	\begin{center}
	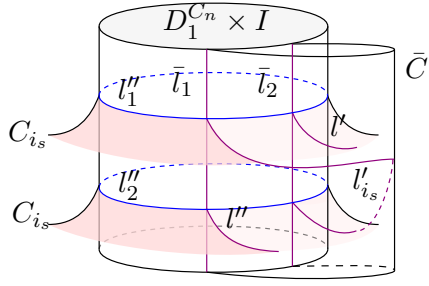
	\caption{ $\bar{l}_1$ and $\bar{l}_2$ are on  $D_1^{C_n}\times I\subset N_{n-1}$. The legs of $l'_{i_s}\in \bar{C}\cap C_{i_s}$ are on different punctures $l''_1$ and $l''_2$ of $C_{i_s}$ and  each component of $\bar{C}\setminus l'_{i_s}$ has nonempty intersection with $C_{i_s}$.}
	\label{fig-semi-reduced-6}
	\end{center}
	\end{figure}
	
	In case $1$, the legs of each $l'_{i_s}$ are on the same puncture of $C_{i_s}$. To see this,  let $\bar{L}\subset L$ consist of the curves with legs on two different punctures. Choose $l'_{i_s}\in\bar{L}$ such that at least one component of $\bar{C}\setminus l'_{i_s}$ has empty intersection with $\bar{L}$ (see Figure \ref{fig-semi-reduced-6}). Suppose that the legs of $l'_{i_s}$ are on the punctures $l''_1$ and  $l''_2$ of $C_{i_s}$ and $\partial l'_{i_s}=\{l''_1\cap l_1,l''_2\cap l_2\}$. There are curves $l'$ and $l''$ in $\bar{C}\cap C_{i_s}$ with legs given by $l_2\cap l''_1$ and  $l_1\cap l''_2$, respectively. Clearly, $l'$ and $l''$ have legs on different punctures of $C_{i_s}$. This contradicts the assumption on $l'_{i_s}$.
	
	\begin{definition}\label{d31}
	Let the adjacent SBCs $l_i\in C\cap C_{j_i}$, for $i=1,2$, and the rectangle $\bar{C}$ be as above. We say that $l_1$ and $l_2$ are {\emph{equivalent}} if $\bar{C}\cap\mathcal{C}'_{\mathscr{C}}\neq\emptyset$	and $j_1=j_2$ only if we are in  case $2$. 
	\end{definition}

Let $\mathscr{C}=\left(C_i\right)_{i=1}^n$ be a semi-reduced simplifier,  $l_i\in C\cap C_{j_i}$, for $i=1,2$, be two equivalent curves  and the rectangle $\bar{C}$ and $L=\bar{C}\cap\mathscr{C}=\{l'_{i_1},\dots,l'_{i_k}\}$ be defined as before. So, $\bar{C}\cap\mathcal{C}'_{\mathscr{C}}\neq\emptyset$ and $\bar{C}\cap\mathcal{C}_\mathscr{C}=\emptyset$.  If $j_1\neq j_2$, slide $C_{j_1}$ over $C_{j_2}$ using $\bar{C}$ to obtain ${C}_{j_1,j_2}$ (Figure \ref{fig-slide-1}-top).  Let $l_{i_s}^i$ be the generating curves on $C_{j_i}$, for $i=1,2$, which  are also disjoint punctures on the punctured cylinders $C_{i_s}$, while we have  $\partial l'_{i_s}=\{l_1\cap l_{i_s}^1,l_2\cap l_{i_s}^2\}$ (Figure \ref{fig-slide-1}-bottom-right). A neighborhood of $C_{j_1}\cup C_{j_2}\cup C_{j_1,j_2}$ is illustrated in Figure \ref{fig-slide-1}-bottom-left. Let $D_{i_s}\subset C_{i_s}$ be a punctured disk with punctures $l^1_{i_s}$ and $l^2_{i_s}$ and with boundary $l''_{i_s}=C_{j_1,j_2}\cap C_{i_s}$. Remove  the disk $D_{i_s}^\circ$,  from $C_{i_s}$ to obtain the cylinder $C'_{i_s}$, with one fewer puncture in comparison with $C_{i_s}$, as the punctures $l^1_{i_s}$ and $l^2_{i_s}$ are replaced with the puncture $l''_{i_s}$.   Over $C_{j_1,j_2}$, we have the generating curves $l''_{i_s}$ for $1\leq s\leq l$. Set $C'_i=C_i$ for $i\neq i_1,\ldots,i_k$, while $C_{i_s}'$ is obtained from $C_{i_s}$ by removing a neighborhood of $l_{i_s}^1\cup l_{i_s}^2\cup l'_{s}$. If $j_1<j_2$, By Lemma \ref{l13}, $\mathscr{C}'=(C'_1,\ldots,C'_{j_2}, C_{j_1,j_2},C'_{j_2+1},\ldots,C'_n)$ is a semi-reduced simplifier which is equivalent to $\mathscr{C}$, and has fewer punctures.\\
	
		\begin{figure}[!h]
	\def\svgwidth{15cm}
	\begin{center}
	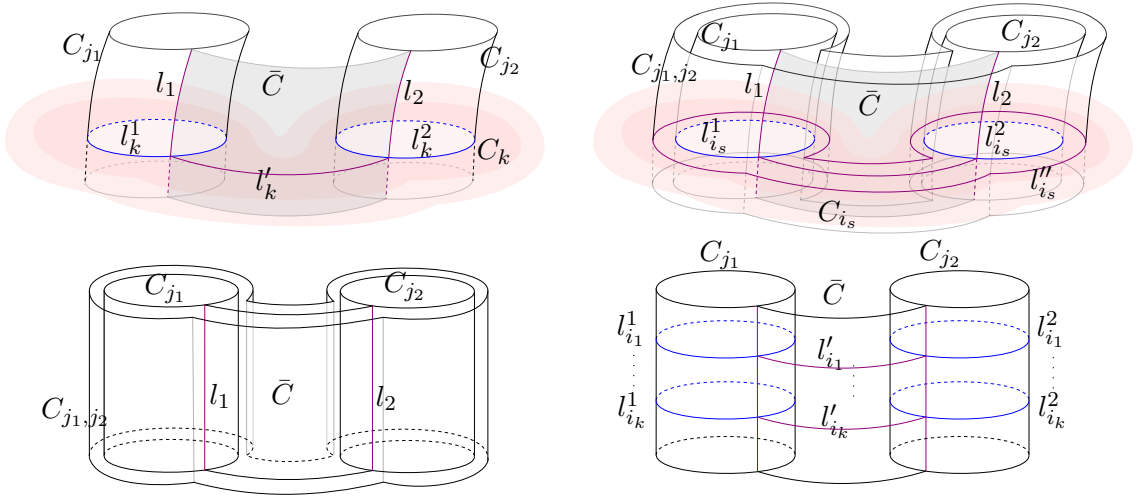
	\caption{ $l_1$ and $l_2$ are SBCs on different cylinders $C_{j_1},C_{j_2}\in\mathcal{C}_\mathscr{C}$. $l'_k\in C\cap C_k$ is a curve in the rectangle $\bar{C}\subset C$ with legs on $l_1$ and $l_2$ (top-left). Slide $C_{j_1}$ over $C_{j_2}$ using $\bar{C}$ to obtain the cylinder $C_{j_1,j_2}$ (bottom-left). $\bar{C}\cap\Cscr=\{l'_{i_1},\ldots,l'_{i_k}\}$, where $l'_{i_s}$ has  legs on $l^i_{i_s}\subset C_{j_i}\cap C_{i_s}$ for $i=1,2$. A neighborhood of $l''_{i_s}$ includes parts from $C_{i_s}$, $C$, $C_{j_1}$, $C_{j_2}$ and $C_{j_1,j_2}$ (right).
	}
	\label{fig-slide-1}
	\end{center}
	\end{figure}

If $j_1=j_2$, consider a parallel copy $C'_{j_1}$ of $C_{j_1}$ which intersects $\bar{C}$.  $C_{j_1}'$ intersects each $C_{i_s}$ in an essential curve $l''_{i_s}$ which is parallel with $l_{i_s}$ in $C_{i_s}$, and they are thus the two boundaries of a (once punctured) disk $D_{i_s}$. Remove $D_{i_s}$  from $C_{i_s}$ and denote the resulting cylinder with $C'_{i_s}$. Set $\Cscr'=(C'_{1},\ldots,C'_{j_1},C_{j_1},C'_{j_1+1},\ldots,C'_n)$, where $C'_k=C_k$ for $k\neq j_1,i_1,i_2,\ldots,i_k$.  By Lemma \ref{l13}, $\mathscr{C}'$ is a semi-reduced simplifier equivalent to $\mathscr{C}$. Corresponding to the SBC $l_1\in C\cap C_{j_1}$, there is a curve $l_1'\in C\cap C_{j_1}'$, which is adjacent to $l_1$ in $C$. As in the case  $j_1\neq j_2$, slide $C'_{j_1}$ over $C_{j_1}$ to obtain a semi-reduced simplifier with fewer punctures.
The above discussion implies that using pairs of equivalent SBCs, we may reduce the number of punctures, while keeping the simplifier semi-reduced. Repeating the above process, we may thus assume that the semi-reduced simplifier $\Cscr=\left(C_i\right)_{i=1}^n$ does not include any equivalent SBCs.

\subsection{Removing two SBCs when there are no equivalent SBCs}\label{ss-5-7-3}
Let us now assume that the semi-reduced simplifier $\Cscr=\left(C_i\right)_{i=1}^n$ does not include any equivalent SBCs. We would next like to reduce the total number of SBCs. By definition, all the punctures of any fixed cylinder $C_i\in\mathcal{C}'_{\mathscr{C}}$ are generating curves on the cylinders in $\mathcal{C}_\mathscr{C}$. Since $C$ intersects all the cylinders in $\mathcal{C}_\mathscr{C}$, there are curves in $C\cap C_i$ which are not SBCs. In fact, let $l_{j_1}\subset C_i$ be a puncture which is a generating curve on  $C_{j_1}\in\mathcal{C}_\mathscr{C}$. Since $C\cap C_{j_1}\neq\emptyset$, there is a SBC $l_1\in C\cap C_{j_1}$. Therefore, there is a curve $l'\in C\cap C_i$ such that $\partial l'$ has a component given by $l_1\cap l_{j_1}$. $l'$ is not a SBC, so it has another leg on a puncture $l_{j_2}$ of $C_i$, where $l_{j_2}$ is a generating curve on some $C_{j_2}\in\mathcal{C}_\mathscr{C}$. Thus, there is a SBC $l_2\in C_{j_2}\cap C$ such that $\partial l'=\{l_2\cap l_{j_2}, l_1\cap l_{j_1}\}$. So, $l_1$ and $l_2$ are adjacent  in $C$. Since $l_1$ and $l_2$ can not be equivalent, we have $j_1=j_2$. As a result, each $l'\in C\cap C_i$ which is not a SBC, has both legs on the same puncture $l_j$. There is a (punctured) disk $D_{l'}\subset C_i$ such that $\partial D_{l'}=l'\cup\bar{l}_j$, for an arc $\bar{l}_j\subset l_j$. By the above argument, for each puncture $l_j$ of $C_i$, there is a curve $l'\in C\cap C_i$ such that $\partial l'\subset l_j$. Thus, we can choose $l'$ such that $D_{l'}$ is a disk without punctures and $(D_{l'})^\circ\cap\mathscr{C}=\emptyset$.\\

	\begin{figure}[!h]
	\def\svgwidth{16cm}
	\begin{center}
	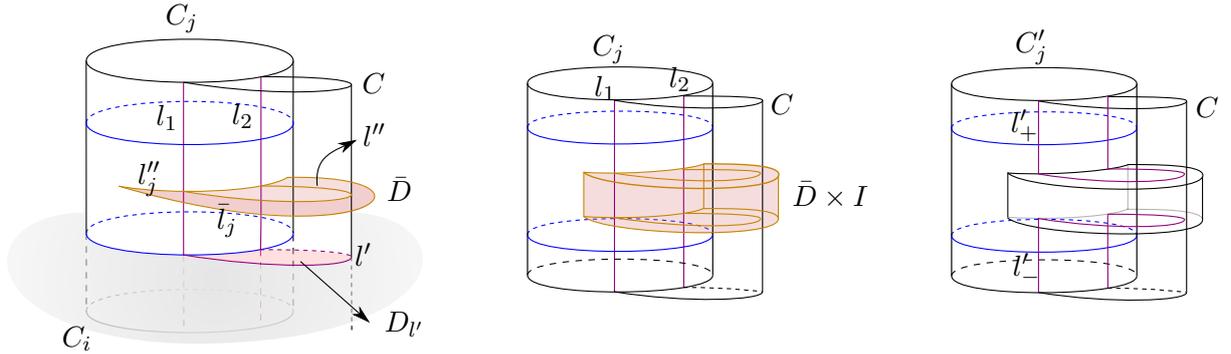
	\caption{ The adjacent SBCs $l_1$ and $l_2$ in $C$ are connected by $l'$ and are adjacent in $C_j$. $D_{l'}$ is a disk with $D_{l'}\cap\mathscr{C}=\emptyset$. $\bar{D}$ is a disk with $\bar{D}^\circ\cap\mathscr{C}=\emptyset$ and $\partial\bar{D}$ is disjoint from the generating curves on $C_j$. The boundary of $l''=\bar{D}\cap \bar{C}$ is on $l_1$ and $l_2$ (left). We have $\partial\bar{D}=l''_j\cup\bar{l}_j$. The interior of $\bar{D}\times I$ is disjoint from $\mathscr{C}$ and $l''_j\times I$ is a disk disjoint from the generating curves on $C_j$ (middle). $C'_j$ is obtained from $C_j$ by replacing $l''_j\times I$ with $\partial(\bar{D}\times I)\setminus(l''_j\times I)$ (right).}
	\label{fig-empty-3}
	\end{center}
	\end{figure}
	
In the above setup,	let $\bar{C}$ be the closure of the component of $C\setminus\{l_1,l_2\}$ with $l'\subset \bar{C}$ and $l_j$ be a generating curve on $C_j\in\mathcal{C}_\mathscr{C}$ (for $j=j_1=j_2$). In particular, $l_1,l_2\in C\cap C_j$ are two SBCs such that $\partial l'=\{l_1\cap l_j,l_2\cap l_j\}$. Note that $l_1$ and $l_2$ are adjacent in  $C_j$; otherwise, there is a SBC $l\in C\cap C_j$ that $l\cap \partial D_{l'}\neq \emptyset$. Thus, $D_{l'}^\circ\cap\mathscr{C}\neq\emptyset$, which contradicts our assumption on $D_{l'}$. Let $\bar{D}$ be an enlarged parallel copy of $D_{l'}$  such that $\bar{D}^\circ\cap\mathscr{C}=\emptyset$ and $\partial\bar{D}=l''_j\cup \bar{l}_j$, where $l''_j=\bar{D}\cap C_j$ is disjoint from the generating curves on $C_j$, and $\bar{l}_j=\partial\bar{D}\setminus l''_j$ is close and parallel to $l''=\bar{D}\cap \bar{C}$. The intersection $l''\cap l''_j$ is a pair of points on $l_1$ and $l_2$. Consider a neighborhood of $\bar{D}$, denoted by $\bar{D}\times I$, such that $(\bar{D}^\circ\times I)\cap\mathscr{C}=\emptyset$ and $l''_j\times I$ is a disk disjoint from the generating curves on $C_j$. Using Lemma \ref{l10}, replace $l''_j\times I$ with {$\partial(\bar{D}\times I)\setminus(l''_j\times I)$} to obtain the cylinder $C'_j$ and an equivalent semi-reduced simplifier (Figure \ref{fig-empty-3}-right). $l_1,l_2\in C\cap C_j$ correspond to $l_+',l_-'\in C\cap C'_j$ with $\partial l'_\pm\subset\partial^\pm$ (Figure \ref{fig-empty-3}-right).\\

\begin{figure}[!h]
\begin{center}
\def\svgwidth{15cm}
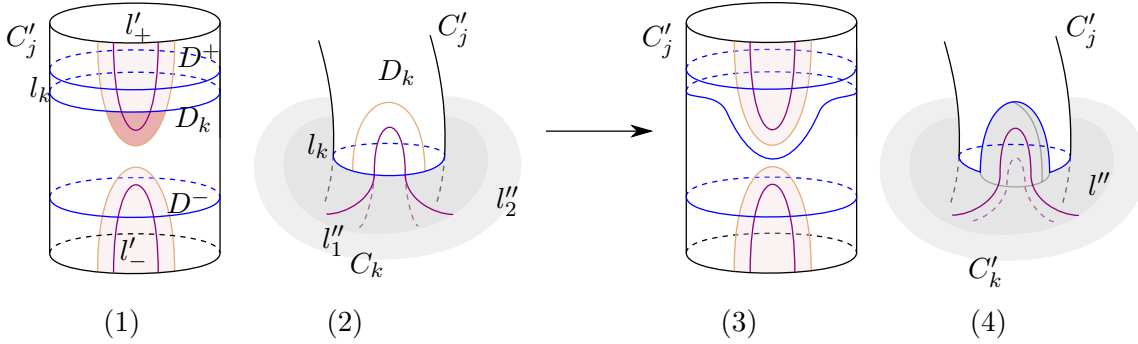
\caption{$D^\pm$ contains $l'_\pm$ and $D_k\subset D^+$ is such that $D_k^\circ\cap \mathscr{C}=\emptyset$ (1). The intersections of $l'_+$ with $l_k$ give $l''_1,l''_2\in C\cap C_k$ (2). Removing $l_k$ from $D_k$ is equivalent to changing $C_k$ to $C'_k$ ((3) and (4)). A leg of $l''_1$ is identified with a leg of $l''_2$.}
\label{fig-empty-5}
\end{center}
\end{figure}

	 We can choose the disjoint disks $D^\pm\subset C'_j$ such that 
$l_\pm'\subset (D^\pm)^\circ$ and the interior of $C\cap C_j'$ is disjoint from $\partial D^\pm$. The curves $l'_\pm$ determine the intersections $\partial D^\pm\cap(C'_j)^\circ$.  Choose $D^\pm$ such that each generating curve $l_k$ on $C_j'$ which is a puncture on some $C_k$, intersects only one of $D^+$ and $D^-$ and meets the corresponding curve $l'_+$ or $l'_-$ in exactly two points. The two intersection points of $l_k$ with $l'_+$ (or $l'_-$) give two curves $l''_1, l''_2\in C\cap C_k$, so that each one has at least one leg in $l_k\cap l_+'$ (or $l_k\cap l_-'$). Note that we may have $l''_1=l''_2$. Let $l_k$ intersect $D^+$. There is a disk $D_k\subset D^+$ with $\partial D_k=\tilde{l}_k\cup\tilde{l}$, where $\tilde{l}\subset\partial D^+$ and $\tilde{l}_k\subset l_k$. Choose $l_k$ such that the generating curves on $C'_j$ do not enter $D_k^\circ$ (Figure \ref{fig-empty-5}-left). Attach a copy of $D_k$ to $C_k$ along $\tilde{l}_k$ to construct $C'_k$. In $C'_k$, the arcs $l''_1$, $l''_2$ and $D_k\cap l'_+$ glue together and give a single curve in $C\cap C_k'$  (Figure \ref{fig-empty-5}). Similarly, by moving other generating curves which intersect $D^\pm$, we may remove them from $D^\pm$. Denote the resulting simplifier by $\mathscr{C}'=\big(C'_i\big)_{i=1}^n$. \\

	The above process changes $l_1$ , $l_2$, and the intersection curves with one leg on $l_1$ or $l_2$, in $C\cap\mathscr{C}$, as illustrated in Figure \ref{fig-empty-6}. In fact, if $l''_1=l''_2$, $l''_1$ is changed to a closed curve in $C\cap C'_k$. If $l''_1\neq l_2''$, the other leg of  $l''_1$ (and $l''_2$) is on a SBC. In this case, $l''_1$ and $l''_2$ are changed to a curve $l''$ which has boundary on SBCs  $\tilde{l}_1$ and $\tilde{l}_2$ (which may be the same). If $l_1$ and $l_2$ are the only SBCs in $C\cap\mathscr{C}$, thus, $C\cap\mathscr{C}'$ is of type $\textrm{I}$ (Figure \ref{fig-empty-6}-left). Using Lemma \ref{l8}, we obtain an equivalent reduced simplifier  such that its positive and negative boundary is included in the positive and negative boundary of $\Cscr'$. Therefore, the intersection of $C$ with this latter reduced simplifier is of type  $\textrm{I}$. Using Proposition \ref{p23}, we obtain an equivalent simplifier which includes $C$.\\

\begin{figure}
\begin{center}
\def\svgwidth{13cm}
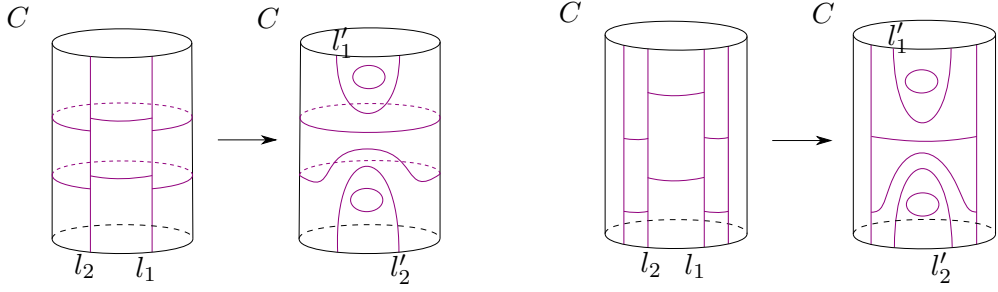
\caption{Left: $C\cap\mathscr{C}$ has only two SBCs $l_1$ and $l_2$. The curves in $C\cap\mathscr{C}\setminus\{l_1,l_2\}$ are changed to closed curves in $C\cap\mathscr{C}'$. $C\cap\mathscr{C}'$ is of type $\textrm{I}$. Right: $C\cap\mathscr{C}$ has more than two SBCs. Closed curves in $C\cap\mathscr{C}'$ are the boundaries of disks in $C$.
}
\label{fig-empty-6}
\end{center}
\end{figure}

Suppose now that $C$ includes SBCs other than $l_1$ and $l_2$. Each closed curve $l\in C\cap\mathscr{C}'$ bounds a disk in $C$ (Figure \ref{fig-empty-6}-right) and is  a curve on a cylinder $C'_i$ which has some punctures. Since $\mathscr{C}$ is a semi-reduced simplifier,  $\mathscr{C}'$ satisfies the first condition in Definition \ref{d-semi-redu}.  Thus, by Proposition \ref{p5}, $l$ bounds a punctured disk on $C'_i$. It is clear that $C\cap\mathscr{C}'$ is nice. By Remark~\ref{r-a-l24}, we can then remove  all the closed curves in $C\cap\mathscr{C}'$. Then using Remark \ref{r-a-l25}, we may also remove $l'_1$ and $l'_2$ to obtain an equivalent semi-reduced simplifier $\mathscr{C}''=\left(C_i''\right)_{i=1}^n$ with $C\cap C_i''\subset C\cap C_i'$ for $1\leq i\leq n$ and $\mathcal{C}_{\mathscr{C}'}\subset\mathcal{C}_{\mathscr{C}''}$, such that the number of SBCs in $C\cap\mathscr{C}''$ is  less than the number of SBCs in $C\cap\mathscr{C}$ and $p_{\mathscr{C}''}\leq p_{\mathscr{C}'}\leq p_{\mathscr{C}}$, where $p_\Cscr$ denotes the number of punctures in $\Cscr$.\\

A few remarks are necessary. First, note that we may have $\mathcal{C}_{\mathscr{C}''}\neq\mathcal{C}_{\mathscr{C}'}$. In fact, when we use Remark~\ref{r-a-l24} or Remark~ \ref{r-a-l25}, some punctures may be removed and some punctured cylinders in $\mathscr{C}'$ may thus become cylinders without punctures in $\mathscr{C}''$. We may also have $\mathcal{C}'_{\mathscr{C}''}\neq\mathcal{C}'_{\mathscr{C}'}$. The second observation is that some cylinders in $\mathcal{C}_{\mathscr{C}''}$ may not intersect $C$. The third observation is that $\mathcal{C}_{\Cscr'}$ and $\mathcal{C}'_{\Cscr'}$ are possibly different from $\mathcal{C}_{\Cscr''}$ and $\mathcal{C}'_{\Cscr''}$.  Thus, some SBCs in the intersection of $\Cscr''$ with $C$ may be equivalent on $C$. Finally, note that we may have $C\cap(\mathscr{C}''\setminus\mathcal{C}_{\mathscr{C}''})=\emptyset$.
\subsection{An equivalent simplifier containing \texorpdfstring{$\boldsymbol{C}$}{Lg}}\label{3.9}\label{ss-5-7-4}
	 \begin{proposition}\label{p29}
	 Given a semi-reduced simplifier $\mathscr{C}=\left(C_i\right)_{i=1}^n$ for $N$ and the cylinder $C$, there is an equivalent semi-reduced simplifier $\mathscr{C}'=\left(C'_i\right)_{i=1}^{n'}$ for $N$ and an integer $m\leq n'$ such that $C$ does not intersect $C_i'$ for $i>m$, and the induced cylinder in $N'=N[(\Cscr')^m]$ only intersects $\mathcal{C}_{\Cscr'_m}$.
	 \end{proposition}
	 \begin{proof}
	 We may assume that $C$ intersect each cylinder in $\mathcal{C}_\mathscr{C}$. Set $\mathscr{C}^{(0)}=\mathscr{C}$, $n^0=n$  and $N^0=N$. If $\mathcal{C}'_{\mathscr{C}^{(0)}}=\emptyset$, we have $\mathscr{C}^{(0)}=\mathcal{C}_{\mathscr{C}^{(0)}}$ and the claim is trivial.  Suppose that the equivalent semi-reduced simplifiers $\Cscr^{(j)}=\{C^j_i\}_{i=1}^{n^j}$ are constructed for $j=0,\ldots,k-1$ so that the following are satisfied. $C^j_i$ is disjoint from $C$ for $i>m^j$, where $m^j\leq n^j$. If $N^j=N[(\Cscr^{(j)})^{m^j}]$, then $C$ (which survives in $N^j$) intersects every non-punctured cylinder in the simplifier $\Cscr^{(j)}_{m^j}$ for $N^j$. Moreover, if $p_{j}=p_{\Cscr^{(j)}}$ and $e_j$ denote the number of punctures and SBCs in $\Cscr^{(j)}$, respectively, we have $p_0\geq p_1\geq \cdots\geq p_{k-1}$, while the equality $p_{j-1}=p_{j}$ is satisfied only if $e_{j-1}>e_j$.\\

If $e_{k-1}\neq 0$ and the intersection of $C$ with $\Cscr^{(k-1)}$ does not include any pair of equivalent SBCs,  we use the construction of $\S$~\ref{ss-5-7-3} to obtain an equivalent semi-reduced simplifier $\mathscr{C}^{(k)}=\{C_i^k\}_{i=1}^{n^k}$. Then  $p_{k}$ is bounded above by $p_{k-1}$ while $e_k$ is less than $e_{k-1}$. On the other hand, if there are pairs of equivalent SBCs in $\Cscr^{(k-1)}$, again we use the discussion of \S~\ref{ss-5-7-3} to obtain an equivalent semi-reduced simplifier $\Cscr^{(k)}$ so that $p_k< p_{k-1}$. \\

By repeating this inductive process, we obtain a simplifier $\Cscr'=\left(C'_i\right)_{i=1}^{n'}=\Cscr^{(m)}$ such that either the intersection of $\Cscr'$ with $C$ is of type $\mathrm{I}$ (i.e. we have $e_{m}=0$), or such that $C$ only intersects the non-punctured cylinders in $N_m$ (i.e. the cylinders in $\mathcal{C}_{\Cscr'_m}$). In the latter case, we are done, while in the former case, Proposition~\ref{p23} completes the proof.
\end{proof}

	Let $\mathscr{C}=\left(C_i\right)_{i=1}^n$ be such that $C\cap\mathscr{C}$ is semi-reduced and $C\cap(\mathscr{C}\setminus\mathcal{C}_\mathscr{C})=\emptyset$. By Remark \ref{r2}, we may assume $C$ intersects $C_{k+1},\ldots,C_n$, and is disjoint from $C_1,\ldots,C_k$. Therefore, for $i\leq k< j$, $C_i$ and $C_j$ are disjoint and there are no generating curves on $C_j$. By Remark \ref{r2}, we may then change the indices and assume that $C$ intersects $C_1,\dots, C_{n-k}$ (which do not have punctures) and  is disjoint from $C_i$, $i>n-k$. The problem of finding a simplifier $\Cscr'$ for $N$ which is equivalent to $\Cscr$ and includes $C$ is then reduced to the similar problem for $N_{n-k}=N[\Cscr^{n-k}]$, the simplifier $\Cscr_{n-k}$ and the cylinder $C$ which survives the removal of the $k$ cylinders in $\Cscr^{n-k}$.  \\
	
	 If $\partial^\pm\setminus\{l_i^\pm\}_{i=1}^n= \{S_j^\pm\}_{j=1}^{k}$,  then  each $S_j^\pm$ is a punctured sphere. Each puncture of $S_j^\pm$, for $1\leq j\leq k$, corresponds to one of the curves $l_i^\pm$, for $1\leq i\leq n$, and each curve $l_i^\pm$ corresponds to two punctures, denoted $l_i^{\pm+}$ and $l_i^{\pm-}$. Each $S_j^\pm$ corresponds to a component of $\partial_0^\pm=\partial^\pm N_0$. We may thus label the spheres so that $S_j^+$ and $S_j^-$ correspond to the same connected  component $N_0^j$ of $N_0$.  It also follows from this observation that 	 $l_j^{++}$ (or $l_j^{+-}$) appears in the boundary of $S_j^+$ if and only if   $l_j^{-+}$ (or $l_j^{--}$) appears in the boundary of $S_j^-$.\\
	 
By adding a parallel copy $C_{n+i}$ of $C_i$ to the simplifier, we obtain a new equivalent simplifier with all the previous properties, which does not include any pair of curves $l_i^{++}$ and $l_i^{+-}$ which are both punctures of the same sphere $S_j^+$ (which means that $l_i^{-+}$ and $l_i^{--}$ are not both punctures of  $S_j^-$). By the above discussion, we may further assume that $\mathscr{C}=\left(C_i\right)_{i=1}^n$ is such that $\mathcal{C}_\mathscr{C}=\mathscr{C}$. Further, $C$ intersects each $C_i$ in $\mathscr{C}$ in some SBCs (condition A). Moreover, $l_i^{++}$ and $l_i^{+-}$ are not both on the boundary of $S_j^+$ and $l_i^{-+}$ and $l_i^{--}$ are not both on the boundary of $S_j^-$ for $i=1,\ldots,n$ and $j=1,\ldots,k$ (Condition B). The boundary of the curves in  $C\cap S_j^\pm$ is on the punctures of $S_j^\pm$. We may distinguish two possibilities. 
	\begin{enumerate}
	\setlength\itemsep{-0.3em}
	\item[{\bf{M.1}}]\label{M1}  Both legs of every curve $l_+\in C\cap S_j^+$ are on the same puncture of $S_j^+$;
	\item[{\bf{M.2}}]\label{M2} There is a curve $l_+\in S_j^+\cap C$ with boundary on two different punctures $\bar{l}_1^+$ and $\bar{l}_2^+$ of $S_j^+$. 
	\end{enumerate}
	
In case {\bf{M.1}}, choose $l_+\in C\cap S_j^+$ and the puncture $\bar{l}$ on $S_j^+$ so that both legs of $l_+$ are on $\bar{l}$, while $l_+$ and (an arc on) $\bar{l}$ bound a disk $D^+_{l_+}$ on $S^+_j$, which does not include any punctures.  Further, assume that $l_+$ is such that $(D^+_{l_+})^\circ\cap C=\emptyset$. Without loss of generality, we can assume that $\bar{l}$ corresponds to $l^+_n$. There are SBCs $l_1,l_2\in C\cap C_n$, such that $\partial l_+=\{\partial l_1\cap l_n^+,\partial l_2\cap l_n^+\}$. Let $C''$ denote  the component of $C_n\setminus\{l_1,l_2\}$ with $\partial D^+_{l_+}\cap C''\neq\emptyset$.  Let $C'$ denote the component of $C\setminus\{l_1,l_2\}$ with $l_+\subset C'$. Note that $C'\cap\partial^-\subset S_j^-$. Let $l_-=C'\cap\partial^-$ and $\bar{l}_n=C''\cap\partial^-$. We then have
	 $$\partial l_-=\partial\bar{l}_n=\{\partial l_1\cap l_n^-,\partial l_2\cap l_n^-\}\quad\quad\text{and}\quad\quad l_-\cup\bar{l}_n=\partial (C'\cup D^+_{l_+}\cup C''),$$ 
	 where $C'\cup D^+_{l_+}\cup C''$ is a disk. By Lemma \ref{l11}, $l_-\cup\bar{l}_n$ bounds a disk $D_{l_-}^-\subset\partial^-$ and from $(D^+_{l_+})^\circ\cap C=\emptyset$, it is clear that $(D_{l_-}^-)^\circ\cap C=\emptyset$ (Figure \ref{figlast-3}-left).\\
	 
	 \begin{figure}[!h]
	 \def\svgwidth{14cm}
	 \begin{center}
	 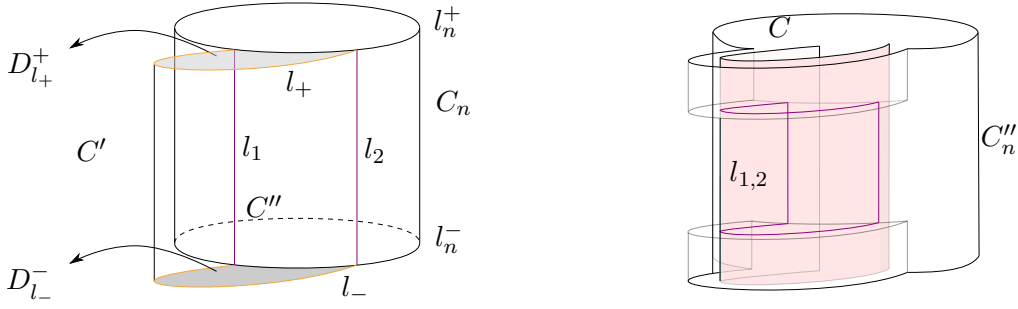
	 \caption{Left: the boundary of $l_+\in C\cap S_j^+$  determines two SBCs $l_1,l_2\in C\cap C_n$, which are adjacent in $C$ and $C_n$. $D_{l_+}^+$, together with $C'\subset C$  and $C''\subset C_n$ determine a disk in $N$ with boundary on $S_j^-$, and thus a disk $D_{l_-}^-\subset S_j^-$ with the same boundary. Right: The cylinder $C''_n$ is obtained from $C_n$ by moving $l_n^+$ through $D_{l_+}^+$ and $l_n^-$ through $D_{l_-}^-$. $l_1$ and $l_2$ are changed to a closed curve $l_{1,2}\in C\cap C''_n$.}
	 \label{figlast-3}
	 \end{center}
	 \end{figure}
	
Move $l_n^+$ through $D_{l_+}^+$ and $l_n^-$ through $D_{l_-}^-$ and correspondingly, isotope $C_n$ to a cylinder $C_n''$ so that in $C\cap C''_n$, $l_1$ and $l_2$ are changed to a closed curve  $l_{1,2}$ (Figure \ref{figlast-3}-right). It is clear that $l_{1,2}$  bounds a disk $D$ (resp. $D'$) in $C$ (resp. in $C''_n$) such that $D^\circ\cap \mathscr{C}=\emptyset$. Using Lemma \ref{l12}, change $C''_n$ by replacing $D'$ with $D$ and denote the resulting cylinder by $C'_n$.  Set $\mathscr{C}'=\left(C_1,\ldots,C_{n-1},C'_n\right)$. Then $\mathscr{C}'$ is a simplifier equivalent to $\Cscr$ which satisfies the conditions A and B. Repeating this process for other curves in $C\cap S_j$ with both legs on $\bar{l}$, we obtain an equivalent simplifier, denoted again by $\mathscr{C}=(C_i)_{i=1}^n$, with $C\cap C_n=\emptyset$. This reduces the problem to the study of the intersections of  $C\subset N_{n-1}$ with the simplifier $\Cscr_{n-1}$, which  satisfies the conditions A and B.\\
	
	\begin{figure}[!h]
	\def\svgwidth{14cm}
	\begin{center}
	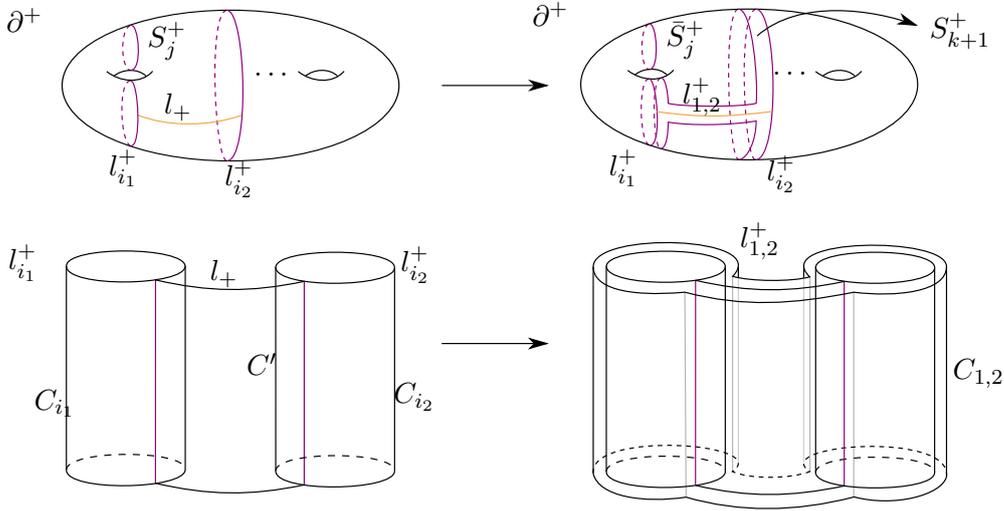
	\caption{The legs of $l_+\in C\cap S_j^+$ are on different punctures of $S_j^+$ (top-left). $\partial l_+$ determines two SBCs $l_1\in C\cap C_{i_1}$  and $l_2\in C\cap C_{i_2}$. $l_+$ is on the boundary of the rectangle $C'\subset C$ (bottom-left). $C_{1,2}$ is obtained by sliding $C_{i_1}$ over $C_{i_2}$ (bottom-right).  $S_j^+$ is changed to $\bar{S}_j^{+}$ and a new component, denoted $S_{k+1}^+$, is added to $\partial^\pm\setminus\{l_j^\pm\}_{j=1}^n$ (top-right).}
	\label{figlast-4}
	\end{center}
	\end{figure}

Next, let us consider case {\bf{M.2}}.	Suppose that $\bar{l}_1^+$ and $\bar{l}_2^+$ are obtained by cutting $\partial^+$ along $l_{i_1}^+=\partial^+ C_{i_1}$ and $l_{i_2}^+=\partial^+ C_{i_2}$, respectively, where $l_{i_1}^+\neq l_{i_2}^+$. Let $C'\subset C$ be the closure of the rectangle in the complement of SBCs determined by $l_+$. Let $l_-=C'\cap\partial^-\subset S_j^-$. Let $\bar{l}_1^-$ and $\bar{l}_2^-$ be the two punctures of $S_j^-$ obtained by cutting $\partial^-$ along $l_{i_1}^-=\partial^- C_{i_1}$ and $l_{i_2}^-=\partial^- C_{i_2}$. Consider the cylinder $C_{1,2}$ obtained by sliding $C_{i_1}$ over $C_{i_2}$ using the rectangle $C'$ (Figure \ref{figlast-4}-bottom). By Lemma \ref{l13}, $\mathscr{C}'=\left(C_1,\ldots,C_{n},C_{1,2}\right)$ is a simplifier for $N$ equivalent to $\mathscr{C}$. After the above process, $\partial^\pm-\cup_{j=1}^n l_j^\pm-l_{1,2}^\pm$ (with $l_{1,2}^\pm=\partial^\pm C_{1,2}$) has one more component, denoted $S^\pm_{k+1}$, in comparison with $\partial^\pm\setminus\{l_j^\pm\}_{j=1}^n$. When we consider $S_{k+1}^\pm$ as a subsurface of $\partial^\pm$, we have $\partial S_{k+1}^\pm=l_{1,2}^\pm\cup l_{i_1}^\pm\cup l_{i_2}^\pm$. Furthermore, $S_j^\pm$ changes to a  corresponding component, denoted $\bar{S}_j^{\pm}$, which is obtained from $S_j^\pm$ by removing the punctures and a punctured disk containing the two punctures $\bar{l}_1^\pm$ and $\bar{l}_2^\pm$ (Figure \ref{figlast-4}-top-right). \\

If $l^+\cap l_{1,2}^+=\emptyset$, let $N'=N[C_{1,2}]$, which is equipped with the simplifier $\Cscr'_{n}$ satisfying conditions A and B. If $\tilde{S}_j^{\pm}$ is obtained from $\bar{S}_j^{\pm}$ by attaching disks to the punctures corresponding to $l_{1,2}^\pm$ and $\tilde{S}_{k+1}^{\pm}$ is obtained from $S_{k+1}^\pm$ by attaching disks to the punctures corresponding to $l_{1,2}^\pm$, then
	$$\partial^\pm N'\setminus \cup_{i=1}^n l_i^\pm=S_1^\pm\cup\dots\cup S_{j-1}^{\pm}\cup \tilde{S}_j^{\pm}\cup S_{j+1}^{\pm}\cup\dots\cup S_k^{\pm}\cup \tilde{S}_{k+1}^{\pm}.$$
Moreover, $\tilde{S}_{k+1}^{\pm}\setminus  C$ is a union of disks, while the number of punctures in $\tilde{S}_j^{\pm}$ is one less than the number of punctures in $\bar{S}_j^{\pm}$ (Figure \ref{fig-last-6}-left). If $l^+\cap l_{1,2}^+\neq\emptyset$ and correspondingly $l^-\cap l_{1,2}^-\neq\emptyset$, since $l\cap(l_1^\pm\cup l_2^\pm)\neq\emptyset$, $S_{k+1}^\pm\setminus C$ is a union of disks (Figure \ref{fig-last-6}-right), and we set $N'=N$.

\begin{figure}[!h]
\begin{center}
\def\svgwidth{\textwidth}
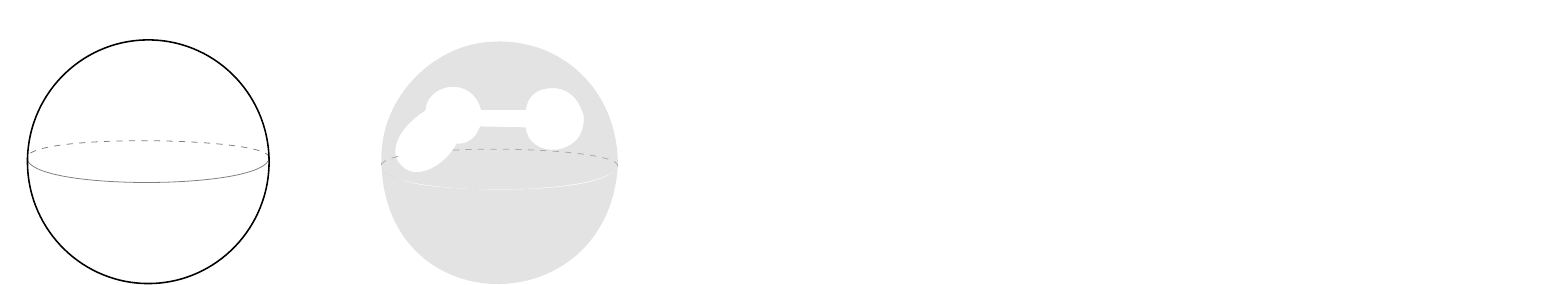	
\caption{Left: $\tilde{S}_{k+1}^{+}\setminus C$ is a union of disks.  Right:  $C\cap S_{k+1}^+\neq\emptyset$ and $S_{k+1}^\pm\setminus C$ consists  disks.}
\label{fig-last-6}
\end{center}
\end{figure}

	\begin{lemma}\label{l30}
	Let $\mathscr{C}=\left(C_i\right)_{i=1}^n$ be a simplifier for $N$ such that $C\cap\mathscr{C}$ is semi-reduced and $C\cap(\mathscr{C}\setminus\mathcal{C}_\mathscr{C})=\emptyset$. Then there is an equivalent simplifier $\mathscr{C}'=\{C'_i\}_{i=1}^{n'}$ and an integer $m\leq n'$ such that $C$ does not intersect $C'_i$ for $i>m$. Moreover, the simplifier $\{C'_i\}_{i=1}^m$ for the manifold $N'=N[(\Cscr')^m]$ consists of non-punctured cylinders which intersect $C$ in SBCs, while   each component of $\partial^\pm N'\setminus (\partial^\pm C\cup \partial^\pm C'_1\cup\cdots\cup\partial^\pm C'_m) $ is a disk.
	\end{lemma}

\begin{proof}
As discussed earlier, we may assume that $\mathscr{C}=\left(C_i\right)_{i=1}^n$  satisfies the conditions A and B. Set $\mathscr{C}^{(0)}=\mathscr{C}=\left(C_i\right)_{i=1}^n$ and $N^0=N$. Let $\partial^+\setminus(\partial^+ C_1\cup\cdots\cup\partial^+ C_n)=S_1^+\cup\dots\cup S_k^+$ and set
	$$D_\mathscr{C}=\{S_j^+\ |\text{ some\ regions\ in\ }S_j^+\setminus C\text{ are\ not\ disks}\}.$$
	Let $\tilde{p}_\mathscr{C}$ denote the total number of the punctures in the punctured spheres in  $D_\mathscr{C}$. If $D_\mathscr{C}=\emptyset$, the claim follows by our earlier considerations.  Let $\{\mathscr{C}^{(i)}\}_{i\geq 1}$ be a sequence of simplifiers such that $\mathscr{C}^{(i)}$ is obtained from $\mathscr{C}^{(i-1)}$ by following the construction in the cases {\bf{M.1}} and {\bf{M.2}} for a sphere $S_j^+\in D_{\mathscr{C}^{(i-1)}}$. It follows that $\tilde{p}_{\mathscr{C}^{(i)}}<\tilde{p}_{\mathscr{C}^{(i-1)}}$. In particular, the above process stops for some simplifier $\Cscr^{(d)}$ with $D_{\Cscr^{(d)}}=\emptyset$, which is equivalent to $\Cscr$ and has the desired properties.
	\end{proof}

	\begin{proposition}\label{p31}
	Let $N$ be a manifold with a simplifier $\mathscr{C}=\left(C_i\right)_{i=1}^n$ such that each $C_i$ is a cylinder without punctures, $C\cap\mathscr{C}$ consists of SBCs, and 
\begin{equation}\label{eq:complement}
\partial^\pm\setminus(\partial^\pm C\cup \partial^\pm C_1\cup\cdots\cup\partial^\pm C_n)=D_1^\pm\amalg\cdots\amalg D_i^\pm,
\end{equation}
where each $D_i^\pm$ is a disk. Then there is an equivalent simplifier  for $N$ which includes $C$.
\end{proposition}
\begin{proof}
Each $D_i^+$ is a polygon and for each edge $\bar{l}$ of $D_i^+$, there is a rectangle $C_{\bar{l}}$, which is a subset of $C$ or some cylinder $C_i$ and is disjoint from SBCs, such that $\bar{l}\subset \partial C_{\bar{l}}$. If ${L}_i$ denotes the set of edges of $D_i^+$, $D_i^+\cup_{\bar{l}\in {L}_i} C_{\bar{l}}$ is a disk with boundary on one of $\partial D_{i'}^-$ for $i'=1,\ldots,k.$  After relabeling the indices, we may assume that $i'=i$.  Then 
$\bar{S}_i=\left(\cup_{\bar{l}\in {L}_i} C_{\bar{l}}\right)\cup D_i^+\cup D_i^-$ for $i=1,\ldots,k$, 
 are spheres disjoint from $\mathscr{C}$, after a slight perturbation. Let $\bar{S}_i=\partial M_i$, where $M_i$  is a  $3$-manifold with one sphere boundary (since $\bar{S}_i\cap\mathscr{C}=\emptyset$, one component of $N\setminus\bar{S}_i$ does not contain $\partial^+\amalg\partial^-$). Then
\[N=\overline{M}_1\#\dots\#\overline{M}_k\#(\partial^+\times I),\]
and $\mathscr{C}\subset(\partial^+\times I)$. If $N'_0=\partial^+\times I[\Cscr]$, we have $N'_0\doteq S^2\times I$. Let $\partial^+$ be of genus $g$. There are $g$ closed curves $\gamma_1,\dots,\gamma_g$ in $\partial^+$, disjoint from $l^+$ such that $\partial^+\setminus (l^+\cup \gamma_1\cup\cdots\cup\gamma_g)$  consists of punctured spheres. Then $\mathscr{C}'=(\gamma_1\times I,\ldots,\gamma_g\times I,C)$ is a simplifier in $\partial^+\times I$, equivalent to $\mathscr{C}$. Since $N_0\doteq\overline{M}_1\#\dots\#\overline{M}_k\# N_0'$,  $\mathscr{C}'$ gives a simplifier in $N$ which is equivalent to $\mathscr{C}$.
\end{proof}
\begin{proof}[Proof of Theorem \ref{thm1}]
As discussed in \S \ref{s-equiv}, $C\cap \mathscr{C}$ is either of type $\textrm{I}$ or of type $\textrm{II}$.  By Lemma \ref{l8}, we can assume that $\mathscr{C}$ is   a reduced simplifier. Proposition \ref{p23} proves the theorem when $C\cap\mathscr{C}$ is of type $\textrm{I}$. We may thus assume that $C\cap \mathscr{C}$ is of type $\textrm{II}$. By Lemma \ref{l17}, $\mathscr{C}$ is equivalent to a reduced simplifier which has nice intersection of type $\textrm{II}$ with $C$. By Proposition \ref{p26} and Proposition \ref{p28}, $\mathscr{C}$ is equivalent to a semi-reduced simplifier $\mathscr{C}'=\left(C_i'\right)_{i=1}^{n'}$ and  by Proposition \ref{p29}, we may further assume that  $C$ survives in $N'_m=N[(\mathscr{C}')^m]$, $m\leq n'$, and only intersects  $\mathcal{C}_{\mathscr{C}_m'}$. \\

Apply Lemma \ref{l30} to $\Cscr''=\Cscr'_m=\left(C_i''\right)_{i=1}^{n''}$. Suppose that $C$ is disjoint from $C_i''$ for $i>k$ and that $C_j''$ is not punctured for $j\leq k$ and cuts $C$ in SBCs. The simplifier $\Cscr''_k$ on $N''_k=N'_m[(\Cscr'')^k]$ is equivalent to a simplifier $\mathscr{C}^*=\{C^*_i\}_{i=1}^{n^*}$ such that $C$ is disjoint from $C_i^*$ for $i>l$ and $C_j^*$ is not punctured and cuts $C$ in SBCs for $j<l$, while for the manifold $N^*_l=N''_k[(\Cscr^*)^l]$, each component of $\partial^\pm N^*_l\setminus (l^+\cup \partial^{\pm}C_1^*\cup\cdots\cup \partial^{\pm}C_l^*)$ is a disk. By Proposition \ref{p31}, $\mathscr{C}^*$  is equivalent to a simplifier for $N^*_l$ which includes $C$. By adding $C^*_{l+1},\ldots, C^*_{n^*},C''_{k+1},\ldots,C''_{n''},C'_{m+1},\ldots,C'_{n'}$  to the latter simplifier for $N^*_l$, we obtain the simplifier $\Cscr^{**}$ for  $N$, which is equivalent to $\mathscr{C}$ and includes $C$. 
\end{proof}

\bibliography{simplifying-structures}
\Addresses
\end{document}

%% file: figures/sliding-n1.pdf_tex
\begingroup%
  \makeatletter%
  \providecommand\color[2][]{%
    \errmessage{(Inkscape) Color is used for the text in Inkscape, but the package 'color.sty' is not loaded}%
    \renewcommand\color[2][]{}%
  }%
  \providecommand\transparent[1]{%
    \errmessage{(Inkscape) Transparency is used (non-zero) for the text in Inkscape, but the package 'transparent.sty' is not loaded}%
    \renewcommand\transparent[1]{}%
  }%
  \providecommand\rotatebox[2]{#2}%
  \newcommand*\fsize{\dimexpr\f@size pt\relax}%
  \newcommand*\lineheight[1]{\fontsize{\fsize}{#1\fsize}\selectfont}%
  \ifx\svgwidth\undefined%
    \setlength{\unitlength}{1083.38980913bp}%
    \ifx\svgscale\undefined%
      \relax%
    \else%
      \setlength{\unitlength}{\unitlength * \real{\svgscale}}%
    \fi%
  \else%
    \setlength{\unitlength}{\svgwidth}%
  \fi%
  \global\let\svgwidth\undefined%
  \global\let\svgscale\undefined%
  \makeatother%
  \begin{picture}(1,0.29175254)%
    \lineheight{1}%
    \setlength\tabcolsep{0pt}%
    \put(0,0){\includegraphics[width=\unitlength,page=1]{sliding-n1.pdf}}%
    \put(0.12421944,0.24694926){\color[rgb]{0,0,0}\makebox(0,0)[lt]{\lineheight{1.25}\smash{\begin{tabular}[t]{l}$C_{i_1}$\end{tabular}}}}%
    \put(-0.0003894,0.09049589){\color[rgb]{0,0,0}\makebox(0,0)[lt]{\lineheight{1.25}\smash{\begin{tabular}[t]{l}$C_{i_2}$\end{tabular}}}}%
    \put(0.03422415,0.25664105){\color[rgb]{0,0,0}\makebox(0,0)[lt]{\lineheight{1.25}\smash{\begin{tabular}[t]{l}$C_{i_3}$\end{tabular}}}}%
    \put(0.91340902,0.25664105){\color[rgb]{0,0,0}\makebox(0,0)[lt]{\lineheight{1.25}\smash{\begin{tabular}[t]{l}$C_{i_1}'$\end{tabular}}}}%
    \put(0.91340902,0.13203226){\color[rgb]{0,0,0}\makebox(0,0)[lt]{\lineheight{1.25}\smash{\begin{tabular}[t]{l}$C_{i_2}'$\end{tabular}}}}%
    \put(0.82341376,0.25664105){\color[rgb]{0,0,0}\makebox(0,0)[lt]{\lineheight{1.25}\smash{\begin{tabular}[t]{l}$C_{i_3}'$\end{tabular}}}}%
    \put(0.03422415,0.13757036){\color[rgb]{0,0,0}\makebox(0,0)[lt]{\lineheight{1.25}\smash{\begin{tabular}[t]{l}$l'$\end{tabular}}}}%
    \put(0.17267845,0.17910666){\color[rgb]{0,0,0}\makebox(0,0)[lt]{\lineheight{1.25}\smash{\begin{tabular}[t]{l}$l$\end{tabular}}}}%
    \put(0.16575575,0.10987955){\color[rgb]{0,0,0}\makebox(0,0)[lt]{\lineheight{1.25}\smash{\begin{tabular}[t]{l}$l_1$\end{tabular}}}}%
    \put(0.04945412,0.17910666){\color[rgb]{0,0,0}\makebox(0,0)[lt]{\lineheight{1.25}\smash{\begin{tabular}[t]{l}$l_2$\end{tabular}}}}%
    \put(0.0785295,0.1624922){\color[rgb]{0,0,0}\makebox(0,0)[lt]{\lineheight{1.25}\smash{\begin{tabular}[t]{l}$D$\end{tabular}}}}%
    \put(0,0){\includegraphics[width=\unitlength,page=2]{sliding-n1.pdf}}%
    \put(0.38174452,0.20956668){\color[rgb]{0,0,0}\makebox(0,0)[lt]{\lineheight{1.25}\smash{\begin{tabular}[t]{l}$D'$\end{tabular}}}}%
    \put(0.41774263,0.23171931){\color[rgb]{0,0,0}\makebox(0,0)[lt]{\lineheight{1.25}\smash{\begin{tabular}[t]{l}$l_1''$\end{tabular}}}}%
    \put(0.49389252,0.24002662){\color[rgb]{0,0,0}\makebox(0,0)[lt]{\lineheight{1.25}\smash{\begin{tabular}[t]{l}$C_{i_1}'=C_{i_1}\cup_{l}D'$\end{tabular}}}}%
    \put(0.83449004,0.22618113){\color[rgb]{0,0,0}\makebox(0,0)[lt]{\lineheight{1.25}\smash{\begin{tabular}[t]{l}$l_1''$\end{tabular}}}}%
    \put(0.87464185,0.1624922){\color[rgb]{0,0,0}\makebox(0,0)[lt]{\lineheight{1.25}\smash{\begin{tabular}[t]{l}$l_2$\end{tabular}}}}%
    \put(0.80679923,0.13480131){\color[rgb]{0,0,0}\makebox(0,0)[lt]{\lineheight{1.25}\smash{\begin{tabular}[t]{l}$l_1'$\end{tabular}}}}%
  \end{picture}%
\endgroup%

%% file: figures/l12.pdf_tex
\begingroup%
  \makeatletter%
  \providecommand\color[2][]{%
    \errmessage{(Inkscape) Color is used for the text in Inkscape, but the package 'color.sty' is not loaded}%
    \renewcommand\color[2][]{}%
  }%
  \providecommand\transparent[1]{%
    \errmessage{(Inkscape) Transparency is used (non-zero) for the text in Inkscape, but the package 'transparent.sty' is not loaded}%
    \renewcommand\transparent[1]{}%
  }%
  \providecommand\rotatebox[2]{#2}%
  \newcommand*\fsize{\dimexpr\f@size pt\relax}%
  \newcommand*\lineheight[1]{\fontsize{\fsize}{#1\fsize}\selectfont}%
  \ifx\svgwidth\undefined%
    \setlength{\unitlength}{610.01602221bp}%
    \ifx\svgscale\undefined%
      \relax%
    \else%
      \setlength{\unitlength}{\unitlength * \real{\svgscale}}%
    \fi%
  \else%
    \setlength{\unitlength}{\svgwidth}%
  \fi%
  \global\let\svgwidth\undefined%
  \global\let\svgscale\undefined%
  \makeatother%
  \begin{picture}(1,0.38584512)%
    \lineheight{1}%
    \setlength\tabcolsep{0pt}%
    \put(0,0){\includegraphics[width=\unitlength,page=1]{l12.pdf}}%
    \put(-0.00223323,0.06016811){\color[rgb]{0,0,0}\makebox(0,0)[lt]{\lineheight{0}\smash{\begin{tabular}[t]{l}$C_k$\end{tabular}}}}%
    \put(0.00006638,0.2881646){\color[rgb]{0,0,0}\makebox(0,0)[lt]{\lineheight{0}\smash{\begin{tabular}[t]{l}$l_k^+$\end{tabular}}}}%
    \put(0.22121939,0.24954177){\color[rgb]{0,0,0}\makebox(0,0)[lt]{\lineheight{0}\smash{\begin{tabular}[t]{l}$l_1$\end{tabular}}}}%
    \put(0.41887734,0.25181371){\color[rgb]{0,0,0}\makebox(0,0)[lt]{\lineheight{0}\smash{\begin{tabular}[t]{l}$l'_1$\end{tabular}}}}%
    \put(0.30755274,0.36313824){\color[rgb]{0,0,0}\makebox(0,0)[lt]{\lineheight{0}\smash{\begin{tabular}[t]{l}$D$\end{tabular}}}}%
    \put(0.11085759,0.08141889){\color[rgb]{0,0,0}\makebox(0,0)[lt]{\lineheight{0}\smash{\begin{tabular}[t]{l}$D'$\end{tabular}}}}%
    \put(0.8997785,0.03972072){\color[rgb]{0,0,0}\makebox(0,0)[lt]{\lineheight{0}\smash{\begin{tabular}[t]{l}$C'_k$\end{tabular}}}}%
    \put(0,0){\includegraphics[width=\unitlength,page=2]{l12.pdf}}%
    \put(0.23069583,0.06827346){\color[rgb]{0,0,0}\makebox(0,0)[lt]{\lineheight{0}\smash{\begin{tabular}[t]{l}$l_2$\end{tabular}}}}%
    \put(0,0){\includegraphics[width=\unitlength,page=3]{l12.pdf}}%
  \end{picture}%
\endgroup%

%% file: figures/Cylinder-A.pdf_tex
\begingroup%
  \makeatletter%
  \providecommand\color[2][]{%
    \errmessage{(Inkscape) Color is used for the text in Inkscape, but the package 'color.sty' is not loaded}%
    \renewcommand\color[2][]{}%
  }%
  \providecommand\transparent[1]{%
    \errmessage{(Inkscape) Transparency is used (non-zero) for the text in Inkscape, but the package 'transparent.sty' is not loaded}%
    \renewcommand\transparent[1]{}%
  }%
  \providecommand\rotatebox[2]{#2}%
  \newcommand*\fsize{\dimexpr\f@size pt\relax}%
  \newcommand*\lineheight[1]{\fontsize{\fsize}{#1\fsize}\selectfont}%
  \ifx\svgwidth\undefined%
    \setlength{\unitlength}{327.20980048bp}%
    \ifx\svgscale\undefined%
      \relax%
    \else%
      \setlength{\unitlength}{\unitlength * \real{\svgscale}}%
    \fi%
  \else%
    \setlength{\unitlength}{\svgwidth}%
  \fi%
  \global\let\svgwidth\undefined%
  \global\let\svgscale\undefined%
  \makeatother%
  \begin{picture}(1,0.39257842)%
    \lineheight{1}%
    \setlength\tabcolsep{0pt}%
    \put(0,0){\includegraphics[width=\unitlength,page=1]{Cylinder-A.pdf}}%
    \put(0.15535834,0.28538664){\color[rgb]{0,0,0}\makebox(0,0)[lt]{\lineheight{1.25}\smash{\begin{tabular}[t]{l}$\partial^+_C$\end{tabular}}}}%
    \put(0,0){\includegraphics[width=\unitlength,page=2]{Cylinder-A.pdf}}%
    \put(0.37843446,0.02061588){\color[rgb]{0,0,0}\makebox(0,0)[lt]{\lineheight{1.25}\smash{\begin{tabular}[t]{l}$D^+$\end{tabular}}}}%
    \put(0,0){\includegraphics[width=\unitlength,page=3]{Cylinder-A.pdf}}%
    \put(0.6194952,0.01892598){\color[rgb]{0,0,0}\makebox(0,0)[lt]{\lineheight{1.25}\smash{\begin{tabular}[t]{l}$D^-$\end{tabular}}}}%
    \put(0.36852425,0.18453395){\color[rgb]{0,0,0}\makebox(0,0)[lt]{\lineheight{1.25}\smash{\begin{tabular}[t]{l}$C$\end{tabular}}}}%
    \put(0.55189279,0.32893676){\color[rgb]{0,0,0}\makebox(0,0)[lt]{\lineheight{1.25}\smash{\begin{tabular}[t]{l}$S$\end{tabular}}}}%
    \put(0.82006931,0.18453395){\color[rgb]{0,0,0}\makebox(0,0)[lt]{\lineheight{1.25}\smash{\begin{tabular}[t]{l}$\partial^-_C$\end{tabular}}}}%
    \put(-0.00064466,0.04846792){\color[rgb]{0,0,0}\makebox(0,0)[lt]{\lineheight{1.25}\smash{\begin{tabular}[t]{l}$N$\end{tabular}}}}%
    \put(0,0){\includegraphics[width=\unitlength,page=4]{Cylinder-A.pdf}}%
  \end{picture}%
\endgroup%

%% file: figures/Nice.pdf_tex
\begingroup%
  \makeatletter%
  \providecommand\color[2][]{%
    \errmessage{(Inkscape) Color is used for the text in Inkscape, but the package 'color.sty' is not loaded}%
    \renewcommand\color[2][]{}%
  }%
  \providecommand\transparent[1]{%
    \errmessage{(Inkscape) Transparency is used (non-zero) for the text in Inkscape, but the package 'transparent.sty' is not loaded}%
    \renewcommand\transparent[1]{}%
  }%
  \providecommand\rotatebox[2]{#2}%
  \newcommand*\fsize{\dimexpr\f@size pt\relax}%
  \newcommand*\lineheight[1]{\fontsize{\fsize}{#1\fsize}\selectfont}%
  \ifx\svgwidth\undefined%
    \setlength{\unitlength}{1091.01498358bp}%
    \ifx\svgscale\undefined%
      \relax%
    \else%
      \setlength{\unitlength}{\unitlength * \real{\svgscale}}%
    \fi%
  \else%
    \setlength{\unitlength}{\svgwidth}%
  \fi%
  \global\let\svgwidth\undefined%
  \global\let\svgscale\undefined%
  \makeatother%
  \begin{picture}(1,0.31396659)%
    \lineheight{1}%
    \setlength\tabcolsep{0pt}%
    \put(0,0){\includegraphics[width=\unitlength,page=1]{Nice.pdf}}%
    \put(0.6811653,0.09351812){\color[rgb]{0,0,0}\makebox(0,0)[lt]{\lineheight{0}\smash{\begin{tabular}[t]{l}$C_i$\end{tabular}}}}%
    \put(0,0){\includegraphics[width=\unitlength,page=2]{Nice.pdf}}%
    \put(-0.00124866,0.08996611){\color[rgb]{0,0,0}\makebox(0,0)[lt]{\lineheight{0}\smash{\begin{tabular}[t]{l}$C_i$\end{tabular}}}}%
    \put(0,0){\includegraphics[width=\unitlength,page=3]{Nice.pdf}}%
  \end{picture}%
\endgroup%

%% file: figures/r16.pdf_tex
\begingroup%
  \makeatletter%
  \providecommand\color[2][]{%
    \errmessage{(Inkscape) Color is used for the text in Inkscape, but the package 'color.sty' is not loaded}%
    \renewcommand\color[2][]{}%
  }%
  \providecommand\transparent[1]{%
    \errmessage{(Inkscape) Transparency is used (non-zero) for the text in Inkscape, but the package 'transparent.sty' is not loaded}%
    \renewcommand\transparent[1]{}%
  }%
  \providecommand\rotatebox[2]{#2}%
  \newcommand*\fsize{\dimexpr\f@size pt\relax}%
  \newcommand*\lineheight[1]{\fontsize{\fsize}{#1\fsize}\selectfont}%
  \ifx\svgwidth\undefined%
    \setlength{\unitlength}{779.47364402bp}%
    \ifx\svgscale\undefined%
      \relax%
    \else%
      \setlength{\unitlength}{\unitlength * \real{\svgscale}}%
    \fi%
  \else%
    \setlength{\unitlength}{\svgwidth}%
  \fi%
  \global\let\svgwidth\undefined%
  \global\let\svgscale\undefined%
  \makeatother%
  \begin{picture}(1,0.32583641)%
    \lineheight{1}%
    \setlength\tabcolsep{0pt}%
    \put(0,0){\includegraphics[width=\unitlength,page=1]{r16.pdf}}%
    \put(-0.00174772,0.02661427){\color[rgb]{0,0,0}\makebox(0,0)[lt]{\lineheight{0}\smash{\begin{tabular}[t]{l}$C_j$\end{tabular}}}}%
    \put(0.60253829,0.0283923){\color[rgb]{0,0,0}\makebox(0,0)[lt]{\lineheight{0}\smash{\begin{tabular}[t]{l}$C_j$\end{tabular}}}}%
    \put(0.15116137,0.10662485){\color[rgb]{0,0,0}\makebox(0,0)[lt]{\lineheight{0}\smash{\begin{tabular}[t]{l}$l$\end{tabular}}}}%
    \put(0.23330241,0.24413894){\color[rgb]{0,0,0}\makebox(0,0)[lt]{\lineheight{0}\smash{\begin{tabular}[t]{l}$l_k$\end{tabular}}}}%
    \put(0.31088978,0.20443719){\color[rgb]{0,0,0}\makebox(0,0)[lt]{\lineheight{0}\smash{\begin{tabular}[t]{l}$C_k$\end{tabular}}}}%
    \put(0.30007532,0.14339934){\color[rgb]{0,0,0}\makebox(0,0)[lt]{\lineheight{0}\smash{\begin{tabular}[t]{l}$l'_k$\end{tabular}}}}%
    \put(0.75768619,0.02515067){\color[rgb]{0,0,0}\makebox(0,0)[lt]{\lineheight{0}\smash{\begin{tabular}[t]{l}$l$\end{tabular}}}}%
    \put(0.92584389,0.02824595){\color[rgb]{0,0,0}\makebox(0,0)[lt]{\lineheight{0}\smash{\begin{tabular}[t]{l}$l'_k$\end{tabular}}}}%
    \put(0.92776826,0.08706673){\color[rgb]{0,0,0}\makebox(0,0)[lt]{\lineheight{0}\smash{\begin{tabular}[t]{l}$C_k$\end{tabular}}}}%
    \put(0,0){\includegraphics[width=\unitlength,page=2]{r16.pdf}}%
    \put(0.86761438,0.15591518){\color[rgb]{0,0,0}\makebox(0,0)[lt]{\lineheight{1.25}\smash{\begin{tabular}[t]{l}$l_k$\end{tabular}}}}%
  \end{picture}%
\endgroup%

%% file: figures/Sliding-Curves.pdf_tex
\begingroup%
  \makeatletter%
  \providecommand\color[2][]{%
    \errmessage{(Inkscape) Color is used for the text in Inkscape, but the package 'color.sty' is not loaded}%
    \renewcommand\color[2][]{}%
  }%
  \providecommand\transparent[1]{%
    \errmessage{(Inkscape) Transparency is used (non-zero) for the text in Inkscape, but the package 'transparent.sty' is not loaded}%
    \renewcommand\transparent[1]{}%
  }%
  \providecommand\rotatebox[2]{#2}%
  \newcommand*\fsize{\dimexpr\f@size pt\relax}%
  \newcommand*\lineheight[1]{\fontsize{\fsize}{#1\fsize}\selectfont}%
  \ifx\svgwidth\undefined%
    \setlength{\unitlength}{1020.74987703bp}%
    \ifx\svgscale\undefined%
      \relax%
    \else%
      \setlength{\unitlength}{\unitlength * \real{\svgscale}}%
    \fi%
  \else%
    \setlength{\unitlength}{\svgwidth}%
  \fi%
  \global\let\svgwidth\undefined%
  \global\let\svgscale\undefined%
  \makeatother%
  \begin{picture}(1,0.26864635)%
    \lineheight{1}%
    \setlength\tabcolsep{0pt}%
    \put(0,0){\includegraphics[width=\unitlength,page=1]{Sliding-Curves.pdf}}%
    \put(0.51558161,0.25507637){\color[rgb]{0,0,0}\makebox(0,0)[lt]{\lineheight{0}\smash{\begin{tabular}[t]{l}$C_j$\end{tabular}}}}%
    \put(0.65943011,0.1905813){\color[rgb]{0,0,0}\makebox(0,0)[lt]{\lineheight{0}\smash{\begin{tabular}[t]{l}$C_k,\ k>j$\end{tabular}}}}%
    \put(0.65665436,0.01157351){\color[rgb]{0,0,0}\makebox(0,0)[lt]{\lineheight{0}\smash{\begin{tabular}[t]{l}$C_i,\ i<j$\end{tabular}}}}%
    \put(0.70661765,0.11258854){\color[rgb]{0,0,0}\makebox(0,0)[lt]{\lineheight{0}\smash{\begin{tabular}[t]{l}slide\end{tabular}}}}%
    \put(0.52963182,0.06094315){\color[rgb]{0,0,0}\makebox(0,0)[lt]{\lineheight{0}\smash{\begin{tabular}[t]{l}$l_i$\end{tabular}}}}%
    \put(0.61895191,0.10695375){\color[rgb]{0,0,0}\makebox(0,0)[lt]{\lineheight{0}\smash{\begin{tabular}[t]{l}$l_k$\end{tabular}}}}%
    \put(0.94010607,0.1905813){\color[rgb]{0,0,0}\makebox(0,0)[lt]{\lineheight{0}\smash{\begin{tabular}[t]{l}$C_k'$\end{tabular}}}}%
    \put(0.89815846,0.10695375){\color[rgb]{0,0,0}\makebox(0,0)[lt]{\lineheight{0}\smash{\begin{tabular}[t]{l}$l_k$\end{tabular}}}}%
    \put(0,0){\includegraphics[width=\unitlength,page=2]{Sliding-Curves.pdf}}%
  \end{picture}%
\endgroup%

%% file: figures/l17-4.pdf_tex
\begingroup%
  \makeatletter%
  \providecommand\color[2][]{%
    \errmessage{(Inkscape) Color is used for the text in Inkscape, but the package 'color.sty' is not loaded}%
    \renewcommand\color[2][]{}%
  }%
  \providecommand\transparent[1]{%
    \errmessage{(Inkscape) Transparency is used (non-zero) for the text in Inkscape, but the package 'transparent.sty' is not loaded}%
    \renewcommand\transparent[1]{}%
  }%
  \providecommand\rotatebox[2]{#2}%
  \newcommand*\fsize{\dimexpr\f@size pt\relax}%
  \newcommand*\lineheight[1]{\fontsize{\fsize}{#1\fsize}\selectfont}%
  \ifx\svgwidth\undefined%
    \setlength{\unitlength}{854.02226772bp}%
    \ifx\svgscale\undefined%
      \relax%
    \else%
      \setlength{\unitlength}{\unitlength * \real{\svgscale}}%
    \fi%
  \else%
    \setlength{\unitlength}{\svgwidth}%
  \fi%
  \global\let\svgwidth\undefined%
  \global\let\svgscale\undefined%
  \makeatother%
  \begin{picture}(1,0.4106354)%
    \lineheight{1}%
    \setlength\tabcolsep{0pt}%
    \put(0,0){\includegraphics[width=\unitlength,page=1]{l17-4.pdf}}%
    \put(-0.00159516,0.13236156){\color[rgb]{0,0,0}\makebox(0,0)[lt]{\lineheight{0}\smash{\begin{tabular}[t]{l}$C_j$\end{tabular}}}}%
    \put(0,0){\includegraphics[width=\unitlength,page=2]{l17-4.pdf}}%
  \end{picture}%
\endgroup%

%% file: figures/F1.pdf_tex
\begingroup%
  \makeatletter%
  \providecommand\color[2][]{%
    \errmessage{(Inkscape) Color is used for the text in Inkscape, but the package 'color.sty' is not loaded}%
    \renewcommand\color[2][]{}%
  }%
  \providecommand\transparent[1]{%
    \errmessage{(Inkscape) Transparency is used (non-zero) for the text in Inkscape, but the package 'transparent.sty' is not loaded}%
    \renewcommand\transparent[1]{}%
  }%
  \providecommand\rotatebox[2]{#2}%
  \newcommand*\fsize{\dimexpr\f@size pt\relax}%
  \newcommand*\lineheight[1]{\fontsize{\fsize}{#1\fsize}\selectfont}%
  \ifx\svgwidth\undefined%
    \setlength{\unitlength}{522.2683816bp}%
    \ifx\svgscale\undefined%
      \relax%
    \else%
      \setlength{\unitlength}{\unitlength * \real{\svgscale}}%
    \fi%
  \else%
    \setlength{\unitlength}{\svgwidth}%
  \fi%
  \global\let\svgwidth\undefined%
  \global\let\svgscale\undefined%
  \makeatother%
  \begin{picture}(1,0.72167941)%
    \lineheight{1}%
    \setlength\tabcolsep{0pt}%
    \put(0,0){\includegraphics[width=\unitlength,page=1]{F1.pdf}}%
    \put(0.39883162,0.1862225){\color[rgb]{0,0,0}\makebox(0,0)[lt]{\lineheight{0}\smash{\begin{tabular}[t]{l} \end{tabular}}}}%
    \put(0.31805418,0.11698468){\color[rgb]{0,0,0}\makebox(0,0)[lt]{\lineheight{0}\smash{\begin{tabular}[t]{l} \end{tabular}}}}%
    \put(0.18206182,0.10045758){\color[rgb]{0,0,0}\makebox(0,0)[lt]{\lineheight{0}\smash{\begin{tabular}[t]{l}$l=\partial D_l$\end{tabular}}}}%
    \put(-0.00326055,0.17852943){\color[rgb]{0,0,0}\makebox(0,0)[lt]{\lineheight{0}\smash{\begin{tabular}[t]{l}$C$\end{tabular}}}}%
    \put(0,0){\includegraphics[width=\unitlength,page=2]{F1.pdf}}%
    \put(0.79922273,0.11660982){\color[rgb]{0,0,0}\makebox(0,0)[lt]{\lineheight{0}\smash{\begin{tabular}[t]{l}$C_i$\end{tabular}}}}%
    \put(0.33467608,0.29806758){\color[rgb]{0,0,0}\makebox(0,0)[lt]{\lineheight{0}\smash{\begin{tabular}[t]{l}$l_i$\end{tabular}}}}%
    \put(0.57981729,0.21464641){\color[rgb]{0,0,0}\makebox(0,0)[lt]{\lineheight{0}\smash{\begin{tabular}[t]{l}$D'_i$\end{tabular}}}}%
    \put(0,0){\includegraphics[width=\unitlength,page=3]{F1.pdf}}%
    \put(0.42681874,0.40825687){\color[rgb]{0,0,0}\makebox(0,0)[lt]{\lineheight{0}\smash{\begin{tabular}[t]{l}$D_{i}$\end{tabular}}}}%
    \put(0,0){\includegraphics[width=\unitlength,page=4]{F1.pdf}}%
  \end{picture}%
\endgroup%

%% file: figures/p23-2.pdf_tex
\begingroup%
  \makeatletter%
  \providecommand\color[2][]{%
    \errmessage{(Inkscape) Color is used for the text in Inkscape, but the package 'color.sty' is not loaded}%
    \renewcommand\color[2][]{}%
  }%
  \providecommand\transparent[1]{%
    \errmessage{(Inkscape) Transparency is used (non-zero) for the text in Inkscape, but the package 'transparent.sty' is not loaded}%
    \renewcommand\transparent[1]{}%
  }%
  \providecommand\rotatebox[2]{#2}%
  \newcommand*\fsize{\dimexpr\f@size pt\relax}%
  \newcommand*\lineheight[1]{\fontsize{\fsize}{#1\fsize}\selectfont}%
  \ifx\svgwidth\undefined%
    \setlength{\unitlength}{634.20824409bp}%
    \ifx\svgscale\undefined%
      \relax%
    \else%
      \setlength{\unitlength}{\unitlength * \real{\svgscale}}%
    \fi%
  \else%
    \setlength{\unitlength}{\svgwidth}%
  \fi%
  \global\let\svgwidth\undefined%
  \global\let\svgscale\undefined%
  \makeatother%
  \begin{picture}(1,0.23450949)%
    \lineheight{1}%
    \setlength\tabcolsep{0pt}%
    \put(0,0){\includegraphics[width=\unitlength,page=1]{p23-2.pdf}}%
    \put(0.20359316,0.17590606){\color[rgb]{0,0,0}\makebox(0,0)[lt]{\lineheight{0}\smash{\begin{tabular}[t]{l}$l^+$\end{tabular}}}}%
    \put(0.15016727,0.13080075){\color[rgb]{0,0,0}\makebox(0,0)[lt]{\lineheight{0}\smash{\begin{tabular}[t]{l}$D^+$\end{tabular}}}}%
    \put(0.09097155,0.00241444){\color[rgb]{0,0,0}\makebox(0,0)[lt]{\lineheight{0}\smash{\begin{tabular}[t]{l}$D_i^{C_n}\times\{1\},\ i=1,2$\end{tabular}}}}%
    \put(0.29231256,0.21651664){\color[rgb]{0,0,0}\makebox(0,0)[lt]{\lineheight{0}\smash{\begin{tabular}[t]{l}$\partial^+_{n-1}$\end{tabular}}}}%
    \put(0,0){\includegraphics[width=\unitlength,page=2]{p23-2.pdf}}%
    \put(-0.00021394,0.17772518){\color[rgb]{0,0,0}\makebox(0,0)[lt]{\lineheight{0}\smash{\begin{tabular}[t]{l}$l$\end{tabular}}}}%
    \put(0.01043265,0.1346569){\color[rgb]{0,0,0}\makebox(0,0)[lt]{\lineheight{0}\smash{\begin{tabular}[t]{l}$D_+$\end{tabular}}}}%
    \put(0.07369354,0.22480248){\color[rgb]{0,0,0}\makebox(0,0)[lt]{\lineheight{0}\smash{\begin{tabular}[t]{l}$C_+$\end{tabular}}}}%
    \put(0,0){\includegraphics[width=\unitlength,page=3]{p23-2.pdf}}%
    \put(0.78226529,0.1786391){\color[rgb]{0,0,0}\makebox(0,0)[lt]{\lineheight{1.25}\smash{\begin{tabular}[t]{l}$l^+$\end{tabular}}}}%
    \put(0.80712553,0.22161477){\color[rgb]{0,0,0}\makebox(0,0)[lt]{\lineheight{1.25}\smash{\begin{tabular}[t]{l}$\partial^+_{n-1}$\end{tabular}}}}%
    \put(0.61605174,0.15008461){\color[rgb]{0,0,0}\makebox(0,0)[lt]{\lineheight{1.25}\smash{\begin{tabular}[t]{l}$l$\end{tabular}}}}%
    \put(0.57628312,0.12539803){\color[rgb]{0,0,0}\makebox(0,0)[lt]{\lineheight{1.25}\smash{\begin{tabular}[t]{l}$D_+$\end{tabular}}}}%
    \put(0.64763529,0.22044021){\color[rgb]{0,0,0}\makebox(0,0)[lt]{\lineheight{1.25}\smash{\begin{tabular}[t]{l}$C_+$\end{tabular}}}}%
    \put(0.69708179,0.12681374){\color[rgb]{0,0,0}\makebox(0,0)[lt]{\lineheight{1.25}\smash{\begin{tabular}[t]{l}$C^+\subset D^+$\end{tabular}}}}%
    \put(0.5042595,0.186814){\color[rgb]{0,0,0}\makebox(0,0)[lt]{\lineheight{1.25}\smash{\begin{tabular}[t]{l}$D_1^{C_n}\times I$\end{tabular}}}}%
    \put(0.83192572,0.10495931){\color[rgb]{0,0,0}\makebox(0,0)[lt]{\lineheight{1.25}\smash{\begin{tabular}[t]{l}$D_2^{C_n}\times\{1\}$\end{tabular}}}}%
    \put(0.63909106,0.00647132){\color[rgb]{0,0,0}\makebox(0,0)[lt]{\lineheight{1.25}\smash{\begin{tabular}[t]{l}$D_1^{C_n}\times\{1\}$\end{tabular}}}}%
    \put(0,0){\includegraphics[width=\unitlength,page=4]{p23-2.pdf}}%
  \end{picture}%
\endgroup%

%% file: figures/p23-7.pdf_tex
\begingroup%
  \makeatletter%
  \providecommand\color[2][]{%
    \errmessage{(Inkscape) Color is used for the text in Inkscape, but the package 'color.sty' is not loaded}%
    \renewcommand\color[2][]{}%
  }%
  \providecommand\transparent[1]{%
    \errmessage{(Inkscape) Transparency is used (non-zero) for the text in Inkscape, but the package 'transparent.sty' is not loaded}%
    \renewcommand\transparent[1]{}%
  }%
  \providecommand\rotatebox[2]{#2}%
  \newcommand*\fsize{\dimexpr\f@size pt\relax}%
  \newcommand*\lineheight[1]{\fontsize{\fsize}{#1\fsize}\selectfont}%
  \ifx\svgwidth\undefined%
    \setlength{\unitlength}{1072.00890541bp}%
    \ifx\svgscale\undefined%
      \relax%
    \else%
      \setlength{\unitlength}{\unitlength * \real{\svgscale}}%
    \fi%
  \else%
    \setlength{\unitlength}{\svgwidth}%
  \fi%
  \global\let\svgwidth\undefined%
  \global\let\svgscale\undefined%
  \makeatother%
  \begin{picture}(1,0.3589056)%
    \lineheight{1}%
    \setlength\tabcolsep{0pt}%
    \put(0,0){\includegraphics[width=\unitlength,page=1]{p23-7.pdf}}%
    \put(0.94060873,0.32907922){\color[rgb]{0,0,0}\makebox(0,0)[lt]{\lineheight{0}\smash{\begin{tabular}[t]{l}$C^+$\end{tabular}}}}%
    \put(0.91721532,0.21985579){\color[rgb]{0,0,0}\makebox(0,0)[lt]{\lineheight{0}\smash{\begin{tabular}[t]{l}$C_+$\end{tabular}}}}%
    \put(0.65640527,0.18526285){\color[rgb]{0,0,0}\makebox(0,0)[lt]{\lineheight{0}\smash{\begin{tabular}[t]{l}$l$\end{tabular}}}}%
    \put(0.65380044,0.12312605){\color[rgb]{0,0,0}\makebox(0,0)[lt]{\lineheight{0}\smash{\begin{tabular}[t]{l}$l'$\end{tabular}}}}%
    \put(0.89495786,0.00890531){\color[rgb]{0,0,0}\makebox(0,0)[lt]{\lineheight{0}\smash{\begin{tabular}[t]{l}$C^-$\end{tabular}}}}%
    \put(0.76407184,0.08114887){\color[rgb]{0,0,0}\makebox(0,0)[lt]{\lineheight{0}\smash{\begin{tabular}[t]{l}$C_-$\end{tabular}}}}%
    \put(0.58097674,0.29252436){\color[rgb]{0,0,0}\makebox(0,0)[lt]{\lineheight{0}\smash{\begin{tabular}[t]{l}$l^+$\end{tabular}}}}%
    \put(0,0){\includegraphics[width=\unitlength,page=2]{p23-7.pdf}}%
    \put(0.77213858,0.03438225){\color[rgb]{0,0,0}\makebox(0,0)[lt]{\lineheight{0}\smash{\begin{tabular}[t]{l}$l^-$\end{tabular}}}}%
    \put(0.39770279,0.34027315){\color[rgb]{0,0,0}\makebox(0,0)[lt]{\lineheight{0}\smash{\begin{tabular}[t]{l}$C^+$\end{tabular}}}}%
    \put(0.37430939,0.23104973){\color[rgb]{0,0,0}\makebox(0,0)[lt]{\lineheight{0}\smash{\begin{tabular}[t]{l}$C_+$\end{tabular}}}}%
    \put(0.38983146,0.01730076){\color[rgb]{0,0,0}\makebox(0,0)[lt]{\lineheight{0}\smash{\begin{tabular}[t]{l}$C^-$\end{tabular}}}}%
    \put(0.37570863,0.10511794){\color[rgb]{0,0,0}\makebox(0,0)[lt]{\lineheight{0}\smash{\begin{tabular}[t]{l}$C_-$\end{tabular}}}}%
    \put(0,0){\includegraphics[width=\unitlength,page=3]{p23-7.pdf}}%
    \put(0.21596666,0.20630494){\color[rgb]{0,0,0}\makebox(0,0)[lt]{\lineheight{1.25}\smash{\begin{tabular}[t]{l}$D_+$\end{tabular}}}}%
    \put(0,0){\includegraphics[width=\unitlength,page=4]{p23-7.pdf}}%
    \put(0.21485615,0.13800862){\color[rgb]{0,0,0}\makebox(0,0)[lt]{\lineheight{1.25}\smash{\begin{tabular}[t]{l}$D_-$\end{tabular}}}}%
    \put(0.19777651,0.30145341){\color[rgb]{0,0,0}\makebox(0,0)[lt]{\lineheight{1.25}\smash{\begin{tabular}[t]{l}$D\times\{1\}$\end{tabular}}}}%
    \put(0.19666601,0.04705789){\color[rgb]{0,0,0}\makebox(0,0)[lt]{\lineheight{1.25}\smash{\begin{tabular}[t]{l}$D\times\{-1\}$\end{tabular}}}}%
    \put(0.11534718,0.19916158){\color[rgb]{0,0,0}\makebox(0,0)[lt]{\lineheight{0}\smash{\begin{tabular}[t]{l}$l$\end{tabular}}}}%
    \put(0.10985892,0.1416643){\color[rgb]{0,0,0}\makebox(0,0)[lt]{\lineheight{0}\smash{\begin{tabular}[t]{l}$l'$\end{tabular}}}}%
    \put(0.04165099,0.04085849){\color[rgb]{0,0,0}\makebox(0,0)[lt]{\lineheight{0}\smash{\begin{tabular}[t]{l}$l^-$\end{tabular}}}}%
    \put(0.04166059,0.30537887){\color[rgb]{0,0,0}\makebox(0,0)[lt]{\lineheight{0}\smash{\begin{tabular}[t]{l}$l^+$\end{tabular}}}}%
    \put(0.38830181,0.17508004){\color[rgb]{0,0,0}\makebox(0,0)[lt]{\lineheight{0}\smash{\begin{tabular}[t]{l}$D\times [-1,1]$\end{tabular}}}}%
    \put(-0.0012708,0.18364754){\color[rgb]{0,0,0}\makebox(0,0)[lt]{\lineheight{0}\smash{\begin{tabular}[t]{l}$N_{n-1}$\end{tabular}}}}%
    \put(0.94756207,0.15197158){\color[rgb]{0,0,0}\makebox(0,0)[lt]{\lineheight{0}\smash{\begin{tabular}[t]{l}$N$\end{tabular}}}}%
    \put(0,0){\includegraphics[width=\unitlength,page=5]{p23-7.pdf}}%
  \end{picture}%
\endgroup%

%% file: figures/l24.pdf_tex
\begingroup%
  \makeatletter%
  \providecommand\color[2][]{%
    \errmessage{(Inkscape) Color is used for the text in Inkscape, but the package 'color.sty' is not loaded}%
    \renewcommand\color[2][]{}%
  }%
  \providecommand\transparent[1]{%
    \errmessage{(Inkscape) Transparency is used (non-zero) for the text in Inkscape, but the package 'transparent.sty' is not loaded}%
    \renewcommand\transparent[1]{}%
  }%
  \providecommand\rotatebox[2]{#2}%
  \newcommand*\fsize{\dimexpr\f@size pt\relax}%
  \newcommand*\lineheight[1]{\fontsize{\fsize}{#1\fsize}\selectfont}%
  \ifx\svgwidth\undefined%
    \setlength{\unitlength}{259.62151162bp}%
    \ifx\svgscale\undefined%
      \relax%
    \else%
      \setlength{\unitlength}{\unitlength * \real{\svgscale}}%
    \fi%
  \else%
    \setlength{\unitlength}{\svgwidth}%
  \fi%
  \global\let\svgwidth\undefined%
  \global\let\svgscale\undefined%
  \makeatother%
  \begin{picture}(1,0.52801523)%
    \lineheight{1}%
    \setlength\tabcolsep{0pt}%
    \put(0,0){\includegraphics[width=\unitlength,page=1]{l24.pdf}}%
    \put(0.37434755,0.4983748){\color[rgb]{0,0,0}\makebox(0,0)[lt]{\lineheight{0}\smash{\begin{tabular}[t]{l}$C_{k'}$\end{tabular}}}}%
    \put(0.53420031,0.41997834){\color[rgb]{0,0,0}\makebox(0,0)[lt]{\lineheight{0}\smash{\begin{tabular}[t]{l}$l'_{k'}$\end{tabular}}}}%
    \put(0,0){\includegraphics[width=\unitlength,page=2]{l24.pdf}}%
    \put(0.34772663,0.25681621){\color[rgb]{0,0,0}\makebox(0,0)[lt]{\lineheight{0}\smash{\begin{tabular}[t]{l}$l_{k'}$\end{tabular}}}}%
    \put(0.59385977,0.0089311){\color[rgb]{0,0,0}\makebox(0,0)[lt]{\lineheight{0}\smash{\begin{tabular}[t]{l}$l_k$\end{tabular}}}}%
    \put(0.78925779,0.26203951){\color[rgb]{0,0,0}\makebox(0,0)[lt]{\lineheight{0}\smash{\begin{tabular}[t]{l}$C_k$\end{tabular}}}}%
    \put(0,0){\includegraphics[width=\unitlength,page=3]{l24.pdf}}%
    \put(0.7760991,0.13041162){\color[rgb]{0,0,0}\makebox(0,0)[lt]{\lineheight{1.25}\smash{\begin{tabular}[t]{l}$C$\end{tabular}}}}%
  \end{picture}%
\endgroup%

%% file: figures/sbc.pdf_tex
\begingroup%
  \makeatletter%
  \providecommand\color[2][]{%
    \errmessage{(Inkscape) Color is used for the text in Inkscape, but the package 'color.sty' is not loaded}%
    \renewcommand\color[2][]{}%
  }%
  \providecommand\transparent[1]{%
    \errmessage{(Inkscape) Transparency is used (non-zero) for the text in Inkscape, but the package 'transparent.sty' is not loaded}%
    \renewcommand\transparent[1]{}%
  }%
  \providecommand\rotatebox[2]{#2}%
  \newcommand*\fsize{\dimexpr\f@size pt\relax}%
  \newcommand*\lineheight[1]{\fontsize{\fsize}{#1\fsize}\selectfont}%
  \ifx\svgwidth\undefined%
    \setlength{\unitlength}{317.2546002bp}%
    \ifx\svgscale\undefined%
      \relax%
    \else%
      \setlength{\unitlength}{\unitlength * \real{\svgscale}}%
    \fi%
  \else%
    \setlength{\unitlength}{\svgwidth}%
  \fi%
  \global\let\svgwidth\undefined%
  \global\let\svgscale\undefined%
  \makeatother%
  \begin{picture}(1,0.9258759)%
    \lineheight{1}%
    \setlength\tabcolsep{0pt}%
    \put(0,0){\includegraphics[width=\unitlength,page=1]{sbc.pdf}}%
    \put(0.45129191,0.14438919){\color[rgb]{0,0,0}\makebox(0,0)[lt]{\lineheight{0}\smash{\begin{tabular}[t]{l}$l$\end{tabular}}}}%
    \put(0.65683676,0.40492127){\color[rgb]{0,0,0}\makebox(0,0)[lt]{\lineheight{0}\smash{\begin{tabular}[t]{l}$\textsf{i}$\end{tabular}}}}%
    \put(0.0484096,0.66036109){\color[rgb]{0,0,0}\makebox(0,0)[lt]{\lineheight{0}\smash{\begin{tabular}[t]{l}$\textsf{ii}$\end{tabular}}}}%
    \put(0.0388291,0.39646045){\color[rgb]{0,0,0}\makebox(0,0)[lt]{\lineheight{0}\smash{\begin{tabular}[t]{l}$\textsf{iv}$\end{tabular}}}}%
    \put(0.03870465,0.16080381){\color[rgb]{0,0,0}\makebox(0,0)[lt]{\lineheight{0}\smash{\begin{tabular}[t]{l}$\textsf{iii}$\end{tabular}}}}%
    \put(0.65977932,0.01000096){\color[rgb]{0,0,0}\makebox(0,0)[lt]{\lineheight{0}\smash{\begin{tabular}[t]{l}$l^-$\end{tabular}}}}%
    \put(0.29317888,0.82822294){\color[rgb]{0,0,0}\makebox(0,0)[lt]{\lineheight{0}\smash{\begin{tabular}[t]{l}$l^+$\end{tabular}}}}%
    \put(0,0){\includegraphics[width=\unitlength,page=2]{sbc.pdf}}%
  \end{picture}%
\endgroup%

%% file: figures/Special.pdf_tex
\begingroup%
  \makeatletter%
  \providecommand\color[2][]{%
    \errmessage{(Inkscape) Color is used for the text in Inkscape, but the package 'color.sty' is not loaded}%
    \renewcommand\color[2][]{}%
  }%
  \providecommand\transparent[1]{%
    \errmessage{(Inkscape) Transparency is used (non-zero) for the text in Inkscape, but the package 'transparent.sty' is not loaded}%
    \renewcommand\transparent[1]{}%
  }%
  \providecommand\rotatebox[2]{#2}%
  \newcommand*\fsize{\dimexpr\f@size pt\relax}%
  \newcommand*\lineheight[1]{\fontsize{\fsize}{#1\fsize}\selectfont}%
  \ifx\svgwidth\undefined%
    \setlength{\unitlength}{2044.64985853bp}%
    \ifx\svgscale\undefined%
      \relax%
    \else%
      \setlength{\unitlength}{\unitlength * \real{\svgscale}}%
    \fi%
  \else%
    \setlength{\unitlength}{\svgwidth}%
  \fi%
  \global\let\svgwidth\undefined%
  \global\let\svgscale\undefined%
  \makeatother%
  \begin{picture}(1,0.24354935)%
    \lineheight{1}%
    \setlength\tabcolsep{0pt}%
    \put(0,0){\includegraphics[width=\unitlength,page=1]{Special.pdf}}%
    \put(0.27034721,0.19592372){\color[rgb]{0,0,0}\makebox(0,0)[lt]{\lineheight{0}\smash{\begin{tabular}[t]{l}$l_s$\end{tabular}}}}%
    \put(0.27332473,0.14013514){\color[rgb]{0,0,0}\makebox(0,0)[lt]{\lineheight{0}\smash{\begin{tabular}[t]{l}$C_s$\end{tabular}}}}%
    \put(0,0){\includegraphics[width=\unitlength,page=2]{Special.pdf}}%
    \put(0.80596578,0.08607241){\color[rgb]{0,0,0}\makebox(0,0)[lt]{\lineheight{0}\smash{\begin{tabular}[t]{l}$D'$\end{tabular}}}}%
    \put(0,0){\includegraphics[width=\unitlength,page=3]{Special.pdf}}%
    \put(-0.00066628,0.05146673){\color[rgb]{0,0,0}\makebox(0,0)[lt]{\lineheight{0}\smash{\begin{tabular}[t]{l}$C_k$\end{tabular}}}}%
    \put(0.16302155,0.11996564){\color[rgb]{0,0,0}\makebox(0,0)[lt]{\lineheight{0}\smash{\begin{tabular}[t]{l}$D_k$\end{tabular}}}}%
    \put(0.14556575,0.20335137){\color[rgb]{0,0,0}\makebox(0,0)[lt]{\lineheight{0}\smash{\begin{tabular}[t]{l}$\bar{C}\subset C_i$\end{tabular}}}}%
    \put(0.07446226,0.06416331){\color[rgb]{0,0,0}\makebox(0,0)[lt]{\lineheight{0}\smash{\begin{tabular}[t]{l}$l_{k}^1$\end{tabular}}}}%
    \put(0.27454095,0.06543045){\color[rgb]{0,0,0}\makebox(0,0)[lt]{\lineheight{0}\smash{\begin{tabular}[t]{l}$l_{k}^2$\end{tabular}}}}%
    \put(0.26802997,0.11244946){\color[rgb]{0,0,0}\makebox(0,0)[lt]{\lineheight{0}\smash{\begin{tabular}[t]{l}$l$\end{tabular}}}}%
    \put(0.16139889,0.05102994){\color[rgb]{0,0,0}\makebox(0,0)[lt]{\lineheight{0}\smash{\begin{tabular}[t]{l}$l_k$\end{tabular}}}}%
    \put(0.87942433,0.18292761){\color[rgb]{0,0,0}\makebox(0,0)[lt]{\lineheight{0}\smash{\begin{tabular}[t]{l}$D$\end{tabular}}}}%
    \put(0.8326897,0.05034575){\color[rgb]{0,0,0}\makebox(0,0)[lt]{\lineheight{0}\smash{\begin{tabular}[t]{l}$C_k$\end{tabular}}}}%
    \put(0.61083121,0.19645761){\color[rgb]{0,0,0}\makebox(0,0)[lt]{\lineheight{0}\smash{\begin{tabular}[t]{l}$\widetilde{C}$\end{tabular}}}}%
    \put(0.48793485,0.04839429){\color[rgb]{0,0,0}\makebox(0,0)[lt]{\lineheight{0}\smash{\begin{tabular}[t]{l}$C_k$\end{tabular}}}}%
    \put(0.16105496,0.00643639){\color[rgb]{0,0,0}\makebox(0,0)[lt]{\lineheight{1.25}\smash{\begin{tabular}[t]{l}$(1)$\end{tabular}}}}%
    \put(0.49118483,0.00643639){\color[rgb]{0,0,0}\makebox(0,0)[lt]{\lineheight{1.25}\smash{\begin{tabular}[t]{l}$(2)$\end{tabular}}}}%
    \put(0.81984746,0.00203466){\color[rgb]{0,0,0}\makebox(0,0)[lt]{\lineheight{1.25}\smash{\begin{tabular}[t]{l}$(3)$\end{tabular}}}}%
  \end{picture}%
\endgroup%

%% file: figures/Special-B.pdf_tex
\begingroup%
  \makeatletter%
  \providecommand\color[2][]{%
    \errmessage{(Inkscape) Color is used for the text in Inkscape, but the package 'color.sty' is not loaded}%
    \renewcommand\color[2][]{}%
  }%
  \providecommand\transparent[1]{%
    \errmessage{(Inkscape) Transparency is used (non-zero) for the text in Inkscape, but the package 'transparent.sty' is not loaded}%
    \renewcommand\transparent[1]{}%
  }%
  \providecommand\rotatebox[2]{#2}%
  \newcommand*\fsize{\dimexpr\f@size pt\relax}%
  \newcommand*\lineheight[1]{\fontsize{\fsize}{#1\fsize}\selectfont}%
  \ifx\svgwidth\undefined%
    \setlength{\unitlength}{1419.87817383bp}%
    \ifx\svgscale\undefined%
      \relax%
    \else%
      \setlength{\unitlength}{\unitlength * \real{\svgscale}}%
    \fi%
  \else%
    \setlength{\unitlength}{\svgwidth}%
  \fi%
  \global\let\svgwidth\undefined%
  \global\let\svgscale\undefined%
  \makeatother%
  \begin{picture}(1,0.23699601)%
    \lineheight{1}%
    \setlength\tabcolsep{0pt}%
    \put(0,0){\includegraphics[width=\unitlength,page=1]{Special-B.pdf}}%
    \put(0.40645629,0.05211961){\color[rgb]{0,0,0}\makebox(0,0)[lt]{\lineheight{0}\smash{\begin{tabular}[t]{l}$C_k$\end{tabular}}}}%
    \put(0.31035261,0.19695777){\color[rgb]{0,0,0}\makebox(0,0)[lt]{\lineheight{0}\smash{\begin{tabular}[t]{l}$\bar{C}\subset C_i$\end{tabular}}}}%
    \put(0.19031484,0.17746938){\color[rgb]{0,0,0}\makebox(0,0)[lt]{\lineheight{0}\smash{\begin{tabular}[t]{l}$\tilde{l}_1$\end{tabular}}}}%
    \put(0.05114344,0.07676846){\color[rgb]{0,0,0}\makebox(0,0)[lt]{\lineheight{0}\smash{\begin{tabular}[t]{l}$l_{k_1}$\end{tabular}}}}%
    \put(0.3325215,0.07452668){\color[rgb]{0,0,0}\makebox(0,0)[lt]{\lineheight{0}\smash{\begin{tabular}[t]{l}$l_{k_2}$\end{tabular}}}}%
    \put(0.16058398,0.05416035){\color[rgb]{0,0,0}\makebox(0,0)[lt]{\lineheight{0}\smash{\begin{tabular}[t]{l}$l'_1=l''_1$\end{tabular}}}}%
    \put(0,0){\includegraphics[width=\unitlength,page=2]{Special-B.pdf}}%
    \put(0.19043573,0.22724055){\color[rgb]{0,0,0}\makebox(0,0)[lt]{\lineheight{0}\smash{\begin{tabular}[t]{l}$\tilde{l}_2$\end{tabular}}}}%
    \put(0.00966518,0.00640321){\color[rgb]{0,0,0}\makebox(0,0)[lt]{\lineheight{0}\smash{\begin{tabular}[t]{l}$l'_2$\end{tabular}}}}%
    \put(0.35840161,0.00223459){\color[rgb]{0,0,0}\makebox(0,0)[lt]{\lineheight{0}\smash{\begin{tabular}[t]{l}$l''_2$\end{tabular}}}}%
    \put(0,0){\includegraphics[width=\unitlength,page=3]{Special-B.pdf}}%
    \put(0.95694228,0.05158256){\color[rgb]{0,0,0}\makebox(0,0)[lt]{\lineheight{0}\smash{\begin{tabular}[t]{l}$C_k'$\end{tabular}}}}%
    \put(0,0){\includegraphics[width=\unitlength,page=4]{Special-B.pdf}}%
  \end{picture}%
\endgroup%

%% file: figures/Special-C.pdf_tex
\begingroup%
  \makeatletter%
  \providecommand\color[2][]{%
    \errmessage{(Inkscape) Color is used for the text in Inkscape, but the package 'color.sty' is not loaded}%
    \renewcommand\color[2][]{}%
  }%
  \providecommand\transparent[1]{%
    \errmessage{(Inkscape) Transparency is used (non-zero) for the text in Inkscape, but the package 'transparent.sty' is not loaded}%
    \renewcommand\transparent[1]{}%
  }%
  \providecommand\rotatebox[2]{#2}%
  \newcommand*\fsize{\dimexpr\f@size pt\relax}%
  \newcommand*\lineheight[1]{\fontsize{\fsize}{#1\fsize}\selectfont}%
  \ifx\svgwidth\undefined%
    \setlength{\unitlength}{908.03012764bp}%
    \ifx\svgscale\undefined%
      \relax%
    \else%
      \setlength{\unitlength}{\unitlength * \real{\svgscale}}%
    \fi%
  \else%
    \setlength{\unitlength}{\svgwidth}%
  \fi%
  \global\let\svgwidth\undefined%
  \global\let\svgscale\undefined%
  \makeatother%
  \begin{picture}(1,0.26495223)%
    \lineheight{1}%
    \setlength\tabcolsep{0pt}%
    \put(0,0){\includegraphics[width=\unitlength,page=1]{Special-C.pdf}}%
    \put(0.89282927,0.24908842){\color[rgb]{0,0,0}\makebox(0,0)[lt]{\lineheight{0}\smash{\begin{tabular}[t]{l}$D'$\end{tabular}}}}%
    \put(0.96165185,0.22468175){\color[rgb]{0,0,0}\makebox(0,0)[lt]{\lineheight{0}\smash{\begin{tabular}[t]{l}$l'$\end{tabular}}}}%
    \put(0.79413229,0.17147101){\color[rgb]{0,0,0}\makebox(0,0)[lt]{\lineheight{0}\smash{\begin{tabular}[t]{l}$l''$\end{tabular}}}}%
    \put(0.88940452,0.11579095){\color[rgb]{0,0,0}\makebox(0,0)[lt]{\lineheight{0}\smash{\begin{tabular}[t]{l}$D$\end{tabular}}}}%
    \put(0.41933091,0.10895371){\color[rgb]{0,0,0}\makebox(0,0)[lt]{\lineheight{0}\smash{\begin{tabular}[t]{l}$\tilde{l}_i$\end{tabular}}}}%
    \put(0.14716858,0.22566738){\color[rgb]{0,0,0}\makebox(0,0)[lt]{\lineheight{0}\smash{\begin{tabular}[t]{l}$\bar{C}\subset C_i$\end{tabular}}}}%
    \put(0.10251623,0.10605457){\color[rgb]{0,0,0}\makebox(0,0)[lt]{\lineheight{0}\smash{\begin{tabular}[t]{l}$l$\end{tabular}}}}%
    \put(0.50923302,0.06293234){\color[rgb]{0,0,0}\makebox(0,0)[lt]{\lineheight{0}\smash{\begin{tabular}[t]{l}$\widetilde{C}$\end{tabular}}}}%
    \put(0.5908659,0.21826778){\color[rgb]{0,0,0}\makebox(0,0)[lt]{\lineheight{0}\smash{\begin{tabular}[t]{l}$\tilde{l}_i^+$\end{tabular}}}}%
    \put(0.0911205,0.0392481){\color[rgb]{0,0,0}\makebox(0,0)[lt]{\lineheight{0}\smash{\begin{tabular}[t]{l}$C_k$\end{tabular}}}}%
    \put(0.03225204,0.11099251){\color[rgb]{0,0,0}\makebox(0,0)[lt]{\lineheight{0}\smash{\begin{tabular}[t]{l}$l_i$\end{tabular}}}}%
    \put(0.03053645,0.19356741){\color[rgb]{0,0,0}\makebox(0,0)[lt]{\lineheight{0}\smash{\begin{tabular}[t]{l}$l_k$\end{tabular}}}}%
    \put(0.1012264,0.19171131){\color[rgb]{0,0,0}\makebox(0,0)[lt]{\lineheight{0}\smash{\begin{tabular}[t]{l}$D_{k}$\end{tabular}}}}%
    \put(0.22658618,0.13595523){\color[rgb]{0,0,0}\makebox(0,0)[lt]{\lineheight{0}\smash{\begin{tabular}[t]{l}$C_s$\end{tabular}}}}%
    \put(0.48758313,0.0392481){\color[rgb]{0,0,0}\makebox(0,0)[lt]{\lineheight{0}\smash{\begin{tabular}[t]{l}$C_k$\end{tabular}}}}%
    \put(0.02728002,0.00229075){\color[rgb]{0,0,0}\makebox(0,0)[lt]{\lineheight{1.25}\smash{\begin{tabular}[t]{l}$(1)$\end{tabular}}}}%
    \put(0.42374264,0.00229075){\color[rgb]{0,0,0}\makebox(0,0)[lt]{\lineheight{1.25}\smash{\begin{tabular}[t]{l}$(2)$\end{tabular}}}}%
    \put(0.80368598,0.00229075){\color[rgb]{0,0,0}\makebox(0,0)[lt]{\lineheight{1.25}\smash{\begin{tabular}[t]{l}$(3)$\end{tabular}}}}%
    \put(0,0){\includegraphics[width=\unitlength,page=2]{Special-C.pdf}}%
    \put(0.18212792,0.03922964){\color[rgb]{0,0,0}\makebox(0,0)[lt]{\lineheight{0}\smash{\begin{tabular}[t]{l}$l_s$\end{tabular}}}}%
    \put(0,0){\includegraphics[width=\unitlength,page=3]{Special-C.pdf}}%
  \end{picture}%
\endgroup%

%% file: figures/Special-D.pdf_tex
\begingroup%
  \makeatletter%
  \providecommand\color[2][]{%
    \errmessage{(Inkscape) Color is used for the text in Inkscape, but the package 'color.sty' is not loaded}%
    \renewcommand\color[2][]{}%
  }%
  \providecommand\transparent[1]{%
    \errmessage{(Inkscape) Transparency is used (non-zero) for the text in Inkscape, but the package 'transparent.sty' is not loaded}%
    \renewcommand\transparent[1]{}%
  }%
  \providecommand\rotatebox[2]{#2}%
  \newcommand*\fsize{\dimexpr\f@size pt\relax}%
  \newcommand*\lineheight[1]{\fontsize{\fsize}{#1\fsize}\selectfont}%
  \ifx\svgwidth\undefined%
    \setlength{\unitlength}{927.52915095bp}%
    \ifx\svgscale\undefined%
      \relax%
    \else%
      \setlength{\unitlength}{\unitlength * \real{\svgscale}}%
    \fi%
  \else%
    \setlength{\unitlength}{\svgwidth}%
  \fi%
  \global\let\svgwidth\undefined%
  \global\let\svgscale\undefined%
  \makeatother%
  \begin{picture}(1,0.29837827)%
    \lineheight{1}%
    \setlength\tabcolsep{0pt}%
    \put(0,0){\includegraphics[width=\unitlength,page=1]{Special-D.pdf}}%
    \put(0.62449807,0.271483){\color[rgb]{0,0,0}\makebox(0,0)[lt]{\lineheight{0}\smash{\begin{tabular}[t]{l}$C_k$\end{tabular}}}}%
    \put(0,0){\includegraphics[width=\unitlength,page=2]{Special-D.pdf}}%
    \put(0.13598874,0.23566129){\color[rgb]{0,0,0}\makebox(0,0)[lt]{\lineheight{0}\smash{\begin{tabular}[t]{l}$\bar{C}\subset C_i$\end{tabular}}}}%
    \put(0.09405652,0.2682486){\color[rgb]{0,0,0}\makebox(0,0)[lt]{\lineheight{0}\smash{\begin{tabular}[t]{l}$C_k$\end{tabular}}}}%
    \put(0.04936322,0.17676477){\color[rgb]{0,0,0}\makebox(0,0)[lt]{\lineheight{0}\smash{\begin{tabular}[t]{l}$l_i$\end{tabular}}}}%
    \put(0.03312889,0.21232215){\color[rgb]{0,0,0}\makebox(0,0)[lt]{\lineheight{0}\smash{\begin{tabular}[t]{l}$l_1'$\end{tabular}}}}%
    \put(0.02827729,0.04736776){\color[rgb]{0,0,0}\makebox(0,0)[lt]{\lineheight{0}\smash{\begin{tabular}[t]{l}$l_2'$\end{tabular}}}}%
    \put(0.14309849,0.18806415){\color[rgb]{0,0,0}\makebox(0,0)[lt]{\lineheight{0}\smash{\begin{tabular}[t]{l}$\tilde{l}_1$\end{tabular}}}}%
    \put(0.09857798,0.10853605){\color[rgb]{0,0,0}\makebox(0,0)[lt]{\lineheight{0}\smash{\begin{tabular}[t]{l}$\tilde{l}_2$\end{tabular}}}}%
    \put(0.3334021,0.2682486){\color[rgb]{0,0,0}\makebox(0,0)[lt]{\lineheight{0}\smash{\begin{tabular}[t]{l}$C_k'$\end{tabular}}}}%
    \put(0.26762286,0.20908775){\color[rgb]{0,0,0}\makebox(0,0)[lt]{\lineheight{0}\smash{\begin{tabular}[t]{l}$l^*_1$\end{tabular}}}}%
    \put(0.26277126,0.04736776){\color[rgb]{0,0,0}\makebox(0,0)[lt]{\lineheight{0}\smash{\begin{tabular}[t]{l}$l^*_2$\end{tabular}}}}%
    \put(0.66966468,0.23566129){\color[rgb]{0,0,0}\makebox(0,0)[lt]{\lineheight{0}\smash{\begin{tabular}[t]{l}$\bar{C}\subset C_i$\end{tabular}}}}%
    \put(0.58303917,0.17838197){\color[rgb]{0,0,0}\makebox(0,0)[lt]{\lineheight{0}\smash{\begin{tabular}[t]{l}$l_i$\end{tabular}}}}%
    \put(0.57003923,0.22202535){\color[rgb]{0,0,0}\makebox(0,0)[lt]{\lineheight{0}\smash{\begin{tabular}[t]{l}$l_1'$\end{tabular}}}}%
    \put(0.56195323,0.04089896){\color[rgb]{0,0,0}\makebox(0,0)[lt]{\lineheight{0}\smash{\begin{tabular}[t]{l}$l_1''$\end{tabular}}}}%
    \put(0.63228248,0.16223949){\color[rgb]{0,0,0}\makebox(0,0)[lt]{\lineheight{0}\smash{\begin{tabular}[t]{l}$\tilde{l}_1$\end{tabular}}}}%
    \put(0.63363736,0.12519437){\color[rgb]{0,0,0}\makebox(0,0)[lt]{\lineheight{0}\smash{\begin{tabular}[t]{l}$\tilde{l}_2$\end{tabular}}}}%
    \put(0.86707805,0.2682486){\color[rgb]{0,0,0}\makebox(0,0)[lt]{\lineheight{0}\smash{\begin{tabular}[t]{l}$C_k'$\end{tabular}}}}%
    \put(0.80129886,0.20908775){\color[rgb]{0,0,0}\makebox(0,0)[lt]{\lineheight{0}\smash{\begin{tabular}[t]{l}$l^*_1$\end{tabular}}}}%
    \put(0.75620678,0.15128549){\color[rgb]{0,0,0}\makebox(0,0)[lt]{\lineheight{0}\smash{\begin{tabular}[t]{l}$l^*_2$\end{tabular}}}}%
    \put(0.48756204,0.16057175){\color[rgb]{0,0,0}\makebox(0,0)[lt]{\lineheight{0}\smash{\begin{tabular}[t]{l}$l_2'=l_2''$\end{tabular}}}}%
    \put(0,0){\includegraphics[width=\unitlength,page=3]{Special-D.pdf}}%
    \put(0.02832373,0.01051181){\color[rgb]{0,0,0}\makebox(0,0)[lt]{\lineheight{1.25}\smash{\begin{tabular}[t]{l}$(1)$\end{tabular}}}}%
    \put(0.2709037,0.01051181){\color[rgb]{0,0,0}\makebox(0,0)[lt]{\lineheight{1.25}\smash{\begin{tabular}[t]{l}$(2)$\end{tabular}}}}%
    \put(0.56199968,0.01051181){\color[rgb]{0,0,0}\makebox(0,0)[lt]{\lineheight{1.25}\smash{\begin{tabular}[t]{l}$(3)$\end{tabular}}}}%
    \put(0.80457965,0.01051181){\color[rgb]{0,0,0}\makebox(0,0)[lt]{\lineheight{1.25}\smash{\begin{tabular}[t]{l}$(4)$\end{tabular}}}}%
  \end{picture}%
\endgroup%

%% file: figures/special2-4.pdf_tex
\begingroup%
  \makeatletter%
  \providecommand\color[2][]{%
    \errmessage{(Inkscape) Color is used for the text in Inkscape, but the package 'color.sty' is not loaded}%
    \renewcommand\color[2][]{}%
  }%
  \providecommand\transparent[1]{%
    \errmessage{(Inkscape) Transparency is used (non-zero) for the text in Inkscape, but the package 'transparent.sty' is not loaded}%
    \renewcommand\transparent[1]{}%
  }%
  \providecommand\rotatebox[2]{#2}%
  \newcommand*\fsize{\dimexpr\f@size pt\relax}%
  \newcommand*\lineheight[1]{\fontsize{\fsize}{#1\fsize}\selectfont}%
  \ifx\svgwidth\undefined%
    \setlength{\unitlength}{697.16425326bp}%
    \ifx\svgscale\undefined%
      \relax%
    \else%
      \setlength{\unitlength}{\unitlength * \real{\svgscale}}%
    \fi%
  \else%
    \setlength{\unitlength}{\svgwidth}%
  \fi%
  \global\let\svgwidth\undefined%
  \global\let\svgscale\undefined%
  \makeatother%
  \begin{picture}(1,0.23201811)%
    \lineheight{1}%
    \setlength\tabcolsep{0pt}%
    \put(0,0){\includegraphics[width=\unitlength,page=1]{special2-4.pdf}}%
    \put(0.08221871,0.01689568){\color[rgb]{0,0,0}\makebox(0,0)[lt]{\lineheight{1.25}\smash{\begin{tabular}[t]{l}$\tilde{l}_2$\end{tabular}}}}%
    \put(0.14125403,0.03649351){\color[rgb]{0,0,0}\makebox(0,0)[lt]{\lineheight{1.25}\smash{\begin{tabular}[t]{l}$C$\end{tabular}}}}%
    \put(0.01519052,0.01764522){\color[rgb]{0,0,0}\makebox(0,0)[lt]{\lineheight{1.25}\smash{\begin{tabular}[t]{l}$\tilde{l}_1$\end{tabular}}}}%
    \put(0.04770319,0.19679585){\color[rgb]{0,0,0}\makebox(0,0)[lt]{\lineheight{1.25}\smash{\begin{tabular}[t]{l}$l_1$\end{tabular}}}}%
    \put(0.07978809,0.19524823){\color[rgb]{0,0,0}\makebox(0,0)[lt]{\lineheight{1.25}\smash{\begin{tabular}[t]{l}$l_2$\end{tabular}}}}%
    \put(0.07144601,0.0826492){\color[rgb]{0,0,0}\makebox(0,0)[lt]{\lineheight{1.25}\smash{\begin{tabular}[t]{l}$l_3$\end{tabular}}}}%
    \put(0.20522103,0.16739104){\color[rgb]{0,0,0}\makebox(0,0)[lt]{\lineheight{1.25}\smash{\begin{tabular}[t]{l}$l_1$\end{tabular}}}}%
    \put(0.37508211,0.16644736){\color[rgb]{0,0,0}\makebox(0,0)[lt]{\lineheight{1.25}\smash{\begin{tabular}[t]{l}$l_2$\end{tabular}}}}%
    \put(0.27694011,0.06913581){\color[rgb]{0,0,0}\makebox(0,0)[lt]{\lineheight{1.25}\smash{\begin{tabular}[t]{l}$l_3$\end{tabular}}}}%
    \put(0.24404073,0.19945603){\color[rgb]{0,0,0}\makebox(0,0)[lt]{\lineheight{1.25}\smash{\begin{tabular}[t]{l}$\tilde{l}_1$\end{tabular}}}}%
    \put(0.30065266,0.2017084){\color[rgb]{0,0,0}\makebox(0,0)[lt]{\lineheight{1.25}\smash{\begin{tabular}[t]{l}$\tilde{l}_2$\end{tabular}}}}%
    \put(0.32846464,0.16138925){\color[rgb]{0,0,0}\makebox(0,0)[lt]{\lineheight{1.25}\smash{\begin{tabular}[t]{l}$C_h$\end{tabular}}}}%
    \put(0.40678951,0.04339244){\color[rgb]{0,0,0}\makebox(0,0)[lt]{\lineheight{1.25}\smash{\begin{tabular}[t]{l}$C_k$\end{tabular}}}}%
    \put(-0.18295568,0.49802513){\color[rgb]{0,0,0}\makebox(0,0)[lt]{\begin{minipage}{1.48835175\unitlength}\end{minipage}}}%
    \put(0,0){\includegraphics[width=\unitlength,page=2]{special2-4.pdf}}%
    \put(0.6798684,0.03892665){\color[rgb]{0,0,0}\makebox(0,0)[lt]{\lineheight{1.25}\smash{\begin{tabular}[t]{l}$C$\end{tabular}}}}%
    \put(0.55380508,0.02007836){\color[rgb]{0,0,0}\makebox(0,0)[lt]{\lineheight{1.25}\smash{\begin{tabular}[t]{l}$\tilde{l}_1$\end{tabular}}}}%
    \put(0.61202964,0.01905421){\color[rgb]{0,0,0}\makebox(0,0)[lt]{\lineheight{1.25}\smash{\begin{tabular}[t]{l}$\tilde{l}_2$\end{tabular}}}}%
    \put(0.58631772,0.19922898){\color[rgb]{0,0,0}\makebox(0,0)[lt]{\lineheight{1.25}\smash{\begin{tabular}[t]{l}$l'_1$\end{tabular}}}}%
    \put(0.61006053,0.08508239){\color[rgb]{0,0,0}\makebox(0,0)[lt]{\lineheight{1.25}\smash{\begin{tabular}[t]{l}$l'_3$\end{tabular}}}}%
    \put(0,0){\includegraphics[width=\unitlength,page=3]{special2-4.pdf}}%
    \put(0.74971306,0.16853793){\color[rgb]{0,0,0}\makebox(0,0)[lt]{\lineheight{1.25}\smash{\begin{tabular}[t]{l}$l'_1$\end{tabular}}}}%
    \put(0.82143212,0.07028271){\color[rgb]{0,0,0}\makebox(0,0)[lt]{\lineheight{1.25}\smash{\begin{tabular}[t]{l}$l'_3$\end{tabular}}}}%
    \put(0.95128156,0.04453934){\color[rgb]{0,0,0}\makebox(0,0)[lt]{\lineheight{1.25}\smash{\begin{tabular}[t]{l}$C'_k$\end{tabular}}}}%
    \put(0,0){\includegraphics[width=\unitlength,page=4]{special2-4.pdf}}%
  \end{picture}%
\endgroup%

%% file: figures/special2-5.pdf_tex
\begingroup%
  \makeatletter%
  \providecommand\color[2][]{%
    \errmessage{(Inkscape) Color is used for the text in Inkscape, but the package 'color.sty' is not loaded}%
    \renewcommand\color[2][]{}%
  }%
  \providecommand\transparent[1]{%
    \errmessage{(Inkscape) Transparency is used (non-zero) for the text in Inkscape, but the package 'transparent.sty' is not loaded}%
    \renewcommand\transparent[1]{}%
  }%
  \providecommand\rotatebox[2]{#2}%
  \newcommand*\fsize{\dimexpr\f@size pt\relax}%
  \newcommand*\lineheight[1]{\fontsize{\fsize}{#1\fsize}\selectfont}%
  \ifx\svgwidth\undefined%
    \setlength{\unitlength}{736.65276705bp}%
    \ifx\svgscale\undefined%
      \relax%
    \else%
      \setlength{\unitlength}{\unitlength * \real{\svgscale}}%
    \fi%
  \else%
    \setlength{\unitlength}{\svgwidth}%
  \fi%
  \global\let\svgwidth\undefined%
  \global\let\svgscale\undefined%
  \makeatother%
  \begin{picture}(1,0.2734917)%
    \lineheight{1}%
    \setlength\tabcolsep{0pt}%
    \put(0,0){\includegraphics[width=\unitlength,page=1]{special2-5.pdf}}%
    \put(0.07846962,0.00837791){\color[rgb]{0,0,0}\makebox(0,0)[lt]{\lineheight{0}\smash{\begin{tabular}[t]{l}$\tilde{l}_1$\end{tabular}}}}%
    \put(0.13434035,0.03506996){\color[rgb]{0,0,0}\makebox(0,0)[lt]{\lineheight{0}\smash{\begin{tabular}[t]{l}$C$\end{tabular}}}}%
    \put(0.01910701,0.01112302){\color[rgb]{0,0,0}\makebox(0,0)[lt]{\lineheight{0}\smash{\begin{tabular}[t]{l}$\tilde{l}_2$\end{tabular}}}}%
    \put(0.2771022,0.23398747){\color[rgb]{0,0,0}\makebox(0,0)[lt]{\lineheight{0}\smash{\begin{tabular}[t]{l}$l_1$\end{tabular}}}}%
    \put(0.31849229,0.05080008){\color[rgb]{0,0,0}\makebox(0,0)[lt]{\lineheight{0}\smash{\begin{tabular}[t]{l}$l_2$\end{tabular}}}}%
    \put(0.05165851,0.14529111){\color[rgb]{0,0,0}\makebox(0,0)[lt]{\lineheight{0}\smash{\begin{tabular}[t]{l}$l_1$\end{tabular}}}}%
    \put(0.05147993,0.10284723){\color[rgb]{0,0,0}\makebox(0,0)[lt]{\lineheight{0}\smash{\begin{tabular}[t]{l}$l_2$\end{tabular}}}}%
    \put(0.41948206,0.15306501){\color[rgb]{0,0,0}\makebox(0,0)[lt]{\lineheight{0}\smash{\begin{tabular}[t]{l}$\tilde{l}_1$\end{tabular}}}}%
    \put(0.42023459,0.20575303){\color[rgb]{0,0,0}\makebox(0,0)[lt]{\lineheight{0}\smash{\begin{tabular}[t]{l}$C_h$\end{tabular}}}}%
    \put(0.39471386,0.03919939){\color[rgb]{0,0,0}\makebox(0,0)[lt]{\lineheight{0}\smash{\begin{tabular}[t]{l}$C_k$\end{tabular}}}}%
    \put(0.64311457,0.01912992){\color[rgb]{0,0,0}\makebox(0,0)[lt]{\lineheight{0}\smash{\begin{tabular}[t]{l}$\tilde{l}_1$\end{tabular}}}}%
    \put(0.69694905,0.03767693){\color[rgb]{0,0,0}\makebox(0,0)[lt]{\lineheight{0}\smash{\begin{tabular}[t]{l}$C$\end{tabular}}}}%
    \put(0.57967874,0.01983881){\color[rgb]{0,0,0}\makebox(0,0)[lt]{\lineheight{0}\smash{\begin{tabular}[t]{l}$\tilde{l}_2$\end{tabular}}}}%
    \put(0.61095845,0.14967982){\color[rgb]{0,0,0}\makebox(0,0)[lt]{\lineheight{0}\smash{\begin{tabular}[t]{l}$l'_1$\end{tabular}}}}%
    \put(0.85869029,0.17223758){\color[rgb]{0,0,0}\makebox(0,0)[lt]{\lineheight{0}\smash{\begin{tabular}[t]{l}$\tilde{l}_2$\end{tabular}}}}%
    \put(0.95056556,0.03872774){\color[rgb]{0,0,0}\makebox(0,0)[lt]{\lineheight{0}\smash{\begin{tabular}[t]{l}$C'_k$\end{tabular}}}}%
    \put(0.26430393,0.1312509){\color[rgb]{0,0,0}\makebox(0,0)[lt]{\lineheight{0}\smash{\begin{tabular}[t]{l}$C_j$\end{tabular}}}}%
    \put(0.28661524,0.17464237){\color[rgb]{0,0,0}\makebox(0,0)[lt]{\lineheight{0}\smash{\begin{tabular}[t]{l}$\tilde{l}_2$\end{tabular}}}}%
    \put(0.82079513,0.12082811){\color[rgb]{0,0,0}\makebox(0,0)[lt]{\lineheight{0}\smash{\begin{tabular}[t]{l}$C_j'$\end{tabular}}}}%
    \put(0.87993799,0.04641117){\color[rgb]{0,0,0}\makebox(0,0)[lt]{\lineheight{0}\smash{\begin{tabular}[t]{l}$l'_1$\end{tabular}}}}%
    \put(0,0){\includegraphics[width=\unitlength,page=2]{special2-5.pdf}}%
  \end{picture}%
\endgroup%

%% file: figures/Local-A.pdf_tex
\begingroup%
  \makeatletter%
  \providecommand\color[2][]{%
    \errmessage{(Inkscape) Color is used for the text in Inkscape, but the package 'color.sty' is not loaded}%
    \renewcommand\color[2][]{}%
  }%
  \providecommand\transparent[1]{%
    \errmessage{(Inkscape) Transparency is used (non-zero) for the text in Inkscape, but the package 'transparent.sty' is not loaded}%
    \renewcommand\transparent[1]{}%
  }%
  \providecommand\rotatebox[2]{#2}%
  \newcommand*\fsize{\dimexpr\f@size pt\relax}%
  \newcommand*\lineheight[1]{\fontsize{\fsize}{#1\fsize}\selectfont}%
  \ifx\svgwidth\undefined%
    \setlength{\unitlength}{747.72526683bp}%
    \ifx\svgscale\undefined%
      \relax%
    \else%
      \setlength{\unitlength}{\unitlength * \real{\svgscale}}%
    \fi%
  \else%
    \setlength{\unitlength}{\svgwidth}%
  \fi%
  \global\let\svgwidth\undefined%
  \global\let\svgscale\undefined%
  \makeatother%
  \begin{picture}(1,0.28248353)%
    \lineheight{1}%
    \setlength\tabcolsep{0pt}%
    \put(0,0){\includegraphics[width=\unitlength,page=1]{Local-A.pdf}}%
    \put(0.02386176,0.03420957){\color[rgb]{0,0,0}\makebox(0,0)[lt]{\lineheight{0}\smash{\begin{tabular}[t]{l}$C_k$\end{tabular}}}}%
    \put(0.14210592,0.05288408){\color[rgb]{0,0,0}\makebox(0,0)[lt]{\lineheight{0}\smash{\begin{tabular}[t]{l}$D_{k}$\end{tabular}}}}%
    \put(0.08205672,0.1078694){\color[rgb]{0,0,0}\makebox(0,0)[lt]{\lineheight{0}\smash{\begin{tabular}[t]{l}$l_k$\end{tabular}}}}%
    \put(0.08982246,0.21017978){\color[rgb]{0,0,0}\makebox(0,0)[lt]{\lineheight{0}\smash{\begin{tabular}[t]{l}$l$\end{tabular}}}}%
    \put(0.14404151,0.2553167){\color[rgb]{0,0,0}\makebox(0,0)[lt]{\lineheight{0}\smash{\begin{tabular}[t]{l}$l_i$\end{tabular}}}}%
    \put(0.1969817,0.15558001){\color[rgb]{0,0,0}\makebox(0,0)[lt]{\lineheight{0}\smash{\begin{tabular}[t]{l}$l_s$\end{tabular}}}}%
    \put(0.16830575,0.20388996){\color[rgb]{0,0,0}\makebox(0,0)[lt]{\lineheight{0}\smash{\begin{tabular}[t]{l}$C_s$\end{tabular}}}}%
    \put(0.21310246,0.22440463){\color[rgb]{0,0,0}\makebox(0,0)[lt]{\lineheight{0}\smash{\begin{tabular}[t]{l}$\bar{C}\subset C_i$\end{tabular}}}}%
    \put(0.38355845,0.03819791){\color[rgb]{0,0,0}\makebox(0,0)[lt]{\lineheight{0}\smash{\begin{tabular}[t]{l}$C_k$\end{tabular}}}}%
    \put(0.54593277,0.26869162){\color[rgb]{0,0,0}\makebox(0,0)[lt]{\lineheight{0}\smash{\begin{tabular}[t]{l}$\widetilde{C}$\end{tabular}}}}%
    \put(0.64387779,0.19070958){\color[rgb]{0,0,0}\makebox(0,0)[lt]{\lineheight{0}\smash{\begin{tabular}[t]{l}$C_j$\end{tabular}}}}%
    \put(0.82384481,0.24451593){\color[rgb]{0,0,0}\makebox(0,0)[lt]{\lineheight{0}\smash{\begin{tabular}[t]{l}$l''$\end{tabular}}}}%
    \put(0.92386111,0.13898027){\color[rgb]{0,0,0}\makebox(0,0)[lt]{\lineheight{0}\smash{\begin{tabular}[t]{l}$D$\end{tabular}}}}%
    \put(0.93307663,0.02216096){\color[rgb]{0,0,0}\makebox(0,0)[lt]{\lineheight{0}\smash{\begin{tabular}[t]{l}$l'$\end{tabular}}}}%
    \put(0.80400486,0.03709281){\color[rgb]{0,0,0}\makebox(0,0)[lt]{\lineheight{0}\smash{\begin{tabular}[t]{l}$C_k$\end{tabular}}}}%
    \put(0.75860071,0.13385606){\color[rgb]{0,0,0}\makebox(0,0)[lt]{\lineheight{0}\smash{\begin{tabular}[t]{l}$D'$\end{tabular}}}}%
    \put(0.16000533,0.00278188){\color[rgb]{0,0,0}\makebox(0,0)[lt]{\lineheight{1.25}\smash{\begin{tabular}[t]{l}$(1)$\end{tabular}}}}%
    \put(0.54116162,0.00278188){\color[rgb]{0,0,0}\makebox(0,0)[lt]{\lineheight{1.25}\smash{\begin{tabular}[t]{l}$(2)$\end{tabular}}}}%
    \put(0.88219696,0.00278188){\color[rgb]{0,0,0}\makebox(0,0)[lt]{\lineheight{1.25}\smash{\begin{tabular}[t]{l}$(3)$\end{tabular}}}}%
    \put(0,0){\includegraphics[width=\unitlength,page=2]{Local-A.pdf}}%
  \end{picture}%
\endgroup%

%% file: figures/type-3-8.pdf_tex
\begingroup%
  \makeatletter%
  \providecommand\color[2][]{%
    \errmessage{(Inkscape) Color is used for the text in Inkscape, but the package 'color.sty' is not loaded}%
    \renewcommand\color[2][]{}%
  }%
  \providecommand\transparent[1]{%
    \errmessage{(Inkscape) Transparency is used (non-zero) for the text in Inkscape, but the package 'transparent.sty' is not loaded}%
    \renewcommand\transparent[1]{}%
  }%
  \providecommand\rotatebox[2]{#2}%
  \newcommand*\fsize{\dimexpr\f@size pt\relax}%
  \newcommand*\lineheight[1]{\fontsize{\fsize}{#1\fsize}\selectfont}%
  \ifx\svgwidth\undefined%
    \setlength{\unitlength}{867.76572116bp}%
    \ifx\svgscale\undefined%
      \relax%
    \else%
      \setlength{\unitlength}{\unitlength * \real{\svgscale}}%
    \fi%
  \else%
    \setlength{\unitlength}{\svgwidth}%
  \fi%
  \global\let\svgwidth\undefined%
  \global\let\svgscale\undefined%
  \makeatother%
  \begin{picture}(1,0.52318864)%
    \lineheight{1}%
    \setlength\tabcolsep{0pt}%
    \put(0,0){\includegraphics[width=\unitlength,page=1]{type-3-8.pdf}}%
    \put(0.2662246,0.37552284){\color[rgb]{0,0,0}\makebox(0,0)[lt]{\lineheight{0}\smash{\begin{tabular}[t]{l}$l'$\end{tabular}}}}%
    \put(0.41625226,0.33998161){\color[rgb]{0,0,0}\makebox(0,0)[lt]{\lineheight{0}\smash{\begin{tabular}[t]{l}$l_2$\end{tabular}}}}%
    \put(0.27431491,0.28844318){\color[rgb]{0,0,0}\makebox(0,0)[lt]{\lineheight{0}\smash{\begin{tabular}[t]{l}$C_k$\end{tabular}}}}%
    \put(0.44332516,0.29312412){\color[rgb]{0,0,0}\makebox(0,0)[lt]{\lineheight{0}\smash{\begin{tabular}[t]{l}$C_i$\end{tabular}}}}%
    \put(0.20795607,0.48381903){\color[rgb]{0,0,0}\makebox(0,0)[lt]{\lineheight{0}\smash{\begin{tabular}[t]{l}$C_j$\end{tabular}}}}%
    \put(0.16675165,0.43758967){\color[rgb]{0,0,0}\makebox(0,0)[lt]{\lineheight{0}\smash{\begin{tabular}[t]{l}$l_1$\end{tabular}}}}%
    \put(0.08678233,0.28272841){\color[rgb]{0,0,0}\makebox(0,0)[lt]{\lineheight{0}\smash{\begin{tabular}[t]{l}$l_2$\end{tabular}}}}%
    \put(-0.00151366,0.51044831){\color[rgb]{0,0,0}\makebox(0,0)[lt]{\lineheight{0}\smash{\begin{tabular}[t]{l}$C$\end{tabular}}}}%
    \put(0.03981114,0.28665669){\color[rgb]{0,0,0}\makebox(0,0)[lt]{\lineheight{0}\smash{\begin{tabular}[t]{l}$l_1$\end{tabular}}}}%
    \put(0.05995665,0.4284751){\color[rgb]{0,0,0}\makebox(0,0)[lt]{\lineheight{0}\smash{\begin{tabular}[t]{l}$l'$\end{tabular}}}}%
    \put(0,0){\includegraphics[width=\unitlength,page=2]{type-3-8.pdf}}%
    \put(0.21609586,0.09549328){\color[rgb]{0,0,0}\makebox(0,0)[lt]{\lineheight{0}\smash{\begin{tabular}[t]{l}$l'$\end{tabular}}}}%
    \put(0.22418617,0.00841363){\color[rgb]{0,0,0}\makebox(0,0)[lt]{\lineheight{0}\smash{\begin{tabular}[t]{l}$C_k$\end{tabular}}}}%
    \put(0.39319642,0.0753234){\color[rgb]{0,0,0}\makebox(0,0)[lt]{\lineheight{0}\smash{\begin{tabular}[t]{l}$C_i$\end{tabular}}}}%
    \put(0.0885109,0.00270085){\color[rgb]{0,0,0}\makebox(0,0)[lt]{\lineheight{0}\smash{\begin{tabular}[t]{l}$l_2$\end{tabular}}}}%
    \put(0.00021488,0.23042451){\color[rgb]{0,0,0}\makebox(0,0)[lt]{\lineheight{0}\smash{\begin{tabular}[t]{l}$C$\end{tabular}}}}%
    \put(0.04153971,0.00663452){\color[rgb]{0,0,0}\makebox(0,0)[lt]{\lineheight{0}\smash{\begin{tabular}[t]{l}$l_1$\end{tabular}}}}%
    \put(0.06168522,0.14844311){\color[rgb]{0,0,0}\makebox(0,0)[lt]{\lineheight{0}\smash{\begin{tabular}[t]{l}$l'$\end{tabular}}}}%
    \put(0,0){\includegraphics[width=\unitlength,page=3]{type-3-8.pdf}}%
    \put(0.3713227,0.00382969){\color[rgb]{0,0,0}\makebox(0,0)[lt]{\lineheight{0}\smash{\begin{tabular}[t]{l}$l_2$\end{tabular}}}}%
    \put(0.33081039,0.00365635){\color[rgb]{0,0,0}\makebox(0,0)[lt]{\lineheight{0}\smash{\begin{tabular}[t]{l}$l_1$\end{tabular}}}}%
    \put(0,0){\includegraphics[width=\unitlength,page=4]{type-3-8.pdf}}%
    \put(0.82281294,0.37699947){\color[rgb]{0,0,0}\makebox(0,0)[lt]{\lineheight{0}\smash{\begin{tabular}[t]{l}$l''$\end{tabular}}}}%
    \put(0.7645443,0.48529566){\color[rgb]{0,0,0}\makebox(0,0)[lt]{\lineheight{0}\smash{\begin{tabular}[t]{l}$C_j$\end{tabular}}}}%
    \put(0.72333994,0.4390663){\color[rgb]{0,0,0}\makebox(0,0)[lt]{\lineheight{0}\smash{\begin{tabular}[t]{l}$l_1$\end{tabular}}}}%
    \put(0.64337062,0.28420521){\color[rgb]{0,0,0}\makebox(0,0)[lt]{\lineheight{0}\smash{\begin{tabular}[t]{l}$l_2$\end{tabular}}}}%
    \put(0.55507463,0.51192484){\color[rgb]{0,0,0}\makebox(0,0)[lt]{\lineheight{0}\smash{\begin{tabular}[t]{l}$C$\end{tabular}}}}%
    \put(0.59639938,0.28813322){\color[rgb]{0,0,0}\makebox(0,0)[lt]{\lineheight{0}\smash{\begin{tabular}[t]{l}$l_1$\end{tabular}}}}%
    \put(0,0){\includegraphics[width=\unitlength,page=5]{type-3-8.pdf}}%
    \put(0.60871581,0.42550354){\color[rgb]{0,0,0}\makebox(0,0)[lt]{\lineheight{0}\smash{\begin{tabular}[t]{l}$l''$\end{tabular}}}}%
    \put(0,0){\includegraphics[width=\unitlength,page=6]{type-3-8.pdf}}%
    \put(0.82071749,0.29187146){\color[rgb]{0,0,0}\makebox(0,0)[lt]{\lineheight{0}\smash{\begin{tabular}[t]{l}$C_k'$\end{tabular}}}}%
    \put(0,0){\includegraphics[width=\unitlength,page=7]{type-3-8.pdf}}%
    \put(0.76534482,0.13585979){\color[rgb]{0,0,0}\makebox(0,0)[lt]{\lineheight{0}\smash{\begin{tabular}[t]{l}$l''$\end{tabular}}}}%
    \put(0.64294566,0.00330844){\color[rgb]{0,0,0}\makebox(0,0)[lt]{\lineheight{0}\smash{\begin{tabular}[t]{l}$l_2$\end{tabular}}}}%
    \put(0.55464972,0.23102788){\color[rgb]{0,0,0}\makebox(0,0)[lt]{\lineheight{0}\smash{\begin{tabular}[t]{l}$C$\end{tabular}}}}%
    \put(0.59597446,0.00723663){\color[rgb]{0,0,0}\makebox(0,0)[lt]{\lineheight{0}\smash{\begin{tabular}[t]{l}$l_1$\end{tabular}}}}%
    \put(0,0){\includegraphics[width=\unitlength,page=8]{type-3-8.pdf}}%
    \put(0.60829089,0.14460669){\color[rgb]{0,0,0}\makebox(0,0)[lt]{\lineheight{0}\smash{\begin{tabular}[t]{l}$l''$\end{tabular}}}}%
    \put(0,0){\includegraphics[width=\unitlength,page=9]{type-3-8.pdf}}%
    \put(0.78572091,0.01097486){\color[rgb]{0,0,0}\makebox(0,0)[lt]{\lineheight{0}\smash{\begin{tabular}[t]{l}$C_k'$\end{tabular}}}}%
    \put(0,0){\includegraphics[width=\unitlength,page=10]{type-3-8.pdf}}%
  \end{picture}%
\endgroup%

%% file: figures/semi-reduced-2.pdf_tex
\begingroup%
  \makeatletter%
  \providecommand\color[2][]{%
    \errmessage{(Inkscape) Color is used for the text in Inkscape, but the package 'color.sty' is not loaded}%
    \renewcommand\color[2][]{}%
  }%
  \providecommand\transparent[1]{%
    \errmessage{(Inkscape) Transparency is used (non-zero) for the text in Inkscape, but the package 'transparent.sty' is not loaded}%
    \renewcommand\transparent[1]{}%
  }%
  \providecommand\rotatebox[2]{#2}%
  \newcommand*\fsize{\dimexpr\f@size pt\relax}%
  \newcommand*\lineheight[1]{\fontsize{\fsize}{#1\fsize}\selectfont}%
  \ifx\svgwidth\undefined%
    \setlength{\unitlength}{711.21374605bp}%
    \ifx\svgscale\undefined%
      \relax%
    \else%
      \setlength{\unitlength}{\unitlength * \real{\svgscale}}%
    \fi%
  \else%
    \setlength{\unitlength}{\svgwidth}%
  \fi%
  \global\let\svgwidth\undefined%
  \global\let\svgscale\undefined%
  \makeatother%
  \begin{picture}(1,0.2777351)%
    \lineheight{1}%
    \setlength\tabcolsep{0pt}%
    \put(0,0){\includegraphics[width=\unitlength,page=1]{semi-reduced-2.pdf}}%
    \put(0.63511061,0.05762865){\color[rgb]{0,0,0}\makebox(0,0)[lt]{\lineheight{0}\smash{\begin{tabular}[t]{l}$l'_k$\end{tabular}}}}%
    \put(0,0){\includegraphics[width=\unitlength,page=2]{semi-reduced-2.pdf}}%
    \put(0.4182999,0.13014315){\color[rgb]{0,0,0}\makebox(0,0)[lt]{\lineheight{0}\smash{\begin{tabular}[t]{l}$C_k$\end{tabular}}}}%
    \put(0.34347105,0.11504748){\color[rgb]{0,0,0}\makebox(0,0)[lt]{\lineheight{0}\smash{\begin{tabular}[t]{l}$l'_{k}$\end{tabular}}}}%
    \put(0.20323057,0.11240052){\color[rgb]{0,0,0}\makebox(0,0)[lt]{\lineheight{0}\smash{\begin{tabular}[t]{l}$l$\end{tabular}}}}%
    \put(-0.00203329,0.0972317){\color[rgb]{0,0,0}\makebox(0,0)[lt]{\lineheight{0}\smash{\begin{tabular}[t]{l}$C_i$\end{tabular}}}}%
    \put(0.58914236,0.1327741){\color[rgb]{0,0,0}\makebox(0,0)[lt]{\lineheight{0}\smash{\begin{tabular}[t]{l}$l$\end{tabular}}}}%
    \put(0.67369657,0.10375488){\color[rgb]{0,0,0}\makebox(0,0)[lt]{\lineheight{0}\smash{\begin{tabular}[t]{l}$C$\end{tabular}}}}%
    \put(0.17196486,0.25825915){\color[rgb]{0,0,0}\makebox(0,0)[lt]{\lineheight{0}\smash{\begin{tabular}[t]{l}$C_j$\end{tabular}}}}%
    \put(0.1453652,0.17214782){\color[rgb]{0,0,0}\makebox(0,0)[lt]{\lineheight{0}\smash{\begin{tabular}[t]{l}$l''_j$\end{tabular}}}}%
    \put(0.54866797,0.05737059){\color[rgb]{0,0,0}\makebox(0,0)[lt]{\lineheight{0}\smash{\begin{tabular}[t]{l}$l''_j$\end{tabular}}}}%
    \put(0,0){\includegraphics[width=\unitlength,page=3]{semi-reduced-2.pdf}}%
    \put(0.85146197,0.14161865){\color[rgb]{0,0,0}\makebox(0,0)[lt]{\lineheight{0}\smash{\begin{tabular}[t]{l}$l$\end{tabular}}}}%
    \put(0.95288884,0.09361691){\color[rgb]{0,0,0}\makebox(0,0)[lt]{\lineheight{0}\smash{\begin{tabular}[t]{l}$C$\end{tabular}}}}%
    \put(0,0){\includegraphics[width=\unitlength,page=4]{semi-reduced-2.pdf}}%
    \put(0.83491401,0.0551924){\color[rgb]{0,0,0}\makebox(0,0)[lt]{\lineheight{0}\smash{\begin{tabular}[t]{l}$l'_k=l''_j$\end{tabular}}}}%
    \put(0,0){\includegraphics[width=\unitlength,page=5]{semi-reduced-2.pdf}}%
  \end{picture}%
\endgroup%

%% file: figures/Semi-Reduced-A.pdf_tex
\begingroup%
  \makeatletter%
  \providecommand\color[2][]{%
    \errmessage{(Inkscape) Color is used for the text in Inkscape, but the package 'color.sty' is not loaded}%
    \renewcommand\color[2][]{}%
  }%
  \providecommand\transparent[1]{%
    \errmessage{(Inkscape) Transparency is used (non-zero) for the text in Inkscape, but the package 'transparent.sty' is not loaded}%
    \renewcommand\transparent[1]{}%
  }%
  \providecommand\rotatebox[2]{#2}%
  \newcommand*\fsize{\dimexpr\f@size pt\relax}%
  \newcommand*\lineheight[1]{\fontsize{\fsize}{#1\fsize}\selectfont}%
  \ifx\svgwidth\undefined%
    \setlength{\unitlength}{1004.05251648bp}%
    \ifx\svgscale\undefined%
      \relax%
    \else%
      \setlength{\unitlength}{\unitlength * \real{\svgscale}}%
    \fi%
  \else%
    \setlength{\unitlength}{\svgwidth}%
  \fi%
  \global\let\svgwidth\undefined%
  \global\let\svgscale\undefined%
  \makeatother%
  \begin{picture}(1,0.27023584)%
    \lineheight{1}%
    \setlength\tabcolsep{0pt}%
    \put(0,0){\includegraphics[width=\unitlength,page=1]{Semi-Reduced-A.pdf}}%
    \put(0.41637325,0.17381073){\color[rgb]{0,0,0}\makebox(0,0)[lt]{\lineheight{0}\smash{\begin{tabular}[t]{l}$\partial^+$\end{tabular}}}}%
    \put(0.32264439,0.15390228){\color[rgb]{0,0,0}\makebox(0,0)[lt]{\lineheight{0}\smash{\begin{tabular}[t]{l}$l_n^+$\end{tabular}}}}%
    \put(0.2577552,0.21541389){\color[rgb]{0,0,0}\makebox(0,0)[lt]{\lineheight{0}\smash{\begin{tabular}[t]{l}$\bar{l}^+$\end{tabular}}}}%
    \put(0,0){\includegraphics[width=\unitlength,page=2]{Semi-Reduced-A.pdf}}%
    \put(0.0518616,0.20830934){\color[rgb]{0,0,0}\rotatebox{-90}{\makebox(0,0)[lt]{\lineheight{0}\smash{\begin{tabular}[t]{l}$D_1^{C_n}\times I$\end{tabular}}}}}%
    \put(0.17455997,0.1462377){\color[rgb]{0,0,0}\makebox(0,0)[lt]{\lineheight{0}\smash{\begin{tabular}[t]{l}$\bar{C}$\end{tabular}}}}%
    \put(0.15199185,0.03161712){\color[rgb]{0,0,0}\makebox(0,0)[lt]{\lineheight{0}\smash{\begin{tabular}[t]{l}$\bar{l}^-$\end{tabular}}}}%
    \put(0.1847315,0.23797234){\color[rgb]{0,0,0}\makebox(0,0)[lt]{\lineheight{0}\smash{\begin{tabular}[t]{l}$\bar{l}^+$\end{tabular}}}}%
    \put(0.07382882,0.23975231){\color[rgb]{0,0,0}\makebox(0,0)[lt]{\lineheight{0}\smash{\begin{tabular}[t]{l}$\bar{l}_1$\end{tabular}}}}%
    \put(0.13374663,0.24003846){\color[rgb]{0,0,0}\makebox(0,0)[lt]{\lineheight{0}\smash{\begin{tabular}[t]{l}$\bar{l}_2$\end{tabular}}}}%
    \put(0,0){\includegraphics[width=\unitlength,page=3]{Semi-Reduced-A.pdf}}%
    \put(0.41677354,0.05026518){\color[rgb]{0,0,0}\makebox(0,0)[lt]{\lineheight{0}\smash{\begin{tabular}[t]{l}$\partial^+$\end{tabular}}}}%
    \put(0.3529236,0.03185068){\color[rgb]{0,0,0}\makebox(0,0)[lt]{\lineheight{0}\smash{\begin{tabular}[t]{l}$l_n^+$\end{tabular}}}}%
    \put(0,0){\includegraphics[width=\unitlength,page=4]{Semi-Reduced-A.pdf}}%
    \put(0.32780043,0.11336642){\color[rgb]{0,0,0}\makebox(0,0)[lt]{\lineheight{0}\smash{\begin{tabular}[t]{l}$\bar{l}^+$\end{tabular}}}}%
    \put(0,0){\includegraphics[width=\unitlength,page=5]{Semi-Reduced-A.pdf}}%
    \put(0.82910257,0.16081721){\color[rgb]{0,0,0}\makebox(0,0)[lt]{\lineheight{0}\smash{\begin{tabular}[t]{l}$\partial^+$\end{tabular}}}}%
    \put(0.93406849,0.16033002){\color[rgb]{0,0,0}\makebox(0,0)[lt]{\lineheight{0}\smash{\begin{tabular}[t]{l}$l_n^+$\end{tabular}}}}%
    \put(0,0){\includegraphics[width=\unitlength,page=6]{Semi-Reduced-A.pdf}}%
    \put(0.88952403,0.24035186){\color[rgb]{0,0,0}\makebox(0,0)[lt]{\lineheight{0}\smash{\begin{tabular}[t]{l}$\bar{l}^+$\end{tabular}}}}%
    \put(0,0){\includegraphics[width=\unitlength,page=7]{Semi-Reduced-A.pdf}}%
    \put(0.8329226,0.03921309){\color[rgb]{0,0,0}\makebox(0,0)[lt]{\lineheight{0}\smash{\begin{tabular}[t]{l}$\partial^+$\end{tabular}}}}%
    \put(0.92892475,0.03872585){\color[rgb]{0,0,0}\makebox(0,0)[lt]{\lineheight{0}\smash{\begin{tabular}[t]{l}$l_n^+$\end{tabular}}}}%
    \put(0,0){\includegraphics[width=\unitlength,page=8]{Semi-Reduced-A.pdf}}%
    \put(0.89334397,0.12173563){\color[rgb]{0,0,0}\makebox(0,0)[lt]{\lineheight{0}\smash{\begin{tabular}[t]{l}$\bar{l}^+$\end{tabular}}}}%
    \put(0,0){\includegraphics[width=\unitlength,page=9]{Semi-Reduced-A.pdf}}%
    \put(0.62619274,0.22525506){\color[rgb]{0,0,0}\makebox(0,0)[lt]{\lineheight{0}\smash{\begin{tabular}[t]{l}$\bar{l}^+$\end{tabular}}}}%
    \put(0.62430542,0.07109647){\color[rgb]{0,0,0}\makebox(0,0)[lt]{\lineheight{0}\smash{\begin{tabular}[t]{l}$\bar{l}^-$\end{tabular}}}}%
    \put(0.70774592,0.23139247){\color[rgb]{0,0,0}\makebox(0,0)[lt]{\lineheight{0}\smash{\begin{tabular}[t]{l}$\bar{l}_2$\end{tabular}}}}%
    \put(0.63176236,0.14145993){\color[rgb]{0,0,0}\makebox(0,0)[lt]{\lineheight{0}\smash{\begin{tabular}[t]{l}$\bar{C}$\end{tabular}}}}%
    \put(0.7497863,0.20055634){\color[rgb]{0,0,0}\rotatebox{-90}{\makebox(0,0)[lt]{\lineheight{0}\smash{\begin{tabular}[t]{l}$D_2^{C_n}\times I$\end{tabular}}}}}%
    \put(0.52299943,0.20134558){\color[rgb]{0,0,0}\rotatebox{-90}{\makebox(0,0)[lt]{\lineheight{0}\smash{\begin{tabular}[t]{l}$D_1^{C_n}\times I$\end{tabular}}}}}%
    \put(0,0){\includegraphics[width=\unitlength,page=10]{Semi-Reduced-A.pdf}}%
    \put(0.54341199,0.23139247){\color[rgb]{0,0,0}\makebox(0,0)[lt]{\lineheight{0}\smash{\begin{tabular}[t]{l}$\bar{l}_1$\end{tabular}}}}%
    \put(0,0){\includegraphics[width=\unitlength,page=11]{Semi-Reduced-A.pdf}}%
    \put(0.07969793,0.0040072){\color[rgb]{0,0,0}\makebox(0,0)[lt]{\lineheight{1.25}\smash{\begin{tabular}[t]{l}$(1)$\end{tabular}}}}%
    \put(0.30378979,0.0040072){\color[rgb]{0,0,0}\makebox(0,0)[lt]{\lineheight{1.25}\smash{\begin{tabular}[t]{l}$(2)$\end{tabular}}}}%
    \put(0.601085,0.0040072){\color[rgb]{0,0,0}\makebox(0,0)[lt]{\lineheight{1.25}\smash{\begin{tabular}[t]{l}$(3)$\end{tabular}}}}%
    \put(0.88642864,0.0040072){\color[rgb]{0,0,0}\makebox(0,0)[lt]{\lineheight{1.25}\smash{\begin{tabular}[t]{l}$(4)$\end{tabular}}}}%
  \end{picture}%
\endgroup%

%% file: figures/semi-reduced-6.pdf_tex
\begingroup%
  \makeatletter%
  \providecommand\color[2][]{%
    \errmessage{(Inkscape) Color is used for the text in Inkscape, but the package 'color.sty' is not loaded}%
    \renewcommand\color[2][]{}%
  }%
  \providecommand\transparent[1]{%
    \errmessage{(Inkscape) Transparency is used (non-zero) for the text in Inkscape, but the package 'transparent.sty' is not loaded}%
    \renewcommand\transparent[1]{}%
  }%
  \providecommand\rotatebox[2]{#2}%
  \newcommand*\fsize{\dimexpr\f@size pt\relax}%
  \newcommand*\lineheight[1]{\fontsize{\fsize}{#1\fsize}\selectfont}%
  \ifx\svgwidth\undefined%
    \setlength{\unitlength}{285.39809564bp}%
    \ifx\svgscale\undefined%
      \relax%
    \else%
      \setlength{\unitlength}{\unitlength * \real{\svgscale}}%
    \fi%
  \else%
    \setlength{\unitlength}{\svgwidth}%
  \fi%
  \global\let\svgwidth\undefined%
  \global\let\svgscale\undefined%
  \makeatother%
  \begin{picture}(1,0.6113509)%
    \lineheight{1}%
    \setlength\tabcolsep{0pt}%
    \put(0,0){\includegraphics[width=\unitlength,page=1]{semi-reduced-6.pdf}}%
    \put(0.47586799,0.11285205){\color[rgb]{0,0,0}\makebox(0,0)[lt]{\lineheight{0}\smash{\begin{tabular}[t]{l}$l''$\end{tabular}}}}%
    \put(0.86902614,0.4451663){\color[rgb]{0,0,0}\makebox(0,0)[lt]{\lineheight{0}\smash{\begin{tabular}[t]{l}$\bar{C}$\end{tabular}}}}%
    \put(0.75293249,0.19857911){\color[rgb]{0,0,0}\makebox(0,0)[lt]{\lineheight{0}\smash{\begin{tabular}[t]{l}$l'_{i_s}$\end{tabular}}}}%
    \put(-0.00187937,0.30503984){\color[rgb]{0,0,0}\makebox(0,0)[lt]{\lineheight{0}\smash{\begin{tabular}[t]{l}$C_{i_s}$\end{tabular}}}}%
    \put(0.00627377,0.12595861){\color[rgb]{0,0,0}\makebox(0,0)[lt]{\lineheight{0}\smash{\begin{tabular}[t]{l}$C_{i_s}$\end{tabular}}}}%
    \put(0.3373968,0.5500285){\color[rgb]{0,0,0}\makebox(0,0)[lt]{\lineheight{0}\smash{\begin{tabular}[t]{l}$D_1^{C_n}\times I$\end{tabular}}}}%
    \put(0.35744649,0.41346171){\color[rgb]{0,0,0}\makebox(0,0)[lt]{\lineheight{0}\smash{\begin{tabular}[t]{l}$\bar{l}_1$\end{tabular}}}}%
    \put(0.54497382,0.41241754){\color[rgb]{0,0,0}\makebox(0,0)[lt]{\lineheight{0}\smash{\begin{tabular}[t]{l}$\bar{l}_2$\end{tabular}}}}%
    \put(0.70451201,0.31589009){\color[rgb]{0,0,0}\makebox(0,0)[lt]{\lineheight{0}\smash{\begin{tabular}[t]{l}$l'$\end{tabular}}}}%
    \put(0.2387661,0.3993048){\color[rgb]{0,0,0}\makebox(0,0)[lt]{\lineheight{0}\smash{\begin{tabular}[t]{l}$l''_1$\end{tabular}}}}%
    \put(0.23940483,0.19248){\color[rgb]{0,0,0}\makebox(0,0)[lt]{\lineheight{0}\smash{\begin{tabular}[t]{l}$l''_2$\end{tabular}}}}%
  \end{picture}%
\endgroup%

%% file: figures/slide.pdf_tex
\begingroup%
  \makeatletter%
  \providecommand\color[2][]{%
    \errmessage{(Inkscape) Color is used for the text in Inkscape, but the package 'color.sty' is not loaded}%
    \renewcommand\color[2][]{}%
  }%
  \providecommand\transparent[1]{%
    \errmessage{(Inkscape) Transparency is used (non-zero) for the text in Inkscape, but the package 'transparent.sty' is not loaded}%
    \renewcommand\transparent[1]{}%
  }%
  \providecommand\rotatebox[2]{#2}%
  \newcommand*\fsize{\dimexpr\f@size pt\relax}%
  \newcommand*\lineheight[1]{\fontsize{\fsize}{#1\fsize}\selectfont}%
  \ifx\svgwidth\undefined%
    \setlength{\unitlength}{1023.59299124bp}%
    \ifx\svgscale\undefined%
      \relax%
    \else%
      \setlength{\unitlength}{\unitlength * \real{\svgscale}}%
    \fi%
  \else%
    \setlength{\unitlength}{\svgwidth}%
  \fi%
  \global\let\svgwidth\undefined%
  \global\let\svgscale\undefined%
  \makeatother%
  \begin{picture}(1,0.43505361)%
    \lineheight{1}%
    \setlength\tabcolsep{0pt}%
    \put(0,0){\includegraphics[width=\unitlength,page=1]{slide.pdf}}%
    \put(0.55524351,0.44794678){\color[rgb]{0,0,0}\makebox(0,0)[lt]{\begin{minipage}{0.58126031\unitlength}\end{minipage}}}%
    \put(0.13243037,0.35628526){\color[rgb]{0,0,0}\makebox(0,0)[lt]{\lineheight{0}\smash{\begin{tabular}[t]{l}$l_1$\end{tabular}}}}%
    \put(0.34967339,0.34921523){\color[rgb]{0,0,0}\makebox(0,0)[lt]{\lineheight{0}\smash{\begin{tabular}[t]{l}$l_2$\end{tabular}}}}%
    \put(0.41628602,0.37813812){\color[rgb]{0,0,0}\makebox(0,0)[lt]{\lineheight{0}\smash{\begin{tabular}[t]{l}$C_{j_2}$\end{tabular}}}}%
    \put(0.05016083,0.38970725){\color[rgb]{0,0,0}\makebox(0,0)[lt]{\lineheight{0}\smash{\begin{tabular}[t]{l}$C_{j_1}$\end{tabular}}}}%
    \put(0.41394647,0.29651132){\color[rgb]{0,0,0}\makebox(0,0)[lt]{\lineheight{0}\smash{\begin{tabular}[t]{l}$C_k$\end{tabular}}}}%
    \put(0.3563835,0.30599804){\color[rgb]{0,0,0}\makebox(0,0)[lt]{\lineheight{0}\smash{\begin{tabular}[t]{l}$l^2_k$\end{tabular}}}}%
    \put(0.10093656,0.30885175){\color[rgb]{0,0,0}\makebox(0,0)[lt]{\lineheight{0}\smash{\begin{tabular}[t]{l}$l^1_k$\end{tabular}}}}%
    \put(0.21984176,0.26951663){\color[rgb]{0,0,0}\makebox(0,0)[lt]{\lineheight{0}\smash{\begin{tabular}[t]{l}$l'_k$\end{tabular}}}}%
    \put(0,0){\includegraphics[width=\unitlength,page=2]{slide.pdf}}%
    \put(0.53674814,0.14116655){\color[rgb]{0,0,0}\makebox(0,0)[lt]{\lineheight{0}\smash{\begin{tabular}[t]{l}$l^1_{i_1}$\end{tabular}}}}%
    \put(0.53734376,0.06955754){\color[rgb]{0,0,0}\makebox(0,0)[lt]{\lineheight{0}\smash{\begin{tabular}[t]{l}$l^1_{i_k}$\end{tabular}}}}%
    \put(0.90503861,0.14067047){\color[rgb]{0,0,0}\makebox(0,0)[lt]{\lineheight{0}\smash{\begin{tabular}[t]{l}$l^2_{i_1}$\end{tabular}}}}%
    \put(0.90629931,0.06955754){\color[rgb]{0,0,0}\makebox(0,0)[lt]{\lineheight{0}\smash{\begin{tabular}[t]{l}$l^2_{i_k}$\end{tabular}}}}%
    \put(0.71251319,0.12027002){\color[rgb]{0,0,0}\makebox(0,0)[lt]{\lineheight{0}\smash{\begin{tabular}[t]{l}$l'_{i_1}$\end{tabular}}}}%
    \put(0.71576404,0.06036129){\color[rgb]{0,0,0}\makebox(0,0)[lt]{\lineheight{0}\smash{\begin{tabular}[t]{l}$l'_{i_k}$\end{tabular}}}}%
    \put(0,0){\includegraphics[width=\unitlength,page=3]{slide.pdf}}%
    \put(0.60922612,0.20691337){\color[rgb]{0,0,0}\makebox(0,0)[lt]{\lineheight{0}\smash{\begin{tabular}[t]{l}$C_{j_1}$\end{tabular}}}}%
    \put(0.80203666,0.20788259){\color[rgb]{0,0,0}\makebox(0,0)[lt]{\lineheight{0}\smash{\begin{tabular}[t]{l}$C_{j_2}$\end{tabular}}}}%
    \put(0.71703975,0.16865214){\color[rgb]{0,0,0}\makebox(0,0)[lt]{\lineheight{0}\smash{\begin{tabular}[t]{l}$\bar{C}$\end{tabular}}}}%
    \put(0,0){\includegraphics[width=\unitlength,page=4]{slide.pdf}}%
    \put(0.03180893,0.06743134){\color[rgb]{0,0,0}\makebox(0,0)[lt]{\lineheight{0}\smash{\begin{tabular}[t]{l}$C_{j_1,j_2}$\end{tabular}}}}%
    \put(0.17927075,0.08033023){\color[rgb]{0,0,0}\makebox(0,0)[lt]{\lineheight{0}\smash{\begin{tabular}[t]{l}$l_1$\end{tabular}}}}%
    \put(0.3251472,0.07879817){\color[rgb]{0,0,0}\makebox(0,0)[lt]{\lineheight{0}\smash{\begin{tabular}[t]{l}$l_2$\end{tabular}}}}%
    \put(0.23335825,0.08033023){\color[rgb]{0,0,0}\makebox(0,0)[lt]{\lineheight{0}\smash{\begin{tabular}[t]{l}$\bar{C}$\end{tabular}}}}%
    \put(0.12231894,0.17724817){\color[rgb]{0,0,0}\makebox(0,0)[lt]{\lineheight{0}\smash{\begin{tabular}[t]{l}$C_{j_1}$\end{tabular}}}}%
    \put(0.33234124,0.17891341){\color[rgb]{0,0,0}\makebox(0,0)[lt]{\lineheight{0}\smash{\begin{tabular}[t]{l}$C_{j_2}$\end{tabular}}}}%
    \put(0,0){\includegraphics[width=\unitlength,page=5]{slide.pdf}}%
    \put(0.64833993,0.35551067){\color[rgb]{0,0,0}\makebox(0,0)[lt]{\lineheight{0}\smash{\begin{tabular}[t]{l}$l_1$\end{tabular}}}}%
    \put(0.86558306,0.34844064){\color[rgb]{0,0,0}\makebox(0,0)[lt]{\lineheight{0}\smash{\begin{tabular}[t]{l}$l_2$\end{tabular}}}}%
    \put(0.8735786,0.40227577){\color[rgb]{0,0,0}\makebox(0,0)[lt]{\lineheight{0}\smash{\begin{tabular}[t]{l}$C_{j_2}$\end{tabular}}}}%
    \put(0.61003327,0.39919063){\color[rgb]{0,0,0}\makebox(0,0)[lt]{\lineheight{0}\smash{\begin{tabular}[t]{l}$C_{j_1}$\end{tabular}}}}%
    \put(0,0){\includegraphics[width=\unitlength,page=6]{slide.pdf}}%
    \put(0.74846752,0.33579071){\color[rgb]{0,0,0}\makebox(0,0)[lt]{\lineheight{0}\smash{\begin{tabular}[t]{l}$\bar{C}$\end{tabular}}}}%
    \put(0.54908184,0.37191009){\color[rgb]{0,0,0}\makebox(0,0)[lt]{\lineheight{0}\smash{\begin{tabular}[t]{l}$C_{j_1,j_2}$\end{tabular}}}}%
    \put(0.61046022,0.31162152){\color[rgb]{0,0,0}\makebox(0,0)[lt]{\lineheight{0}\smash{\begin{tabular}[t]{l}$l^1_{i_s}$\end{tabular}}}}%
    \put(0.85945929,0.30869311){\color[rgb]{0,0,0}\makebox(0,0)[lt]{\lineheight{0}\smash{\begin{tabular}[t]{l}$l^2_{i_s}$\end{tabular}}}}%
    \put(0.89924566,0.27055612){\color[rgb]{0,0,0}\makebox(0,0)[lt]{\lineheight{0}\smash{\begin{tabular}[t]{l}$l''_{i_s}$\end{tabular}}}}%
    \put(0.71296256,0.24191547){\color[rgb]{0,0,0}\makebox(0,0)[lt]{\lineheight{0}\smash{\begin{tabular}[t]{l}$C_{i_s}$\end{tabular}}}}%
    \put(0.22498361,0.35307159){\color[rgb]{0,0,0}\makebox(0,0)[lt]{\lineheight{0}\smash{\begin{tabular}[t]{l}$\bar{C}$\end{tabular}}}}%
    \put(0,0){\includegraphics[width=\unitlength,page=7]{slide.pdf}}%
  \end{picture}%
\endgroup%

%% file: figures/empty-A.pdf_tex
\begingroup%
  \makeatletter%
  \providecommand\color[2][]{%
    \errmessage{(Inkscape) Color is used for the text in Inkscape, but the package 'color.sty' is not loaded}%
    \renewcommand\color[2][]{}%
  }%
  \providecommand\transparent[1]{%
    \errmessage{(Inkscape) Transparency is used (non-zero) for the text in Inkscape, but the package 'transparent.sty' is not loaded}%
    \renewcommand\transparent[1]{}%
  }%
  \providecommand\rotatebox[2]{#2}%
  \newcommand*\fsize{\dimexpr\f@size pt\relax}%
  \newcommand*\lineheight[1]{\fontsize{\fsize}{#1\fsize}\selectfont}%
  \ifx\svgwidth\undefined%
    \setlength{\unitlength}{841.70608443bp}%
    \ifx\svgscale\undefined%
      \relax%
    \else%
      \setlength{\unitlength}{\unitlength * \real{\svgscale}}%
    \fi%
  \else%
    \setlength{\unitlength}{\svgwidth}%
  \fi%
  \global\let\svgwidth\undefined%
  \global\let\svgscale\undefined%
  \makeatother%
  \begin{picture}(1,0.2798709)%
    \lineheight{1}%
    \setlength\tabcolsep{0pt}%
    \put(0,0){\includegraphics[width=\unitlength,page=1]{empty-A.pdf}}%
    \put(0.62804644,0.19517479){\color[rgb]{0,0,0}\makebox(0,0)[lt]{\lineheight{0}\smash{\begin{tabular}[t]{l}$C$\end{tabular}}}}%
    \put(0.48319157,0.20995266){\color[rgb]{0,0,0}\makebox(0,0)[lt]{\lineheight{0}\smash{\begin{tabular}[t]{l}$l_1$\end{tabular}}}}%
    \put(0.54456676,0.21722373){\color[rgb]{0,0,0}\makebox(0,0)[lt]{\lineheight{0}\smash{\begin{tabular}[t]{l}$l_2$\end{tabular}}}}%
    \put(0.48127038,0.2438201){\color[rgb]{0,0,0}\makebox(0,0)[lt]{\lineheight{0}\smash{\begin{tabular}[t]{l}$C_j$\end{tabular}}}}%
    \put(0.97760725,0.19209933){\color[rgb]{0,0,0}\makebox(0,0)[lt]{\lineheight{0}\smash{\begin{tabular}[t]{l}$C$\end{tabular}}}}%
    \put(0.82562447,0.18014572){\color[rgb]{0,0,0}\makebox(0,0)[lt]{\lineheight{0}\smash{\begin{tabular}[t]{l}$l'_+$\end{tabular}}}}%
    \put(0.82714591,0.06106334){\color[rgb]{0,0,0}\makebox(0,0)[lt]{\lineheight{0}\smash{\begin{tabular}[t]{l}$l'_-$\end{tabular}}}}%
    \put(0.82904954,0.24609099){\color[rgb]{0,0,0}\makebox(0,0)[lt]{\lineheight{0}\smash{\begin{tabular}[t]{l}$C'_j$\end{tabular}}}}%
    \put(0.04529029,0.00242213){\color[rgb]{0,0,0}\makebox(0,0)[lt]{\lineheight{0}\smash{\begin{tabular}[t]{l}$C_i$\end{tabular}}}}%
    \put(0.13042852,0.26929672){\color[rgb]{0,0,0}\makebox(0,0)[lt]{\lineheight{0}\smash{\begin{tabular}[t]{l}$C_j$\end{tabular}}}}%
    \put(0.31046164,0.01457525){\color[rgb]{0,0,0}\makebox(0,0)[lt]{\lineheight{0}\smash{\begin{tabular}[t]{l}$D_{l'}$\end{tabular}}}}%
    \put(0.12275215,0.18542949){\color[rgb]{0,0,0}\makebox(0,0)[lt]{\lineheight{0}\smash{\begin{tabular}[t]{l}$l_1$\end{tabular}}}}%
    \put(0.10757327,0.13729565){\color[rgb]{0,0,0}\makebox(0,0)[lt]{\lineheight{0}\smash{\begin{tabular}[t]{l}$l''_j$\end{tabular}}}}%
    \put(0.18506556,0.18658402){\color[rgb]{0,0,0}\makebox(0,0)[lt]{\lineheight{0}\smash{\begin{tabular}[t]{l}$l_2$\end{tabular}}}}%
    \put(0.28610329,0.06924802){\color[rgb]{0,0,0}\makebox(0,0)[lt]{\lineheight{0}\smash{\begin{tabular}[t]{l}$l'$\end{tabular}}}}%
    \put(0.29214139,0.21200363){\color[rgb]{0,0,0}\makebox(0,0)[lt]{\lineheight{1.25}\smash{\begin{tabular}[t]{l}$C$\end{tabular}}}}%
    \put(0.29265733,0.16834031){\color[rgb]{0,0,0}\makebox(0,0)[lt]{\lineheight{0}\smash{\begin{tabular}[t]{l}$l''$\end{tabular}}}}%
    \put(0.31254109,0.12271854){\color[rgb]{0,0,0}\makebox(0,0)[lt]{\lineheight{0}\smash{\begin{tabular}[t]{l}$\bar{D}$\end{tabular}}}}%
    \put(0.64579279,0.11915465){\color[rgb]{0,0,0}\makebox(0,0)[lt]{\lineheight{0}\smash{\begin{tabular}[t]{l}$\bar{D}\times I$\end{tabular}}}}%
    \put(0,0){\includegraphics[width=\unitlength,page=2]{empty-A.pdf}}%
    \put(0.17257039,0.1000361){\color[rgb]{0,0,0}\makebox(0,0)[lt]{\lineheight{0}\smash{\begin{tabular}[t]{l}$\bar{l}_j$\end{tabular}}}}%
  \end{picture}%
\endgroup%

%% file: figures/empty-5.pdf_tex
\begingroup%
  \makeatletter%
  \providecommand\color[2][]{%
    \errmessage{(Inkscape) Color is used for the text in Inkscape, but the package 'color.sty' is not loaded}%
    \renewcommand\color[2][]{}%
  }%
  \providecommand\transparent[1]{%
    \errmessage{(Inkscape) Transparency is used (non-zero) for the text in Inkscape, but the package 'transparent.sty' is not loaded}%
    \renewcommand\transparent[1]{}%
  }%
  \providecommand\rotatebox[2]{#2}%
  \newcommand*\fsize{\dimexpr\f@size pt\relax}%
  \newcommand*\lineheight[1]{\fontsize{\fsize}{#1\fsize}\selectfont}%
  \ifx\svgwidth\undefined%
    \setlength{\unitlength}{650.80480362bp}%
    \ifx\svgscale\undefined%
      \relax%
    \else%
      \setlength{\unitlength}{\unitlength * \real{\svgscale}}%
    \fi%
  \else%
    \setlength{\unitlength}{\svgwidth}%
  \fi%
  \global\let\svgwidth\undefined%
  \global\let\svgscale\undefined%
  \makeatother%
  \begin{picture}(1,0.29668074)%
    \lineheight{1}%
    \setlength\tabcolsep{0pt}%
    \put(0,0){\includegraphics[width=\unitlength,page=1]{empty-5.pdf}}%
    \put(0.27646007,0.07941905){\color[rgb]{0,0,0}\makebox(0,0)[lt]{\lineheight{0}\smash{\begin{tabular}[t]{l}$l''_1$\end{tabular}}}}%
    \put(0,0){\includegraphics[width=\unitlength,page=2]{empty-5.pdf}}%
    \put(-0.00100446,0.25795332){\color[rgb]{0,0,0}\makebox(0,0)[lt]{\lineheight{0}\smash{\begin{tabular}[t]{l}$C'_j$\end{tabular}}}}%
    \put(0.10371845,0.26821239){\color[rgb]{0,0,0}\makebox(0,0)[lt]{\lineheight{0}\smash{\begin{tabular}[t]{l}$l'_+$\end{tabular}}}}%
    \put(0.0990855,0.0689021){\color[rgb]{0,0,0}\makebox(0,0)[lt]{\lineheight{0}\smash{\begin{tabular}[t]{l}$l'_-$\end{tabular}}}}%
    \put(0.13990044,0.10988571){\color[rgb]{0,0,0}\makebox(0,0)[lt]{\lineheight{0}\smash{\begin{tabular}[t]{l}$D^-$\end{tabular}}}}%
    \put(0.14918662,0.23883154){\color[rgb]{0,0,0}\makebox(0,0)[lt]{\lineheight{0}\smash{\begin{tabular}[t]{l}$D^+$\end{tabular}}}}%
    \put(0.14678865,0.18623933){\color[rgb]{0,0,0}\makebox(0,0)[lt]{\lineheight{0}\smash{\begin{tabular}[t]{l}$D_k$\end{tabular}}}}%
    \put(0.01983627,0.21257032){\color[rgb]{0,0,0}\makebox(0,0)[lt]{\lineheight{0}\smash{\begin{tabular}[t]{l}$l_k$\end{tabular}}}}%
    \put(0.3775495,0.26768802){\color[rgb]{0,0,0}\makebox(0,0)[lt]{\lineheight{0}\smash{\begin{tabular}[t]{l}$C'_j$\end{tabular}}}}%
    \put(0.32510429,0.22781835){\color[rgb]{0,0,0}\makebox(0,0)[lt]{\lineheight{0}\smash{\begin{tabular}[t]{l}$D_k$\end{tabular}}}}%
    \put(0.4292264,0.11435556){\color[rgb]{0,0,0}\makebox(0,0)[lt]{\lineheight{0}\smash{\begin{tabular}[t]{l}$l''_2$\end{tabular}}}}%
    \put(0.30084282,0.05742201){\color[rgb]{0,0,0}\makebox(0,0)[lt]{\lineheight{0}\smash{\begin{tabular}[t]{l}$C_k$\end{tabular}}}}%
    \put(0.26337425,0.1625821){\color[rgb]{0,0,0}\makebox(0,0)[lt]{\lineheight{0}\smash{\begin{tabular}[t]{l}$l_k$\end{tabular}}}}%
    \put(0.55676647,0.25795332){\color[rgb]{0,0,0}\makebox(0,0)[lt]{\lineheight{0}\smash{\begin{tabular}[t]{l}$C'_j$\end{tabular}}}}%
    \put(0.92872936,0.26679846){\color[rgb]{0,0,0}\makebox(0,0)[lt]{\lineheight{0}\smash{\begin{tabular}[t]{l}$C'_j$\end{tabular}}}}%
    \put(0.84324199,0.05234163){\color[rgb]{0,0,0}\makebox(0,0)[lt]{\lineheight{0}\smash{\begin{tabular}[t]{l}$C'_k$\end{tabular}}}}%
    \put(0.94837507,0.12613691){\color[rgb]{0,0,0}\makebox(0,0)[lt]{\lineheight{0}\smash{\begin{tabular}[t]{l}$l''$\end{tabular}}}}%
    \put(0.08426182,0.00618226){\color[rgb]{0,0,0}\makebox(0,0)[lt]{\lineheight{1.25}\smash{\begin{tabular}[t]{l}$(1)$\end{tabular}}}}%
    \put(0.28017307,0.00618226){\color[rgb]{0,0,0}\makebox(0,0)[lt]{\lineheight{1.25}\smash{\begin{tabular}[t]{l}$(2)$\end{tabular}}}}%
    \put(0.62589884,0.00618226){\color[rgb]{0,0,0}\makebox(0,0)[lt]{\lineheight{1.25}\smash{\begin{tabular}[t]{l}$(3)$\end{tabular}}}}%
    \put(0.8448585,0.00618226){\color[rgb]{0,0,0}\makebox(0,0)[lt]{\lineheight{1.25}\smash{\begin{tabular}[t]{l}$(4)$\end{tabular}}}}%
  \end{picture}%
\endgroup%

%% file: figures/empty-7.pdf_tex
\begingroup%
  \makeatletter%
  \providecommand\color[2][]{%
    \errmessage{(Inkscape) Color is used for the text in Inkscape, but the package 'color.sty' is not loaded}%
    \renewcommand\color[2][]{}%
  }%
  \providecommand\transparent[1]{%
    \errmessage{(Inkscape) Transparency is used (non-zero) for the text in Inkscape, but the package 'transparent.sty' is not loaded}%
    \renewcommand\transparent[1]{}%
  }%
  \providecommand\rotatebox[2]{#2}%
  \newcommand*\fsize{\dimexpr\f@size pt\relax}%
  \newcommand*\lineheight[1]{\fontsize{\fsize}{#1\fsize}\selectfont}%
  \ifx\svgwidth\undefined%
    \setlength{\unitlength}{777.21924665bp}%
    \ifx\svgscale\undefined%
      \relax%
    \else%
      \setlength{\unitlength}{\unitlength * \real{\svgscale}}%
    \fi%
  \else%
    \setlength{\unitlength}{\svgwidth}%
  \fi%
  \global\let\svgwidth\undefined%
  \global\let\svgscale\undefined%
  \makeatother%
  \begin{picture}(1,0.27227974)%
    \lineheight{1}%
    \setlength\tabcolsep{0pt}%
    \put(0,0){\includegraphics[width=\unitlength,page=1]{empty-7.pdf}}%
    \put(0.0889301,0.16468134){\color[rgb]{0,0,0}\makebox(0,0)[lt]{\lineheight{0}\smash{\begin{tabular}[t]{l} \end{tabular}}}}%
    \put(0.12895537,0.00253682){\color[rgb]{0,0,0}\makebox(0,0)[lt]{\lineheight{0}\smash{\begin{tabular}[t]{l}$l_1$\end{tabular}}}}%
    \put(-0.00166288,0.25749989){\color[rgb]{0,0,0}\makebox(0,0)[lt]{\lineheight{0}\smash{\begin{tabular}[t]{l}$C$\end{tabular}}}}%
    \put(0.06690515,0.0060868){\color[rgb]{0,0,0}\makebox(0,0)[lt]{\lineheight{0}\smash{\begin{tabular}[t]{l}$l_2$\end{tabular}}}}%
    \put(0.25174207,0.25699675){\color[rgb]{0,0,0}\makebox(0,0)[lt]{\lineheight{0}\smash{\begin{tabular}[t]{l}$C$\end{tabular}}}}%
    \put(0.32785154,0.23430895){\color[rgb]{0,0,0}\makebox(0,0)[lt]{\lineheight{0}\smash{\begin{tabular}[t]{l}$l'_1$\end{tabular}}}}%
    \put(0.38676869,0.00333313){\color[rgb]{0,0,0}\makebox(0,0)[lt]{\lineheight{0}\smash{\begin{tabular}[t]{l}$l'_2$\end{tabular}}}}%
    \put(0,0){\includegraphics[width=\unitlength,page=2]{empty-7.pdf}}%
    \put(0.68486913,0.0057375){\color[rgb]{0,0,0}\makebox(0,0)[lt]{\lineheight{0}\smash{\begin{tabular}[t]{l}$l_1$\end{tabular}}}}%
    \put(0.55765063,0.26263067){\color[rgb]{0,0,0}\makebox(0,0)[lt]{\lineheight{0}\smash{\begin{tabular}[t]{l}$C$\end{tabular}}}}%
    \put(0.64029059,0.00735757){\color[rgb]{0,0,0}\makebox(0,0)[lt]{\lineheight{0}\smash{\begin{tabular}[t]{l}$l_2$\end{tabular}}}}%
    \put(0.81351142,0.26212756){\color[rgb]{0,0,0}\makebox(0,0)[lt]{\lineheight{0}\smash{\begin{tabular}[t]{l}$C$\end{tabular}}}}%
    \put(0.89026072,0.23943968){\color[rgb]{0,0,0}\makebox(0,0)[lt]{\lineheight{0}\smash{\begin{tabular}[t]{l}$l'_1$\end{tabular}}}}%
    \put(0.9339289,0.00460399){\color[rgb]{0,0,0}\makebox(0,0)[lt]{\lineheight{0}\smash{\begin{tabular}[t]{l}$l'_2$\end{tabular}}}}%
    \put(0,0){\includegraphics[width=\unitlength,page=3]{empty-7.pdf}}%
  \end{picture}%
\endgroup%

%% file: figures/Last-B.pdf_tex
\begingroup%
  \makeatletter%
  \providecommand\color[2][]{%
    \errmessage{(Inkscape) Color is used for the text in Inkscape, but the package 'color.sty' is not loaded}%
    \renewcommand\color[2][]{}%
  }%
  \providecommand\transparent[1]{%
    \errmessage{(Inkscape) Transparency is used (non-zero) for the text in Inkscape, but the package 'transparent.sty' is not loaded}%
    \renewcommand\transparent[1]{}%
  }%
  \providecommand\rotatebox[2]{#2}%
  \newcommand*\fsize{\dimexpr\f@size pt\relax}%
  \newcommand*\lineheight[1]{\fontsize{\fsize}{#1\fsize}\selectfont}%
  \ifx\svgwidth\undefined%
    \setlength{\unitlength}{956.13067bp}%
    \ifx\svgscale\undefined%
      \relax%
    \else%
      \setlength{\unitlength}{\unitlength * \real{\svgscale}}%
    \fi%
  \else%
    \setlength{\unitlength}{\svgwidth}%
  \fi%
  \global\let\svgwidth\undefined%
  \global\let\svgscale\undefined%
  \makeatother%
  \begin{picture}(1,0.31175358)%
    \lineheight{1}%
    \setlength\tabcolsep{0pt}%
    \put(0,0){\includegraphics[width=\unitlength,page=1]{Last-B.pdf}}%
    \put(0.71457653,0.27873908){\color[rgb]{0,0,0}\makebox(0,0)[lt]{\lineheight{0}\smash{\begin{tabular}[t]{l}$C$\end{tabular}}}}%
    \put(0.67820165,0.14240704){\color[rgb]{0,0,0}\makebox(0,0)[lt]{\lineheight{0}\smash{\begin{tabular}[t]{l}$l_{1,2}$\end{tabular}}}}%
    \put(0.91440586,0.17552869){\color[rgb]{0,0,0}\makebox(0,0)[lt]{\lineheight{0}\smash{\begin{tabular}[t]{l}$C''_n$\end{tabular}}}}%
    \put(0,0){\includegraphics[width=\unitlength,page=2]{Last-B.pdf}}%
    \put(0.26103522,0.22792105){\color[rgb]{0,0,0}\makebox(0,0)[lt]{\lineheight{0}\smash{\begin{tabular}[t]{l}$l_+$\end{tabular}}}}%
    \put(0.33462032,0.1640646){\color[rgb]{0,0,0}\makebox(0,0)[lt]{\lineheight{0}\smash{\begin{tabular}[t]{l}$l_2$\end{tabular}}}}%
    \put(0.22003207,0.16630573){\color[rgb]{0,0,0}\makebox(0,0)[lt]{\lineheight{0}\smash{\begin{tabular}[t]{l}$l_1$\end{tabular}}}}%
    \put(0.06432362,0.16204742){\color[rgb]{0,0,0}\makebox(0,0)[lt]{\lineheight{0}\smash{\begin{tabular}[t]{l}$C'$\end{tabular}}}}%
    \put(0.40013853,0.28620853){\color[rgb]{0,0,0}\makebox(0,0)[lt]{\lineheight{0}\smash{\begin{tabular}[t]{l}$l_n^+$\end{tabular}}}}%
    \put(0.40150918,0.08114106){\color[rgb]{0,0,0}\makebox(0,0)[lt]{\lineheight{0}\smash{\begin{tabular}[t]{l}$l_n^-$\end{tabular}}}}%
    \put(-0.00092041,0.24519506){\color[rgb]{0,0,0}\makebox(0,0)[lt]{\lineheight{0}\smash{\begin{tabular}[t]{l}$D_{l_+}^+$\end{tabular}}}}%
    \put(-0.00174279,0.04012759){\color[rgb]{0,0,0}\makebox(0,0)[lt]{\lineheight{0}\smash{\begin{tabular}[t]{l}$D_{l_-}^-$\end{tabular}}}}%
    \put(0.22359586,0.10825929){\color[rgb]{0,0,0}\makebox(0,0)[lt]{\lineheight{0}\smash{\begin{tabular}[t]{l}$C''$\end{tabular}}}}%
    \put(0.31343105,0.03394768){\color[rgb]{0,0,0}\makebox(0,0)[lt]{\lineheight{0}\smash{\begin{tabular}[t]{l}$l_-$\end{tabular}}}}%
    \put(0.40205745,0.20911211){\color[rgb]{0,0,0}\makebox(0,0)[lt]{\lineheight{0}\smash{\begin{tabular}[t]{l}$C_n$\end{tabular}}}}%
    \put(0,0){\includegraphics[width=\unitlength,page=3]{Last-B.pdf}}%
  \end{picture}%
\endgroup%

%% file: figures/last-4.pdf_tex
\begingroup%
  \makeatletter%
  \providecommand\color[2][]{%
    \errmessage{(Inkscape) Color is used for the text in Inkscape, but the package 'color.sty' is not loaded}%
    \renewcommand\color[2][]{}%
  }%
  \providecommand\transparent[1]{%
    \errmessage{(Inkscape) Transparency is used (non-zero) for the text in Inkscape, but the package 'transparent.sty' is not loaded}%
    \renewcommand\transparent[1]{}%
  }%
  \providecommand\rotatebox[2]{#2}%
  \newcommand*\fsize{\dimexpr\f@size pt\relax}%
  \newcommand*\lineheight[1]{\fontsize{\fsize}{#1\fsize}\selectfont}%
  \ifx\svgwidth\undefined%
    \setlength{\unitlength}{670.89413095bp}%
    \ifx\svgscale\undefined%
      \relax%
    \else%
      \setlength{\unitlength}{\unitlength * \real{\svgscale}}%
    \fi%
  \else%
    \setlength{\unitlength}{\svgwidth}%
  \fi%
  \global\let\svgwidth\undefined%
  \global\let\svgscale\undefined%
  \makeatother%
  \begin{picture}(1,0.47419172)%
    \lineheight{1}%
    \setlength\tabcolsep{0pt}%
    \put(0,0){\includegraphics[width=\unitlength,page=1]{last-4.pdf}}%
    \put(0.09495611,0.31649522){\color[rgb]{0,0,0}\makebox(0,0)[lt]{\lineheight{0}\smash{\begin{tabular}[t]{l}$l_{i_1}^+$\end{tabular}}}}%
    \put(0.13113279,0.43775746){\color[rgb]{0,0,0}\makebox(0,0)[lt]{\lineheight{0}\smash{\begin{tabular}[t]{l}$S_j^+$\end{tabular}}}}%
    \put(0,0){\includegraphics[width=\unitlength,page=2]{last-4.pdf}}%
    \put(0.23124459,0.41065044){\color[rgb]{0,0,0}\makebox(0,0)[lt]{\lineheight{0}\smash{\begin{tabular}[t]{l}$\dots$\end{tabular}}}}%
    \put(0.14691093,0.37436739){\color[rgb]{0,0,0}\makebox(0,0)[lt]{\lineheight{0}\smash{\begin{tabular}[t]{l}$l_+$\end{tabular}}}}%
    \put(0.20378712,0.30642972){\color[rgb]{0,0,0}\makebox(0,0)[lt]{\lineheight{0}\smash{\begin{tabular}[t]{l}$l_{i_2}^+$\end{tabular}}}}%
    \put(0,0){\includegraphics[width=\unitlength,page=3]{last-4.pdf}}%
    \put(0.56493942,0.31644289){\color[rgb]{0,0,0}\makebox(0,0)[lt]{\lineheight{0}\smash{\begin{tabular}[t]{l}$l_{i_1}^+$\end{tabular}}}}%
    \put(0.61859778,0.43880926){\color[rgb]{0,0,0}\makebox(0,0)[lt]{\lineheight{0}\smash{\begin{tabular}[t]{l}$\bar{S}_j^{+}$\end{tabular}}}}%
    \put(0,0){\includegraphics[width=\unitlength,page=4]{last-4.pdf}}%
    \put(0.71870958,0.41170223){\color[rgb]{0,0,0}\makebox(0,0)[lt]{\lineheight{0}\smash{\begin{tabular}[t]{l}$\dots$\end{tabular}}}}%
    \put(0.71336224,0.31145186){\color[rgb]{0,0,0}\makebox(0,0)[lt]{\lineheight{0}\smash{\begin{tabular}[t]{l}$l_{i_2}^+$\end{tabular}}}}%
    \put(0,0){\includegraphics[width=\unitlength,page=5]{last-4.pdf}}%
    \put(0.63388868,0.38444817){\color[rgb]{0,0,0}\makebox(0,0)[lt]{\lineheight{0}\smash{\begin{tabular}[t]{l}$l_{1,2}^+$\end{tabular}}}}%
    \put(0.86632229,0.44667344){\color[rgb]{0,0,0}\makebox(0,0)[lt]{\lineheight{0}\smash{\begin{tabular}[t]{l}$S_{k+1}^+$\end{tabular}}}}%
    \put(0,0){\includegraphics[width=\unitlength,page=6]{last-4.pdf}}%
    \put(-0.00135372,0.45216927){\color[rgb]{0,0,0}\makebox(0,0)[lt]{\lineheight{0}\smash{\begin{tabular}[t]{l}$\partial^+$\end{tabular}}}}%
    \put(0.49296041,0.45782876){\color[rgb]{0,0,0}\makebox(0,0)[lt]{\lineheight{0}\smash{\begin{tabular}[t]{l}$\partial^+$\end{tabular}}}}%
    \put(-0.29654114,0.29969387){\color[rgb]{0,0,0}\makebox(0,0)[lt]{\begin{minipage}{0.55659019\unitlength}\end{minipage}}}%
    \put(0,0){\includegraphics[width=\unitlength,page=7]{last-4.pdf}}%
    \put(0.00045349,0.22702981){\color[rgb]{0,0,0}\makebox(0,0)[lt]{\lineheight{0}\smash{\begin{tabular}[t]{l}$l_{i_1}^+$\end{tabular}}}}%
    \put(0.02467179,0.09429496){\color[rgb]{0,0,0}\makebox(0,0)[lt]{\lineheight{0}\smash{\begin{tabular}[t]{l}$C_{i_1}$\end{tabular}}}}%
    \put(0,0){\includegraphics[width=\unitlength,page=8]{last-4.pdf}}%
    \put(0.36932885,0.22551988){\color[rgb]{0,0,0}\makebox(0,0)[lt]{\lineheight{0}\smash{\begin{tabular}[t]{l}$l_{i_2}^+$\end{tabular}}}}%
    \put(0.2244446,0.12711285){\color[rgb]{0,0,0}\makebox(0,0)[lt]{\lineheight{0}\smash{\begin{tabular}[t]{l}$C'$\end{tabular}}}}%
    \put(0,0){\includegraphics[width=\unitlength,page=9]{last-4.pdf}}%
    \put(0.36262364,0.10054431){\color[rgb]{0,0,0}\makebox(0,0)[lt]{\lineheight{0}\smash{\begin{tabular}[t]{l}$C_{i_2}$\end{tabular}}}}%
    \put(0.18999571,0.21680547){\color[rgb]{0,0,0}\makebox(0,0)[lt]{\lineheight{0}\smash{\begin{tabular}[t]{l}$l_+$\end{tabular}}}}%
    \put(0,0){\includegraphics[width=\unitlength,page=10]{last-4.pdf}}%
    \put(0.68816651,0.24800527){\color[rgb]{0,0,0}\makebox(0,0)[lt]{\lineheight{0}\smash{\begin{tabular}[t]{l}$l_{1,2}^+$\end{tabular}}}}%
    \put(0.88812619,0.12305731){\color[rgb]{0,0,0}\makebox(0,0)[lt]{\lineheight{0}\smash{\begin{tabular}[t]{l}$C_{1,2}$\end{tabular}}}}%
    \put(0,0){\includegraphics[width=\unitlength,page=11]{last-4.pdf}}%
  \end{picture}%
\endgroup%

%% file: figures/last-6.pdf_tex
\begingroup%
  \makeatletter%
  \providecommand\color[2][]{%
    \errmessage{(Inkscape) Color is used for the text in Inkscape, but the package 'color.sty' is not loaded}%
    \renewcommand\color[2][]{}%
  }%
  \providecommand\transparent[1]{%
    \errmessage{(Inkscape) Transparency is used (non-zero) for the text in Inkscape, but the package 'transparent.sty' is not loaded}%
    \renewcommand\transparent[1]{}%
  }%
  \providecommand\rotatebox[2]{#2}%
  \newcommand*\fsize{\dimexpr\f@size pt\relax}%
  \newcommand*\lineheight[1]{\fontsize{\fsize}{#1\fsize}\selectfont}%
  \ifx\svgwidth\undefined%
    \setlength{\unitlength}{749.2486016bp}%
    \ifx\svgscale\undefined%
      \relax%
    \else%
      \setlength{\unitlength}{\unitlength * \real{\svgscale}}%
    \fi%
  \else%
    \setlength{\unitlength}{\svgwidth}%
  \fi%
  \global\let\svgwidth\undefined%
  \global\let\svgscale\undefined%
  \makeatother%
  \begin{picture}(1,0.18250743)%
    \lineheight{1}%
    \setlength\tabcolsep{0pt}%
    \put(0,0){\includegraphics[width=\unitlength,page=1]{last-6.pdf}}%
    \put(-0.00117454,0.16832418){\color[rgb]{0,0,0}\makebox(0,0)[lt]{\lineheight{0}\smash{\begin{tabular}[t]{l}$\tilde{S}_{k+1}^{+}$\end{tabular}}}}%
    \put(0,0){\includegraphics[width=\unitlength,page=2]{last-6.pdf}}%
    \put(0.25700355,0.16774488){\color[rgb]{0,0,0}\makebox(0,0)[lt]{\lineheight{0}\smash{\begin{tabular}[t]{l}$\tilde{S}_{k+1}^{+}\setminus C$\end{tabular}}}}%
    \put(0.53910195,0.15843113){\color[rgb]{0,0,0}\makebox(0,0)[lt]{\lineheight{0}\smash{\begin{tabular}[t]{l}$S_{k+1}^{+}$\end{tabular}}}}%
    \put(0,0){\includegraphics[width=\unitlength,page=3]{last-6.pdf}}%
    \put(0.78997467,0.15882623){\color[rgb]{0,0,0}\makebox(0,0)[lt]{\lineheight{0}\smash{\begin{tabular}[t]{l}$S_{k+1}^+\setminus C$\end{tabular}}}}%
    \put(0,0){\includegraphics[width=\unitlength,page=4]{last-6.pdf}}%
  \end{picture}%
\endgroup%

%% file: Paper.bbl
\begin{thebibliography}{FGMW10}

\bibitem[AC65]{A-C}
J.~J. Andrews and M.~L. Curtis.
\newblock Free groups and handlebodies.
\newblock {\em Proc. Amer. Math. Soc.}, 16:192--195, 1965.

\bibitem[AK85]{Akb}
Selman Akbulut and Robion Kirby.
\newblock A potential smooth counterexample in dimension $4$ to the
  poincar\'{e} conjecture, the schoenflies conjecture, and the andrews-curtis
  conjecture.
\newblock {\em Topology}, 24(4):375--390, 1985.

\bibitem[Akb10]{Akb2}
Selman Akbulut.
\newblock Cappell-shaneson homotopy spheres are standard.
\newblock {\em Annals of Mathematics}, 171(3):2171--2175, 2010.

\bibitem[BM93]{Bu-Mac}
R.~G. Burns and Olga Macedo\'{n}ska.
\newblock Balanced presentations of the trivial group.
\newblock {\em Bull. London Math. Soc.}, 25(6):513--526, 1993.

\bibitem[Bro84]{Brown}
Ronald Brown.
\newblock Coproducts of crossed p-modules: Applications to second homotopy
  groups and to the homology of groups.
\newblock {\em Topology}, 23(3):337--345, 1984.

\bibitem[Dun76]{Dun}
M.~J. Dunwoody.
\newblock The homotopy type of a two-dimensional complex.
\newblock {\em Bulletin of the London Mathematical Society}, 8(3):282--285, 11
  1976.

\bibitem[FGMW10]{Freedman}
Michael Freedman, Robert Gompf, Scott Morrison, and Kevin Walker.
\newblock Man and machine thinking about the smooth 4-dimensional
  {P}oincar\'{e} conjecture.
\newblock {\em Quantum Topol.}, 1(2):171--208, 2010.

\bibitem[Fre82]{freedman1982}
Michael~Hartley Freedman.
\newblock The topology of four-dimensional manifolds.
\newblock {\em J. Differential Geom.}, 17(3):357--453, 1982.

\bibitem[Gom91]{Gom}
Robert~E. Gompf.
\newblock Killing the akbulut-kirby 4-sphere, with relevance to the
  andrews-curtis and schoenflies problems.
\newblock {\em Topology}, 30(1):97 -- 115, 1991.

\bibitem[HAM93]{Hog-Met}
Cynthia Hog-Angeloni and Wolfgang Metzler.
\newblock {\em The Andrews-Curtis conjecture and its generalizations}, volume
  197 of {\em London Math. Soc. Lecture Note Ser.}
\newblock Cambridge Univ. Press, Cambridge, 1993.

\bibitem[Joh80]{John}
D.~L. Johnson.
\newblock {\em Topics in the theory of group presentations}, volume~42 of {\em
  London Mathematical Society Lecture Note Series}.
\newblock Cambridge University Press, Cambridge-New York, 1980.

\bibitem[Met76]{Met}
W.~Metzler.
\newblock \"uber den homotopietyp zweidimensionaler cw-komplexe und
  elementartransformationen bei darstellungen von gruppen durch erzeugende und
  definierende relationen.
\newblock {\em J. Reine Angew. Math.}, 285:7--23, 1976.

\bibitem[MMS02]{Shpil}
Alexei~D. Myasnikov, Alexei~G. Myasnikov, and Vladimir Shpilrain.
\newblock On the {A}ndrews-{C}urtis equivalence.
\newblock In {\em Combinatorial and geometric group theory (New York,
  2000/Hoboken, NJ, 2001)}, volume 296 of {\em Contemp. Math.}, pages 183--198.
  Amer. Math. Soc., Providence, RI, 2002.

\bibitem[MS99]{Miller}
Charles~F. Miller, III and Paul~E. Schupp.
\newblock Some presentations of the trivial group.
\newblock In {\em Groups, languages and geometry (South Hadley, MA, 1998)},
  volume 250 of {\em Contemp. Math.}, pages 113--115. Amer. Math. Soc.,
  Providence, RI, 1999.

\bibitem[Wri75]{WR}
Perrin Wright.
\newblock Group presentations and formal deformations.
\newblock {\em Trans. Amer. Math. Soc.}, 208:161--169, 1975.

\end{thebibliography}
